\documentclass[10pt, reqno]{amsart}
\usepackage{amsrefs, amsmath, amssymb, amsthm, graphicx}

\def\eq#1{(\ref{#1})}
\def\nn{\nonumber}
\def\({\left(\begin{array}{cccccc}}
\def\){\end{array}\right)}

\def\eq#1{(\ref{#1})}
\def\nn{\nonumber}

\def\({\left(\begin{array}{cccccc}}
\def\){\end{array}\right)}

\def\bes{\begin{eqnarray}}
\def\ees{\end{eqnarray}}

\newcommand{\jdown}{\overset{\,\,>}{J}}
\newcommand{\jup}{\overset{\,\,<}{J}}

\newcommand{\ba}{\overset{\raisebox{0pt}[0pt][0pt]{\text{\raisebox{-.5ex}{\scriptsize$\leftharpoonup$}}}}}
\newcommand{\fa}{\overset{\raisebox{0pt}[0pt][0pt]{\text{\raisebox{-.5ex}{\scriptsize$\rightharpoonup$}}}}}
\newcommand{\rhs}{\mathrm{RHS}\, \eq}
\newcommand{\lhs}{\mathrm{LHS}\, \eq}

\newcommand{\del}{\partial}

\newcommand{\beq}{\begin{equation}}
\newcommand{\eeq}{\end{equation}}
\newcommand{\bea}{\begin{eqnarray}}
\newcommand{\eea}{\end{eqnarray}}
\newcommand{\beann}{\begin{eqnarray*}}
\newcommand{\eeann}{\end{eqnarray*}}

\newcommand{\RR}{\mathbb{R}}

\newcommand{\bp}{\begin{proof}}
\newcommand{\ep}{\end{proof}}
\newcommand{\nquad}{\negthickspace\negthickspace
\negthickspace\negthickspace}
\newcommand{\nqquad}{\nquad\nquad}

\newcommand{\cB}{{\mathcal B}}

\newcommand{\cG}{{\mathcal G}}
\newcommand{\cH}{{\mathcal H}}

\newcommand{\cK}{{\mathcal K}}

\newcommand{\ty}{\tilde{y}}

\newcommand{\tx}{\tilde{x}}

\newtheorem{theorem}{Theorem}[section]
\newtheorem{proposition}[theorem]{Proposition}

\newtheorem{lemma}[theorem]{Lemma}

\newtheorem{remark}[theorem]{Remark}

\numberwithin{equation}{section}

\begin{document}

\title{Pairwise wave interactions in ideal polytropic gases}

\author{Geng Chen}
\address{G.~Chen, Department of Mathematics, Penn State University, University Park, 
State College, PA 16802, USA ({\tt chen@math.psu.edu}).}
\author{Erik E. Endres}
\address{E.~E.~Endres, ({\tt eeendres${}_-$usa@yahoo.com})}
\author{Helge Kristian Jenssen}
\address{ H.~K.~Jenssen, Department of Mathematics, Penn State University, University Park, 
State College, PA 16802, USA ({\tt jenssen@math.psu.edu}).}
\thanks{H.~K.~Jenssen was partially supported by NSF grants DMS-0539549 (CAREER) and DMS-1009002.}

\date{\today}
\begin{abstract}
	We consider the problem of resolving all pairwise interactions of shock 
	waves, contact waves, and rarefaction waves in 1-dimensional 
	flow of an ideal polytropic gas.
	Resolving an interaction means here to determine the types of the three 
	outgoing (backward, contact, and forward) waves in the 
	Riemann problem defined by the extreme left and right states of the two incoming 
	waves, together with possible vacuum formation. This problem has been 
	considered by several authors and turns out to be
	surprisingly involved. For each type of interaction (head-on, involving a contact, 
	or overtaking) the outcome depends on the strengths of the incoming waves. 
	In the case of overtaking waves the type of the reflected wave also depends on the 
	value of the adiabatic constant. Our analysis provides a complete breakdown and 
	gives the exact outcome of each interaction.
\end{abstract}

\maketitle


\tableofcontents

\section{Introduction}
We consider the problem of resolving all interactions of pairs of elementary waves in 
one-dimensional compressible, adiabatic flow of an ideal and polytropic gas as described by the 
Euler system \eq{continuity}, \eq{momentum}, \eq{tot_energy}. The elementary 
waves are shocks, contacts, and centered rarefactions. 
The interactions are treated as in the Glimm scheme \cite{gl}: if the incoming waves connect 
a left state $U_l$, via a middle state $U_m$, to right state $U_r$, then the ``interaction problem" 
is to determine the type and strength of the resulting outgoing waves obtained by resolving the 
Riemann problem $(U_l,U_r)$. Here $U$ denotes any triple needed to specify the state the gas; we 
will mostly work with specific volume $\tau=1/\rho$ ($\rho=$ density), fluid velocity $u$, and pressure $p$.  

For interactions that involve only shocks and/or contacts the solution of the Riemann problem 
$(U_l,U_r)$  provides the exact solution of the wave interaction. When one of the waves is a rarefaction 
the actual interaction involves ``penetration" of a rarefaction wave by a shock, a contact or by another 
rarefaction. In these cases the exact solution is more involved.  We do not treat the problem of 
penetration and consider only the Riemann problem defined by the extreme 
states $U_l$ and $U_r$.

\begin{remark}
In cases of wave penetration the extreme Riemann problem $(U_l,U_r)$ may not provide an accurate 
description of the asymptotic behavior of the wave interaction. E.g., in an overtaking interaction 
of a shock and a rarefaction, the shock may not pass through the rarefaction (incomplete penetration), 
and in this case the solution of the Riemann problem $(U_l,U_r)$ presumably does not describe the 
asymptotic behavior in the exact solution. See Part D of Chapter III in \cite{coufr}, Section 3.5.2 in 
\cite{ch}, and \cite{gr} for further details and references.
\end{remark}

Many authors have studied the problem of pairwise interactions. Since it gives rise to 
Riemann problems it may be said to originate with Riemann's and Hugoniot's pioneering works 
\cites{ri,jc,had} on gas dynamics. According to \cite{coufr}, Jouguet \cite{jou} in his work on combustion 
considered interaction phenomena. A systematic approach was initiated by von Neumann 
\cite{vneu} in his work on hydrodynamics. Among other issues he considered the interaction of two 
shock waves, and called attention to the fact that a contact wave emerges in such interactions. 
He also observed that in interactions of overtaking shocks in an ideal gas, the reflected wave 
is necessarily a rarefaction whenever the adiabatic constant $\gamma$ satisfies 
$\gamma\leq \frac{5}{3}$. Shock-shock and shock-contact interactions 
were analyzed in detail by Ro{\v{z}}destvenski{\u\i} \& Janenko \cites{rj}; see also Smoller \cite{smol}. 
A comprehensive treatment is presented by Courant and Friedrichs \cite{coufr} 
who also list results for shock-contact and rarefaction-contact interactions. A partial treatment is 
also given in \cites{rj,ll}. 
The most complete work on pairwise interactions in ideal polytropic gases to date appears to be the 
analysis by Chang and Hsiao \cite{ch}. These authors considered all the ten essentially different 
types of pairwise interactions (see \eq{gr1}-\eq{gr3}), for all values of the adiabatic constant $\gamma>1$. 
In particular they list the possible outcomes of all ten interactions. However their analysis does 
not completely delimit all the various outcomes. 

In revisiting the interaction problem our objectives are to provide a complete breakdown 
of all possible cases, and to determine the {\em exact} outcome of each interaction. 
In addition we specify precisely when a vacuum state appears among the outgoing waves. 
Our analysis is motivated by a desire to investigate relevant measures for wave-strengths 
in large data solutions of the Euler system. The latter issue will be pursued elsewhere.

Let us explain what we mean by ``complete breakdown of all cases." This requires a little
background. We work in a Lagrangian frame such that contact waves appear stationary, 
and the primary unknowns are specific volume $\tau$, particle velocity $u$,
and pressure $p$ (alternatively, specific entropy $S$).
We parametrize shocks and rarefactions by pressure ratios across 
the wave, while contacts are parametrized by specific volume ratios. With a slight abuse of terminology 
``a shock $f$," say, refers to a shock wave connecting two states $U_-$ (to the left of the shock) and 
$U_+$ (to the right of the shock), and such that the pressures satisfy $\frac{p(U_+)}{p(U_-)}=f$.  
We also refer to $f$ as the {\em strength} of the wave.
A wave changes type (shock $\leftrightarrow$ rarefaction, or up-contact $\leftrightarrow$ down-contact) 
as this pressure (or specific volume) ratio crosses the value $1$. 

Now, an interaction of two incoming waves $x$ and $y$ generically gives 
rise to three outgoing waves: a backward wave $B$, a contact $C$, and a forward wave $F$, and, possibly,
a vacuum state. The interaction problem amounts to determining whether a vacuum occurs, together
with the values of $B$, $C$, $F$, in terms of $x$, $y$. 

It does not seem possible to determine 
the functions $B(x,y)$, $C(x,y)$, and $F(x,y)$ explicitly in all cases. 
A somewhat weaker request is to ask for the ``transitional curves" in the plane of incoming 
strengths. That is, to determine the curves in the $(x,y)$-plane which delimit the regions 
$\{(x,y)\,|\,B(x,y)\gtrless 1\}$, $\{(x,y)\,|\,C(x,y)\gtrless 1\}$, $\{(x,y)\,|\,F(x,y)\gtrless 1\}$, 
as well as the vacuums regions. (It turns out that these transitional sets are either 
$C^2$-curves, or finite unions of such.) It is possible to find explicit equations for these curves and to determine 
their relative locations and intersections, and thereby determine all possible outcomes in all pairwise interactions. 

In this work we carry out the detailed computations that yield a complete breakdown in this sense. In doing 
so we improve on the existing results in the literature. E.g., one of the more involved cases occurs when a 
backward shock is overtaken by a backward rarefaction ($\ba S\ba R$) in a gas with adiabatic exponent $\gamma<\frac{5}{3}$. 
Depending on the strengths $x$, $y$ of the incoming waves, the interaction 
may produce any one of four outgoing wave configurations:
\[\ba S\jup\fa S,\quad \ba R\jup\fa S,\quad \ba S\jup\fa R,\quad\text{or}\quad \ba R\jup\fa R,\]
with possible vacuum formation in the latter case, see Figure 5 (left diagram). 
Here $R$, $J$, and $S$ denote rarefactions, contacts, and shocks, respectively; an arrow indicates the 
direction the wave is moving relative to the Lagrangian frame, and $\gtrless$ indicate increase/decrease 
of density across a contact. In our analysis we delimit the various cases in terms of 
certain curves in the $(x,y)$-plane. These curves are given in terms of explicit equations - in some cases
as explicit graphs $y(x)$ or $x(y)$.

The calculations are surprisingly involved, with a number of 
sub-cases for various combinations of strengths $x$, $y$ and values of $\gamma$. It is possible that 
there are more efficient ways of parametrizing waves than the present choice of pressure and specific 
volume ratios. In fact, for some cases we find it useful to consider the entropy variable as well. 
However, we have not been able to determine more efficient parameters that would, e.g., provide 
explicit formulas for outgoing strengths in terms of incoming strengths.

The rest of the article is organized as follows.  In Section \ref{bckgr} we recall the Euler system and 
two parametrizations of its wave curves, one in $(\tau,u,p)$-space and the other in $(\tau,u,S)$-space. 
In Section \ref{rip_vac} we briefly review the solution of
the Riemann problem and characterize the occurrence of a vacuum. We also list and group the 
ten essentially different pairwise interactions, modulo reflections of the spatial variable $x\leftrightarrow -x$. 
We analyze head-on (Group I) interactions in Section \ref{Gr_I},
and interactions with a contact (Group II) in Section \ref{Gr_II}. The results for 
these groups are summed up in Theorems \ref{sum_grI} and \ref{sum_grII}.
All of the remaining analysis deals with the more involved cases where a shock 
or rarefaction overtakes another shock or rarefaction (Group III). 
The results for overtaking interactions are given in Theorem \ref{sum_grIII}. It turns out that the outcome depends 
not only on the incoming wave-strengths but also on the adiabatic constant $\gamma$. 
Considering (for concreteness) the case of overtaking backward waves we divide the analysis into three sub-groups: 
$\ba S\ba S$ (IIIa), $\ba S \ba R$ (IIIb), and $\ba R\ba S$ (IIIc) interactions. 
As observed by Chang and Hsiao \cite{ch} there are two values ($\gamma=\frac{5}{3}$ 
and $\gamma=2$) at which the global structure of the transitional curve for the reflected wave
(i.e., the curve $\{(x,y)\,|\,F(x,y)= 1\}$) changes character. The type of the transmitted 
(backward) wave for Group III interactions is analyzed in Section \ref{Gr_III_i}. 
In Section \ref{group3_reflected} we treat the reflected wave; this is where 
the most involved analysis occurs, due to the different behaviors for $\gamma\gtrless\frac{5}{3}$ 
and $\gamma\gtrless 2$. To analyze the outgoing contact it turns out to be convenient to use 
$(\tau,u,S)$ variables; the details are given in Section \ref{group3_contact}. Finally, in Section 
\ref{group3_vac} we analyze vacuum formation in overtaking interactions.  
Section \ref{aux} collects various auxiliary results needed in the analysis.

\section{The Euler system, wave curves, Riemann problems, and vacuum}\label{bckgr}
\subsection{The 1-d compressible Euler system for an ideal gas}
We consider 1-d compressible flow of an ideal, polytropic gas. The conservation laws 
for mass, momentum, and energy are given by
\begin{eqnarray}
	\tau_t - u_x &=&0 \label{continuity}\\
	u_t + p_x &=&0 \label{momentum}\\
	E_t + (up)_x &=& 0\, , \label{tot_energy}
\end{eqnarray}
where $x$ denotes\footnote{$x$ and $t$ will later be used for other parameters 
as well. The context will make it clear which use is intended.} a Lagrangian (mass) coordinate,
$\tau=1/\rho$ is specific volume, $u$ is the fluid velocity, $p$ is 
pressure, and $E=e+\frac{u^2}{2}$ is specific total energy. Here $e$ denotes specific internal 
energy, which for an ideal gas is given by
\beq\label{ip_gas}
	e(\tau,p)=\frac{p\tau}{\gamma-1}\,,\qquad \text{where $\gamma>1$ is the adiabatic constant.}
\eeq
For later reference we note that the local sound speed $\mathfrak c$ is given by 
\beq\label{soundspeed}
	\mathfrak c^2=\gamma \tau p\,.
\eeq

\subsection{Wave curves}\label{curves}
Each fixed state $(\bar \tau,\bar u,\bar p)\in\RR^+\times\RR\times\RR^+$ has three associated 
{\it wave curves} consisting of those states that can be connected to $(\bar \tau,\bar u,\bar p)$ 
(on its right) by a single backward wave, a single contact, or a single forward wave, respectively. 
The backward and forward waves are either rarefactions or entropy admissible (i.e.\ compressive) 
shocks. We use pressure ratios $p_{right}/p_{left}$ to parametrize the wave curves in 
the extremal fields, and specific volume ratios $\tau_{right}/\tau_{left}$ for the contact field. 
We let $b$, $c$, $f$ denote the pressure ratio (strength) of a backward, contact, and 
forward wave, respectively. Using the expressions from \cite{smol} (p.\ 354) the wave curves 
are given as follows:
\beq\label{bkwd_wave}
	{\ba{W}}(b;\bar \tau,\bar u,\bar p) = \left(\begin{array}{c}
		{\ba{\phi}}(b)\bar \tau \\
		\bar u -{\ba{\psi}}(b)\sqrt{\bar \tau\bar p}\\
		b\bar p
	\end{array}\right)
	\qquad\left\{\begin{array}{ll}
	\ba R:\quad 0<b<1 \\
	\ba S:\quad b>1
	\end{array}\right.
\eeq
\beq\label{contactwave}
	{\overset{\lessgtr}{J}}(c;\bar \tau,\bar u,\bar p) = \left(\begin{array}{c}
		c\bar \tau \\
		\bar u \\
		\bar p
	\end{array}\right)\qquad 1\lessgtr c \quad(c>0)
\eeq
\beq\label{frwd_wave}
	{\fa{W}}(f;\bar \tau,\bar u,\bar p) = \left(\begin{array}{c}
		{\fa{\phi}}(f)\bar \tau \\
		\bar u +{\fa{\psi}}(f)\sqrt{\bar \tau\bar p}\\
		f\bar p
	\end{array}\right)
	\qquad\left\{\begin{array}{ll}
	\fa R:\quad f>1 \\
	\fa S:\quad 0<f<1\,.
	\end{array}\right.
\eeq
The auxiliary functions $\ba\phi$, $\fa\phi$, $\ba\psi$, and $\fa\psi$ are given by
\[{\ba{\phi}}(b) = \left\{\begin{array}{ll}
	b^{-1/\gamma} \quad & 0<b<1\\\\
	\frac{1+a b}{b+a}\quad & b>1
\end{array}\right\}\qquad
{\ba{\psi}}(b)= \left\{\begin{array}{ll}
	\nu\big(b^\zeta-1\big)\quad & 0<b<1\\\\
	\frac{\kappa(b-1)}{\sqrt{b+a}}\quad & b>1
\end{array}\right\}\] 
\[{\fa{\phi}}(f) = \left\{\begin{array}{ll}
	f^{-1/\gamma} \quad & f>1\\\\
	\frac{1+a f}{f+a}\quad & 0<f<1
\end{array}\right\}\qquad
{\fa{\psi}}(f)= \left\{\begin{array}{ll}
	\nu\big(f^\zeta-1\big)\quad & f>1\\\\
	\frac{\kappa(f-1)}{\sqrt{f+a}}\quad & 0<f<1
\end{array}\right\}\,,\] 
where the constants $\kappa$, $\nu$, $\zeta$ are given in terms of the parameter\footnote{For 
brevity we use $a$ instead of the more common notation $\mu^2$.}
\beq\label{a}
	a:= \frac{\gamma -1}{\gamma +1}\in\, (0,1)
\eeq
as follows:
\[\kappa = \sqrt{1-a}\,\in\, (0,1)\,, 
\qquad \nu=\frac{\sqrt{1-a^2}}{a}\,\in\, (0,+\infty)\,,
\qquad \zeta=\frac{a}{1+a}\in(0,\textstyle\frac{1}{2})\,.\]
We note that $a$ is a strictly increasing function of $\gamma$, and that 
\[a= \textstyle\frac{1}{4}\,\,\Leftrightarrow\,\,  \gamma= \frac{5}{3},\qquad \text{and}
\qquad a= \frac{1}{3}\,\,\Leftrightarrow\,\, \gamma= 2\,.\]
The auxiliary functions $\fa \phi$, $\ba \phi$, $\fa\psi$, and $\ba\psi$ are analyzed in 
Section \ref{aux}.

It will be convenient to perform some of the calculations in $(\tau,u,S)$-variables
(see Section \ref{group3_contact}), where $S=$ specific entropy, and we proceed to record 
the expressions for the wave curves in these variables. The pressure is now given by 
\beq\label{p_tau_u_S}
	p=\hat p(\tau,S)=K\tau^{-\gamma}\exp\left(\textstyle\frac{S}{c_v}\right)\quad \qquad\text{($K$, $c_v$ positive constants).}
\eeq
We parametrize all wave curves in $(\tau,u,S)$-space by specific volume ratios 
$\tau_{right}/\tau_{left}$. We use $l$ for backward waves, $c$ for contacts (as above), and $\iota$ 
for forward waves. From the parametrizations in $(\tau,u,p)$-space, together with \eq{p_tau_u_S}, 
we obtain the following expressions for the wave curves $(\tau,u,S)$-variables:
\beq\label{bkwd_wave_tau_u_s}
    {\ba{W}}(l;\bar \tau,\bar u,\bar S)
        = \left(\begin{array}{c}
        l\bar \tau \\
        \bar u -{\ba{\xi}}(l)\sqrt{\bar \tau\bar p}\\
        \bar S+c_v{\ba{\eta}}(l)
    \end{array}\right)
    \qquad\left\{\begin{array}{ll}
    \ba R:\quad l>1  \\
    \ba S:\quad a<l<1
    \end{array}\right.
\eeq 
\beq\label{contactwave_tau_u_s}
    {\overset{\lessgtr}{J}}(c;\bar \tau,\bar u,\bar S) 
    = \left(\begin{array}{c}
        c\bar \tau \\
        \bar u \\
        \bar S+c_v\gamma\log c
    \end{array}\right)\qquad 1\lessgtr c \quad (c>0)
\eeq 
\beq\label{frwd_wave_tau_u_s}
    {\fa{W}}(\iota;\bar \tau,\bar u,\bar S) 
    =\left(\begin{array}{c}
        \iota\bar \tau \\
        \bar u +{\fa{\xi}}(\iota)\sqrt{\bar \tau\bar p}\\
        \bar S+c_v{\fa{\eta}}(\iota)
    \end{array}\right)
    \qquad\left\{\begin{array}{ll}
    \fa R:\quad 0<\iota<1 \\
    \fa S:\quad 1<\iota<\frac{1}{a}\,.
    \end{array}\right.
\eeq 
The auxiliary function $\ba\xi$, $\fa\xi$, $\ba\eta$, $\fa\eta$ are given by
\[{\ba{\xi}}(l) =
    \left\{\begin{array}{ll}
    \nu \big(l^\frac{1-\gamma}{2}-1\big) \quad & l>1\\\\
    \sqrt{a +1}\frac{(1-l)}{\sqrt{l-a}}\quad & a <l<1
\end{array}\right\}\qquad
{\ba{\eta}(l)}
    = \left\{\begin{array}{ll}
    0 \quad & l>1\\\\
    \log\Big(\textstyle\frac{l^\gamma (1-al)}{l-a}\Big)\quad & a <l<1
\end{array}\right\}\]
\[{\fa{\xi}}(\iota) =
    \left\{\begin{array}{ll}
    \nu \big(\iota^\frac{1-\gamma}{2} -1\big)\quad & 0<\iota<1\\\\
    \sqrt{a +1}\frac{(1- \iota)}{\sqrt{\iota-a}}\quad & 1<\iota<\frac{1}{a}
\end{array}\right\}\qquad
{\fa{\eta}}(\iota)=
    \left\{\begin{array}{ll}
    0\quad & 0< \iota<1\\\\
    \log\Big(\textstyle\frac{\iota^\gamma (1-a \iota)}{\iota-a}\Big)\quad & 1<\iota<\frac{1}{a}
\end{array}\right\}.\]

\begin{remark}\label{entr_jmp}
	As is clear from the parametrizations \eq{bkwd_wave_tau_u_s}-\eq{frwd_wave_tau_u_s} 
	the entropy remains constant across rarefactions, while it necessarily jumps across
	contacts.
\end{remark}

\medskip

\subsection{Riemann problems and vacuum formation}\label{rip_vac}
A {\it Riemann problem} refers to the problem of resolving a single, initial jump 
discontinuity connecting two constant states. It is known that the Euler system 
for an ideal, polytropic gas admits a self-similar solution to any Riemann 
problem. Futhermore, this solution is unique within the class of self-similar solutions 
consisting of constant states connected by compressible shocks, contacts, and 
centered rarefactions, possibly including a vacuum state; see \cite{coufr}.

Consider the Riemann problem for \eq{continuity}-\eq{tot_energy} 
with left state $(\bar \tau,\bar u,\bar p)$ and right 
state $(\tau,u,p)$. We use capital letters $B$, $C$, and $F$ to denote the resulting 
outgoing backward, contact and forward waves, respectively. 
Traversing the resulting wave-fan from left to right and using the
expressions \eq{bkwd_wave}-\eq{frwd_wave}, yield
\beq\label{BCF}
	\tau = C {\ba{\phi}}(B) {\fa{\phi}}(F) \bar \tau\,,\qquad 
	\frac{u-\bar u}{\sqrt{\bar \tau \bar p}} = {\fa{\psi}}(F)
	\sqrt{C\, B\, {\ba{\phi}}(B)} - {\ba{\psi}}(B)\,,\qquad 
	p = BF\bar p\,.
\eeq
By using the first and last equations in the middle equation, together with the relation \eq{rel3},
we obtain the following equation for $B$:
\beq\label{Riem}
	\frac{\bar u-u}{\sqrt{\bar \tau \bar p}}={\ba{\psi}}(B) + \sqrt{\frac{\tau p}{\bar \tau \bar p}}\,\,
	{\ba{\psi}}\left(\frac{\bar p B}{p}\right)=:\mathcal F(B;\tau,p,\bar\tau,\bar p)=\mathcal F(B) \,.
\eeq
To solve the Riemann problem $\big((\bar \tau,\bar u,\bar p),(\tau,u,p)\big)$
amounts to determining the pressure ratio $B$ of the outgoing backward wave
from $\mathcal F(B)=\frac{\bar u-u}{\sqrt{\bar \tau \bar p}}$. 
>From this one then determines $C$ and $F$ from \eq{BCF}. 

The function ${\ba{\psi}}$ is strictly increasing and tends to $+\infty$ at $+\infty$ (see Section \ref{aux}).
It follows that the map $B\mapsto \mathcal F(B;\tau,p,\bar\tau,\bar p)$ 
has the same properties, and that the Riemann problem has a unique solution without vacuum
if and only if $(\bar \tau,\bar u,\bar p)$ and $(\tau,u,p)$ are such that 
\[\mathcal F(0;\tau,p,\bar\tau,\bar p)<\frac{\bar u-u}{\sqrt{\bar \tau \bar p}}\,.\]
This is the case if and only if
\beq\label{vac_cond}
	u-\bar u < \frac{2}{\gamma-1}(\mathfrak c+\bar{\mathfrak c})\qquad\quad\mbox{(no vaccum),}
\eeq
where the sound speeds $\mathfrak c$ and $\bar{\mathfrak c}$ are given by \eq{soundspeed}.
If the condition \eq{vac_cond} is not satisfied,
then we agree to solve the Riemann problem as follows: a backward rarefaction 
connects $(\bar\tau,\bar u,\bar p)$ to $(\infty,0,0)$, followed by a forward 
rarefaction connecting $(\infty,0,0)$ to $(\tau,u,p)$. In the $(x,t)$-plane the 
two rarefactions are separated by a vertical line along which $\tau=\infty$. 
The two rarefactions, together with the line $x=0$, span the fan (see \cite{smol})
\[-\sqrt{\frac{\gamma \bar p}{\bar \tau}}\,<\,\frac{x}{t}\,<\,\sqrt{\frac{\gamma p}{\tau}}\,.\] 

\begin{remark}\label{entr_jump}
	We note that the rarefactions on either side of a vacuum may be in different entropy states. 
	In particular this can occur in vacuums resulting from interactions: see region $\{c<1<f\}$ in 
	Figure 3, and region $\{y<1<x\}$ in Figures 4-6.	  
	Recalling Remark \ref{entr_jmp} we denote the outcome of such interactions by 
	$\ba R{\overset{\lessgtr}{J}}\fa R$ according to whether the left and right entropy states 
	satisfy $S_l\lessgtr S_r$. Thus, for these outcomes, ${\overset{\lessgtr}{J}}$ indicates the direction 
	of the {\em entropy} jump across the vacuum.
\end{remark}

\section{Pairwise interactions and summary of results}\label{intrcns_results}
\subsection{Pairwise interactions}
We shall consider all possible wave configurations where two elementary waves
(compressible shocks, contacts, and centered rarefactions) interact. We collect 
the interactions into three groups: head-on (I), involving a contact (II), and
overtaking (III), see Figure 1. In interactions, lower case letters 
are used for the incoming waves ($f$ for forward waves, $c$ for contacts, and $b$ for 
backward waves), while capital letters $B$, $C$, $F$, denote outgoing waves. 
In Figure 1 the waves are depicted schematically as if no penetration occurs: 
all waves, including possible incoming and outgoing rarefaction fans, are drawn as 
single lines in the $(x,t)$-plane. As described in the introduction the interaction problem 
is to resolve the particular Riemann problem defined by the extreme left and right
states in the incoming waves.

Due to the symmetry $x\leftrightarrow -x$ there are ten essentially different pairwise 
interactions:
\beq\label{gr1}
	\left.
	\begin{array}{ll}
		\mbox{Ia}: & {\fa{S}}{\ba{S}}\\
		\mbox{Ib}: & {\fa{S}}{\ba{R}}\\
		\mbox{Ic}: & {\fa{R}}{\ba{R}}
	\end{array}	
	\qquad\right\}\qquad \mbox{head-on interactions}
\eeq
\beq\label{gr2}
	\left.
	\qquad\,\,\begin{array}{ll}
		\mbox{IIa}: & {\fa{S}}\jdown\\
		\mbox{IIb}: & {\fa{S}}\jup\\
		\mbox{IIc}: & {\fa{R}}\jdown\\
		\mbox{IId}: & {\fa{R}}\jup
	\end{array}	
	\qquad\right\}\qquad \mbox{interactions with a contact}
\eeq

\beq\label{gr3}
	\left.
	\,\,\,\begin{array}{ll}
		\mbox{IIIa}: & {\ba{S}}{\ba{S}}\\
		\mbox{IIIb}: & {\ba{S}}{\ba{R}}\\
		\mbox{IIIc}: & {\ba{R}}{\ba{S}}
	\end{array}	
	\qquad\right\}\qquad \mbox{overtaking interactions.}
\eeq
Each interaction defines a particular Riemann problem, the {\em interaction Riemann problem}, 
for which we obtain a 
non-linear algebraic equation determining the strength of the outgoing backward 
wave, say, as detailed above for a general Riemann problem.
When expressed in terms of the auxiliary functions $\ba\phi$, $\fa\phi$, $\ba\psi$, 
$\fa\psi$, this nonlinear equation is the same within each of the groups I - III. However, 
a breakdown into individual cases is necessary to determine exactly the outcomes 
of each type of interaction.

\begin{remark}
	The parametrizations of the wave curves in Section \ref{curves} are in terms of 
	pressure and specific volume ratios. This has the following useful consequence: 
	the outgoing wave strengths in interactions are given in terms of three nonlinear 
	algebraic equations that involves {\em only} the outgoing wave strengths (and 
	$\gamma$). In particular, the equations are independent of, say, 
	the left-most state in the interaction. See equations \eq{Iint1}-\eq{Iint3}, 
	\eq{IIint1}-\eq{IIint3}, and \eq{IIIint1}-\eq{IIIint3}.
\end{remark}

\begin{figure}\label{interactns1}
	\centering
	\includegraphics[width=11cm,height=2.7cm]{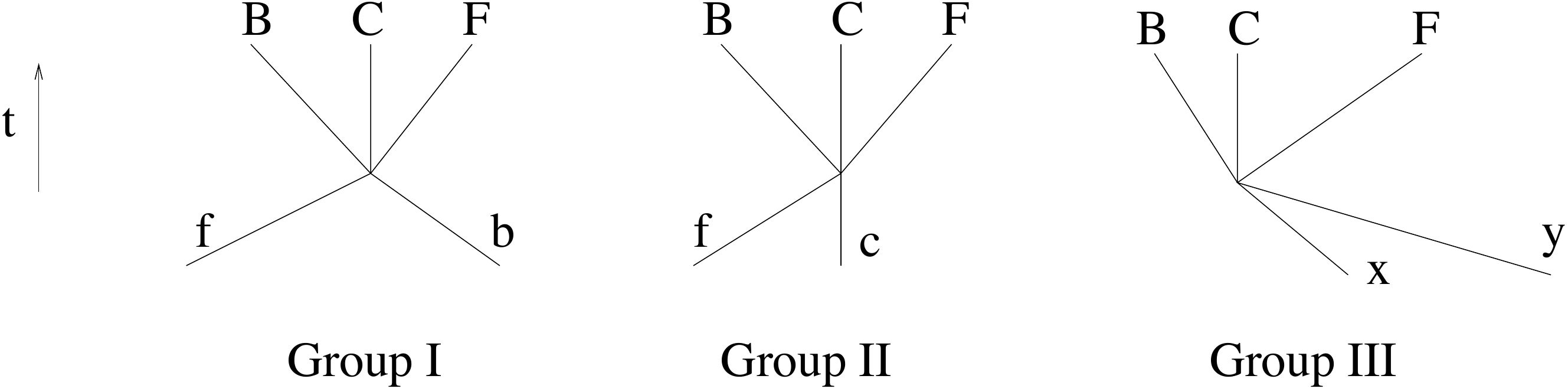}
	\caption{Interactions: head-on (I), involving a contact (II), overtaking (III). (Schematic)}
\end{figure} 
%
%

\subsection{Summary of results}
In this section we collect our results for interactions in each of the three
groups. For each group we consider every combination of incoming waves,
as listed in \eq{gr1}-\eq{gr3}, and we describe all possible combinations of outgoing 
waves in the corresponding interaction Riemann problem. As discussed in the 
introduction it does not seem possible to determine outgoing strengths as explicit 
functions of incoming strengths. Instead we describe the ``transitional curves" in the 
plane of incoming strengths. For example, for head-on interactions with incoming 
wave strengths $b$ and $f$, the outgoing backward wave $B=B(b,f)$ may be 
a rarefaction or a shock. According to \eq{bkwd_wave} these outcomes correspond 
to $B(b,f)<1$ and $B(b,f)>1$, respectively. The transitional curve for the backward wave 
in this case is the curve $\{(b,f)\,|\, B(b,f)=1\}$. Similarly, there are transitional curves for the outgoing 
contact $C$ and forward wave $F$. It will turn out that these sets really are curves in the 
plane of incoming strengths. In addition we determine the location of all ``vacuum-curves", 
i.e., the curves in the plane of incoming strengths which delimit interactions for which 
the interaction Riemann problem contains a vacuum state.

To describe the results we make reference to Figures 2 - 6. 
First, for Group I and Group II we give two diagrams in the plane of incoming 
strengths (i.e., the $(b,f)$-plane for Group I and the $(f,c)$-plane for Group II).
The left diagram depicts the various combinations of incoming waves: see left diagrams 
in Figure 2 and Figure 3. The diagram on the right in each of these figures shows the resulting,
outgoing wave configurations, together with the transitional curves and vacuum 
curves. It turns out that the vacuum curve in both of the right diagrams in 
Figure 2 and Figure 3 are explicitly given as graphs. The same is true for the 
transitional curve for the outgoing contact in Ia (${\fa{S}}{\ba{S}}$) interactions (it is given by $bf=1$).
However, the transitional curve for the outgoing contact in IId (${\fa{R}}\jup$) interactions is only
implicitly determined in our analysis. The figures are schematic; see Remark \ref{C=1_curve}.

Before considering Group III interactions we state the results for Groups I and II.
We recall that the {\em strength} of a shock and rarefaction is defined as the 
pressure ratio $p_{right}/p_{left}$, while the strength of a contact wave is defined
as the specific volume ratio $\tau_{right}/\tau_{left}=\rho_{left}/\rho_{right}$. 
As a consequence all incoming strengths in pairwise interactions are assumed 
to be $\neq 1$.

\begin{theorem}[Group I interactions]\label{sum_grI}
Consider the head-on interactions of two elementary waves in \eq{gr1}. Let the 
incoming forward and backward waves have strengths $f$ and $b$, respectively 
(see left diagram in Figure 1). Then:
\begin{itemize}
	\item [(i)] the outgoing backward wave $B$ satisfies
	\beq\label{sumI1}
		B\gtrless 1 \quad \Leftrightarrow \quad b\gtrless 1
	\eeq
	\item [(ii)] the outgoing forward wave $F$ satisfies
	\beq\label{sumI2}
		F\gtrless 1 \quad \Leftrightarrow \quad f\gtrless 1
	\eeq
	\item [(iii)] the outgoing contact $C$ satisfies
	\begin{itemize}
		\item Ia (${\fa{S}}{\ba{S}}$): $\quad C\gtrless 1\quad \Leftrightarrow \quad bf\gtrless 1$
		\item Ib (${\fa{S}}{\ba{R}}$): $\quad C> 1$
		\item Ic (${\fa{R}}{\ba{R}}$): $\quad C=1$ (no outgoing contact)
	\end{itemize}
	\item [(iv)] for vacuum formation we have: there is never vacuum formation 
	in Ia and Ib interactions, while for Ic interactions vacuum occurs
	if and only if the incoming strengths $f$ and $b$ satisfy
	\beq\label{Ic_vac} 
		b^\zeta+f^{-\zeta}\leq 1\qquad\text{where $\zeta=\textstyle\frac{\gamma-1}{2\gamma}$.}
	\eeq
\end{itemize}
The situation is summarized in Figure 2.
\end{theorem}
The proof of Theorem \ref{sum_grI} is given in Section \ref{Gr_I}. Some additional 
information about the outgoing strengths in Ia and Ib interactions are recorded in 
Lemma \ref{out_in}.

\begin{figure}\label{interactns2}
	\centering
	\includegraphics[width=13.5cm,height=6cm]{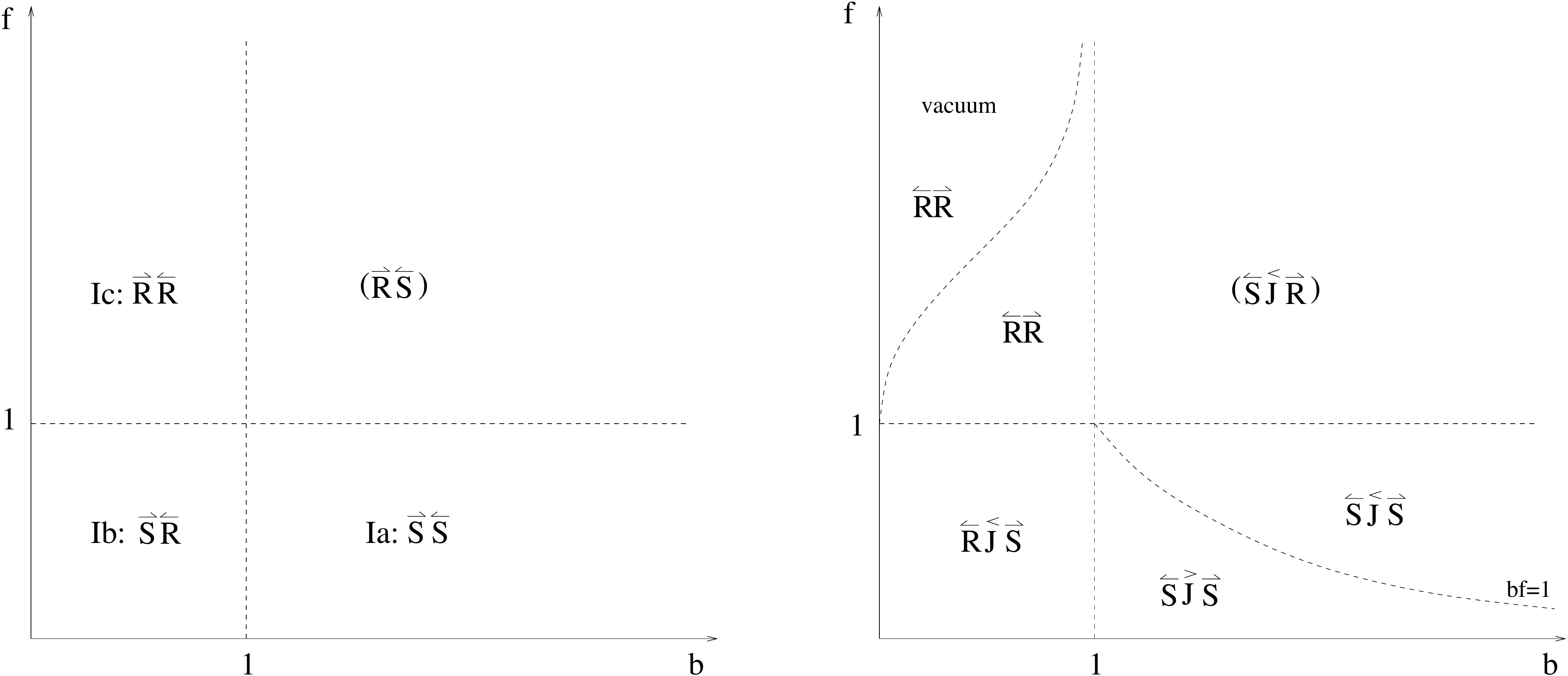}
	\caption{Group I interactions: incoming (left) and outgoing (right) waves; schematic.
	For vacuum configurations, see Remark \ref{entr_jump}.}
\end{figure} 
%
%

\begin{theorem}[Group II interactions]\label{sum_grII}
Consider the interactions of an elementary forward wave with a contact wave in \eq{gr2}. 
Let the incoming forward and contact waves have strengths $f$ and $c$, respectively 
(see middle diagram in Figure 1). Then:
\begin{itemize}
	\item[(i)] the outgoing backward wave $B$ satisfies
	\beq\label{IIa and IIb}
		- \text{ IIa and IIb (${\fa{S}}J$):}\quad B\gtrless 1 \quad \Leftrightarrow \quad c\lessgtr 1\,.
	\eeq
	\beq\label{IIc and IId}
		- \text{ IIc and IId (${\fa{R}}J$):}\quad B\gtrless 1 \quad \Leftrightarrow \quad c\gtrless 1\,.
	\eeq
	\item[(ii)]  the outgoing forward wave $F$ satisfies
	\beq\label{IIint_1}
		  F\gtrless 1\quad \Leftrightarrow \quad f\gtrless 1\,.
	\eeq
	\item[(iii)]  the outgoing contact $C$ satisfies
	\begin{itemize}
		\item IIa (${\fa{S}}\jdown$): $\quad c<C<1$
		\item IIb (${\fa{S}}\jup$): $\quad 1<C<c$
		\item IIc (${\fa{R}}\jdown$): $\quad C=c<1$
		\item IId (${\fa{R}}\jup$): in this case $f,\, c>1$ and $C<c$. However, the 
		outgoing contact may be either a $\jup$ ($C>1$) or a $\jdown$ ($C<1$) contact.
		More precisely, for fixed $c>1$ the map 
		\[f\mapsto C=C(f,c)\] 
		takes on every value in $(0,c)$ as $f$ increases from $1$ to $\infty$.
	\end{itemize}
	\item[(iv)]  for vacuum formation we have: there is never vacuum formation in 
	IIa, IIb, and IId interactions, while for IIc interactions a vacuum occurs
	if and only if the incoming strengths $f$ and $c$ satisfy
	\beq\label{IIc_vac}
		f\geq f^*(c):=\Big(\frac{2}{1-\sqrt{c}}\Big)^\frac{1}{\zeta}
		\qquad\text{where $\zeta=\textstyle \frac{\gamma-1}{2\gamma}$.}
	\eeq
\end{itemize}
The situation is summarized in Figure 3.
\end{theorem}
The proof of Theorem \ref{sum_grII} is given in Section \ref{Gr_II}. 
The location of the transition curve $\{C=1\}$ in IId interactions is analyzed in 
Section \ref{IId_trans}.

\begin{figure}\label{interactns3}
	\centering
	\includegraphics[width=13.5cm,height=6cm]{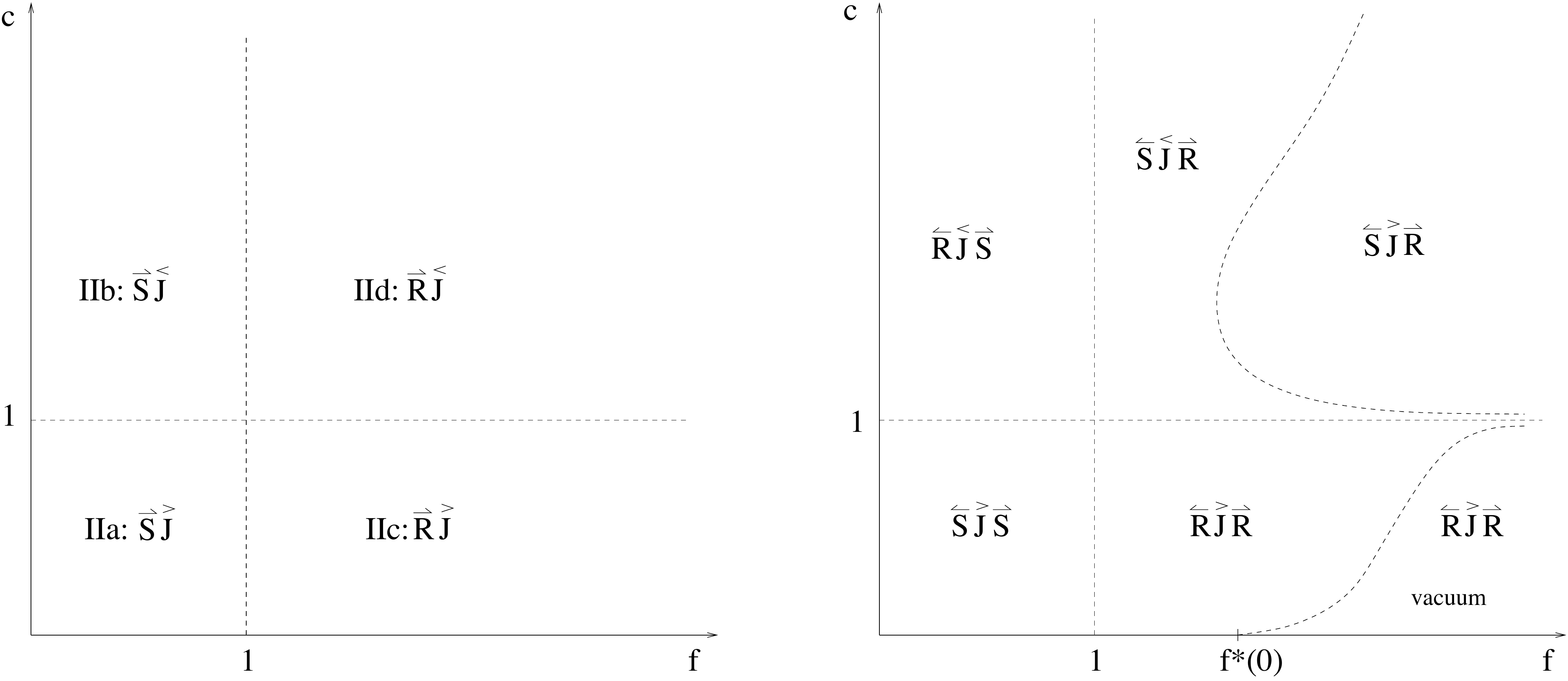}
	\caption{Group II interactions: incoming (left) and outgoing (right) waves; 
	schematic (see Remark \ref{C=1_curve}). For vacuum configurations, see Remark \ref{entr_jump}.}
\end{figure} 

We finally consider Group III interactions of two overtaking waves, which we take to be
backward waves. The possible combinations of incoming waves 
are ${\ba{S}}{\ba{S}}$ (IIIa), ${\ba{S}}{\ba{R}}$ (IIIb), and ${\ba{R}}{\ba{S}}$ (IIIc), since 
two backward rarefactions do not meet. Letting $x$ and $y$ denote the strengths of the 
left and right incoming waves, respectively, the incoming configurations are listed in the 
$(x,y)$-plane in Figure 4. 
As described in the introduction the outcomes of these interactions depend sensitively on the 
value of the adiabatic constant $\gamma$ (as well as on the incoming strengths), and 
we consider four different cases:
$\gamma\in (1,\frac{5}{3})$, $\gamma=\frac{5}{3}$, $\gamma\in (\frac{5}{3},2)$, and 
$\gamma\geq 2$. The resulting outgoing configurations are given in the $(x,y)$-plane
by Figure 5 for $\gamma\in (1,\frac{5}{3}]$, and in Figure 6 for $\gamma>\frac{5}{3}$. 

In the following statement $y_0$ denotes the unique root different
from $1$ of the function $\alpha(y):=\fa\psi(y)-\ba\psi(y)$ when $\gamma\neq \frac{5}{3}$, 
while $y_0=1$ when $\gamma=\frac{5}{3}$. (The function $\alpha$ is analyzed in 
Section \ref{alfa}.) Also, to shorten notation we write e.g.\ $B$ instead of $B(x,y)$ for the 
strength of the backward outgoing wave.

\begin{theorem}[Group III interactions]\label{sum_grIII}
Consider the interactions of two overtaking, elementary backward waves listed in \eq{gr3}. 
Let the left and right incoming waves have strengths $x$ and $y$, respectively,
and let the outgoing waves have strengths $B$, $C$ and $F$ 
(see right diagram in Figure 1). Then for all values of $\gamma>1$:
\begin{itemize}
	\item[(i)] the locus $\{B=1\}$, indicated by a dashed curve in Figures 5 - 6, coincides 
	with a $C^2$-smooth and decreasing graph $y=k(x)$.  
	The graph satisfies $k(1)=1$ and has asymptotes $x=0$ and 
	$y=\hat y(\gamma)\in(0,1)$ where $\hat y(\gamma)$ is given in \eq{yhat}.
	The outgoing backward wave $B$ is a shock (i.e.\ $B>1$) 
	if and only if $y>k(x)$.
	\item[(ii)] the locus $\{F=1\}$ consists of three curves, indicated by solid lines in 
	Figures 5 - 6: $\{x\equiv 1\}$, $\{y\equiv 1\}$, and $\{y=h(x)\}$. 
	We refer to these figures for the type of the outgoing forward (reflected) wave in 
	Group III interactions (see also Figures 8-14). As indicated there the graph 
	$y=h(x)$ have different global properties according to the value of the adiabatic 
	constant $\gamma>1$. However, for each value of $\gamma>1$:
	\begin{itemize}
		\item $y=h(x)$ has a horizontal asymptote $y_*=y_*(\gamma)$ as $x\uparrow +\infty$,
		\item the graphs $y=k(x)$ and $y=h(x)$ intersect at the 
		point $(\frac{1}{y_0},y_0)$. 
	\end{itemize}
	\item[(iii)] the outgoing contact $C$ satisfies:
	\begin{itemize}
		\item $C<1$ in ${\ba{S}}{\ba{S}}$ interactions
		\item $C>1$ in ${\ba{S}}{\ba{R}}$ and ${\ba{R}}{\ba{S}}$ interactions.
	\end{itemize}
	\item[(iv)] for vacuum formation we have: a vacuum never 
	appears in ${\ba{S}}{\ba{S}}$ and ${\ba{R}}{\ba{S}}$ interactions, while a vacuum 
	appears in an ${\ba{S}}{\ba{R}}$ interaction if and only if the incoming strengths 
	$y<1<x$ satisfy
	\beq\label{IIIb_vacu}
		0<y\leq \left[\textstyle\frac{1}{2}\left(1-\frac{v(x)}{\nu}\right)\right]^\frac{1}{\zeta}\,.
	\eeq
	where 
	\[\zeta={\textstyle\frac{\gamma-1}{2\gamma}}\qquad\text{and}\qquad 
	v(x):=\frac{\nu\sqrt{x+a}+\kappa(x-1)}{\sqrt{x+a x^2}}\]
\end{itemize}
The situation is summarized in Figures 4, 5, and 6.
\end{theorem}
Theorem \ref{sum_grIII} follows from the propositions in Sections \ref{Gr_III_i}, 
\ref{group3_reflected}, \ref{group3_contact}, and \ref{group3_vac}. The analysis in 
these sections provide further details about the location and intersection properties 
of the graphs $y=k(x)$ and $y=h(x)$; see Figures 8-14. When possible we also
give explicit expressions for $k(x)$ and $h(x)$.

\begin{figure}\label{interactns4}
	\centering
	\includegraphics[width=6.8cm,height=5.8cm]{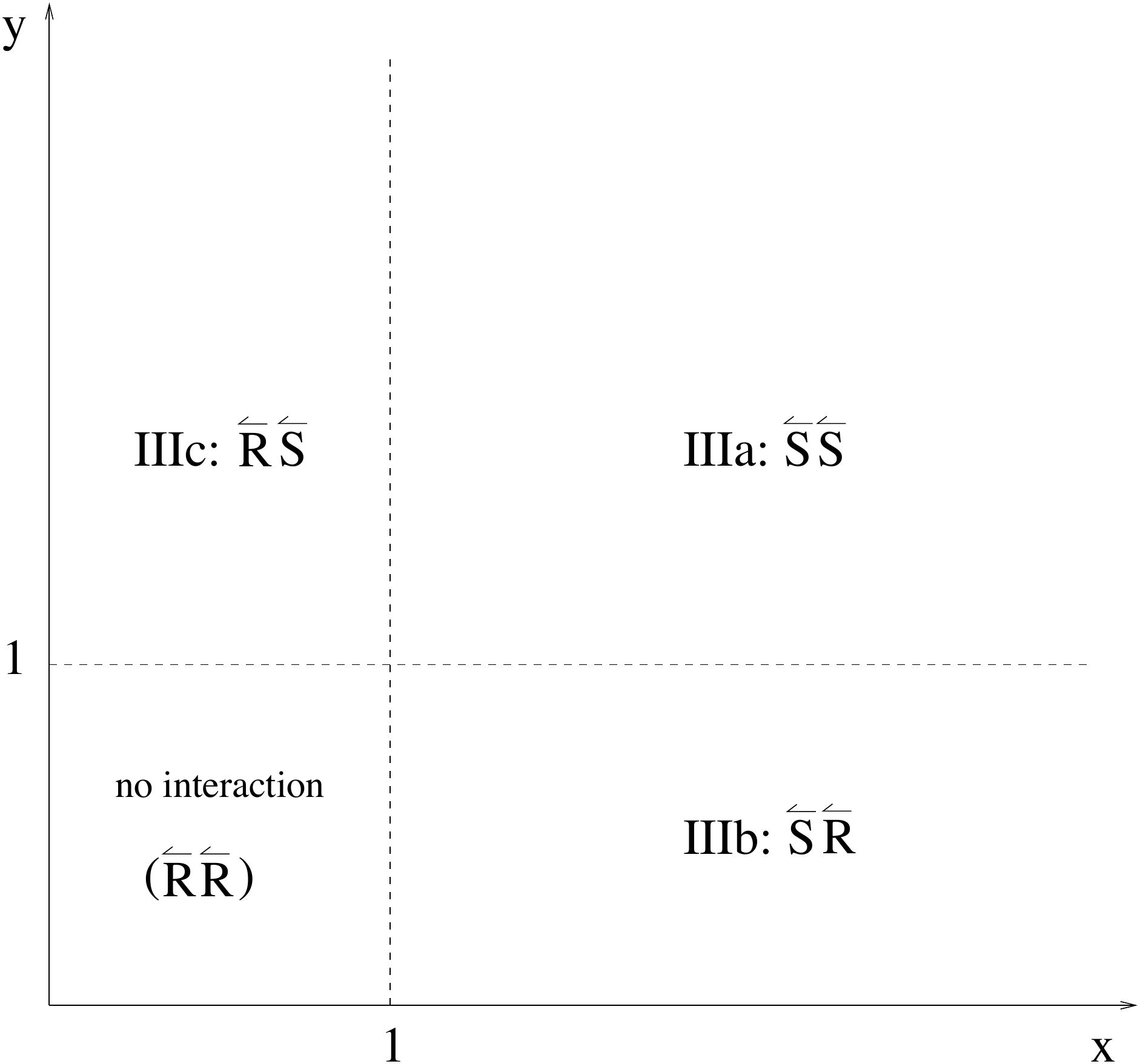}
	\caption{Group III interactions: incoming waves.}
\end{figure} 
%
%

\begin{figure}\label{interactns5}
	\centering
	\includegraphics[width=14cm,height=6.2cm]{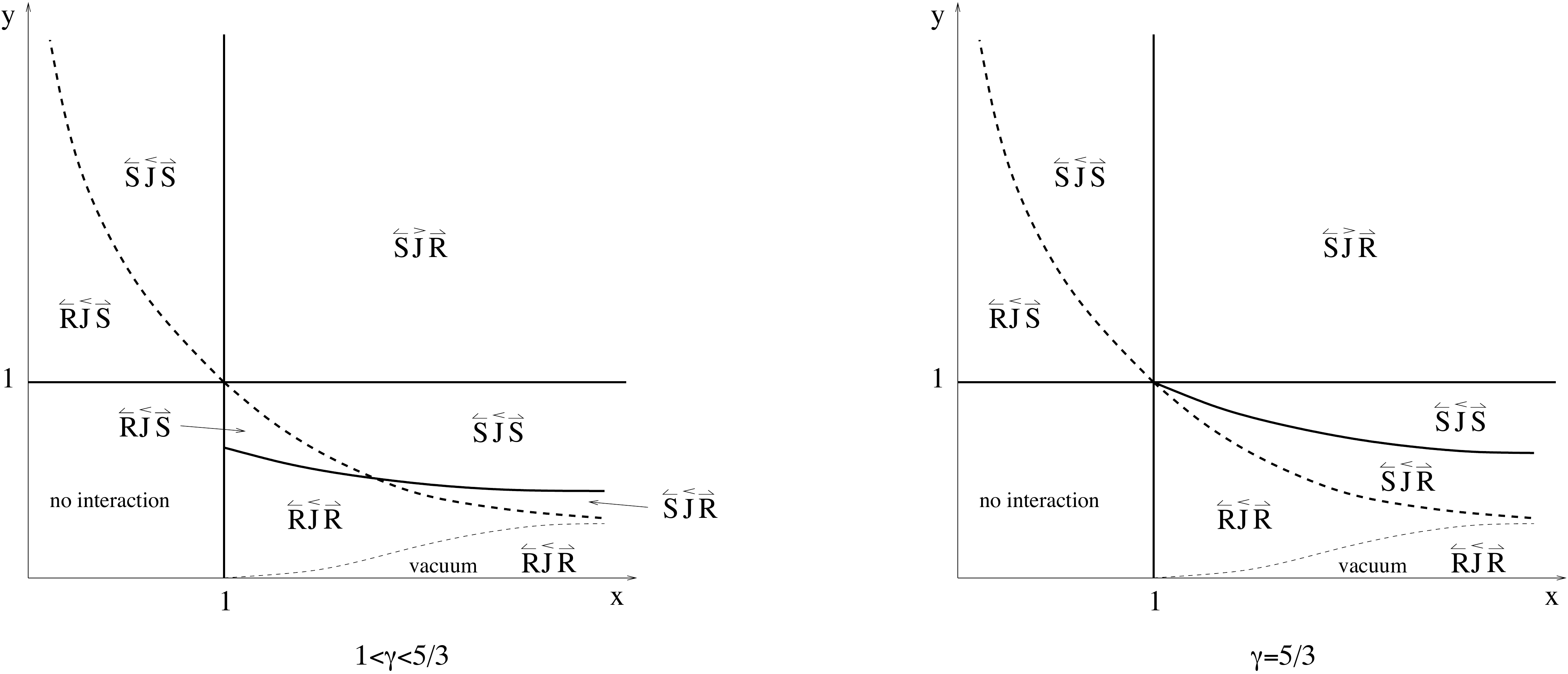}
	\caption{Group III interactions: outgoing waves when $1<\gamma<\frac{5}{3}$ 
	(left) and when $\gamma=\frac{5}{3}$ (right); schematic. For vacuum configurations, see Remark \ref{entr_jump}.}
\end{figure} 
%
%

\begin{figure}\label{interactns6}
	\centering
	\includegraphics[width=14cm,height=6.2cm]{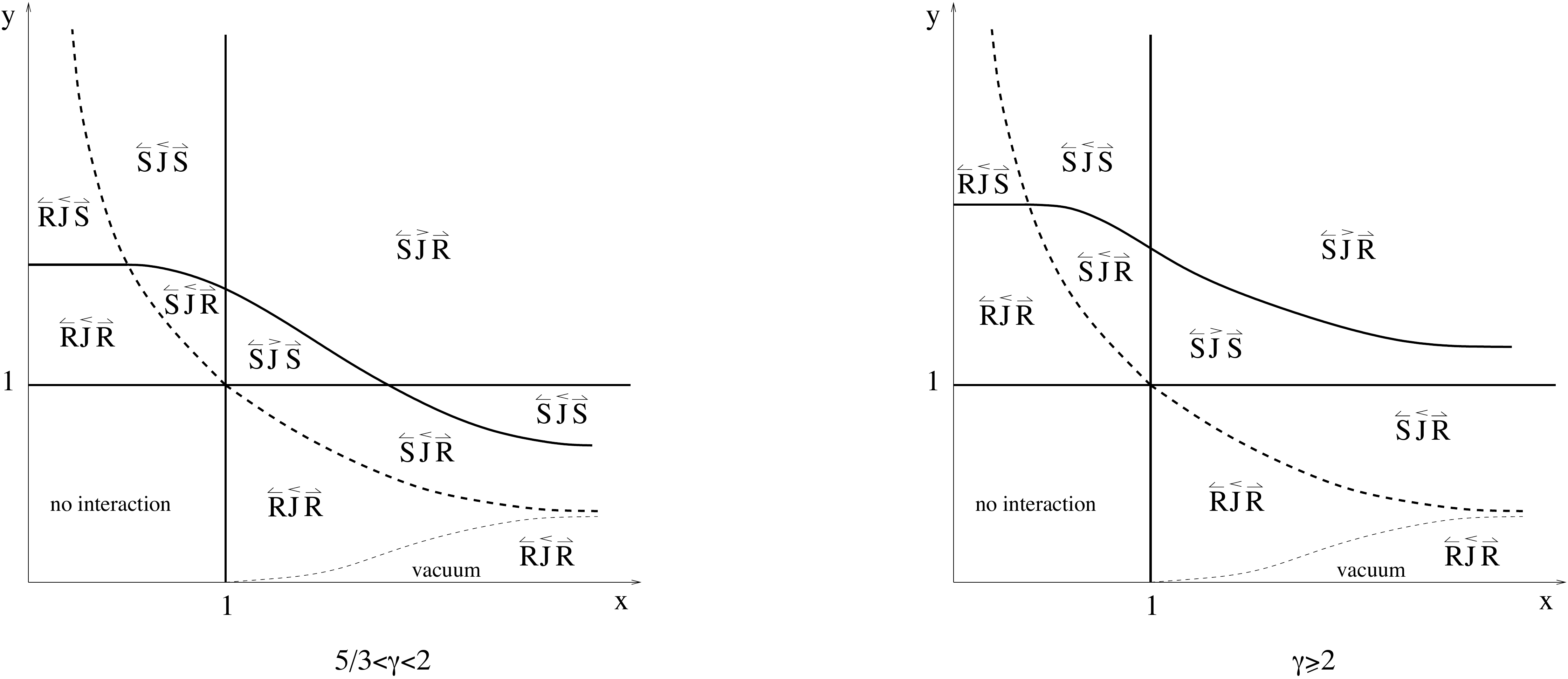}
	\caption{Group III interactions: outgoing waves when $\frac{5}{3}<\gamma<2$ (left) 
	and when $\gamma\geq 2$ (right); schematic. For vacuum configurations, see Remark \ref{entr_jump}.}
\end{figure} 

\begin{remark}
	Before starting on the proofs of Theorems  \ref{sum_grI},  \ref{sum_grII}, and  \ref{sum_grIII} 
	we note that the analysis will make use of a number of auxiliary functions and results about 
	their properties. These are for the most part collected in Section \ref{aux}.
\end{remark}

\section{Proof of Theorem \ref{sum_grI}}\label{Gr_I}

\subsection{Group I: head-on interactions} Consider the interactions in Group I 
listed in \eq{gr1}. Referring to the left diagram in Figure 1 we use 
\eq{bkwd_wave}-\eq{frwd_wave} to traverse the waves before and after interaction. 
We obtain three equations for the outgoing strengths $B$, $C$, $F$ in terms of the 
incoming strengths $b$ and $f$:
\bea
	\ba\phi(b)\fa\phi(f) &=& C\ba\phi(B)\fa\phi(F)\label{Iint1}\\
	\fa\psi(f) -\ba\psi(b)\sqrt{f\fa\phi(f)}&=& -\ba\psi(B)+\fa\psi(F)\sqrt{CB\ba\phi(B)}\label{Iint2}\\
	bf&=&BF\,.\label{Iint3}
\eea
Eliminating $F$ and $C$, using the relation \eq{rel3} and the definitions \eq{M}, \eq{N} of the 
auxiliary functions $M$ and $N$, we obtain the following nonlinear equation for $B=B(b,f)$:
\beq\label{IB}
	\cG(B;b,f)=0\,,
\eeq
where
\beq\label{G}
	\cG(B;b,f):= \ba\psi(B)+\ba\psi\big({\textstyle\frac{B}{bf}}\big)M(b)N(f)
	+\fa\psi(f)-N(f)\ba\psi(b)\,.
\eeq
We observe that since $\ba\psi$ is strictly increasing, the map $B\mapsto \cG(B;b,f)$
has the same property.

In what follows we first analyze the strengths of the outgoing backward and forward 
waves $B$ and $F$ and establish parts (i) and (ii) of Theorem \ref{sum_grI}.
This information is then used to prove the claims about the outgoing contact $C$. 
Finally we establish the claims about vacuum formation in head-on interactions.

\medskip

\subsection{Proof of Theorem \ref{sum_grI} part (i)}
By the defining equation \eq{IB} for $B=B(b,f)$ and the monotonicity of $B\mapsto \cG(B;b,f)$, 
we have by  \eq{M} that
\beq\label{IB_type2}
	B\gtrless 1\quad\Leftrightarrow\quad \cG(B;b,f)=0\gtrless  \cG(1;b,f) =N(f)g(b,f)\,,
\eeq
where $N(f)>0$ and 
\[g(b,f):=\ba\psi\big({\textstyle\frac{1}{bf}}\big)M(b)-\ba\psi(b)-\ba\psi\big(\textstyle\frac{1}{f}\big)\,.\]
Thus, to verify \eq{sumI1} we need to show that 
\beq\label{g1}
	g(b,f)\gtrless 0\quad\Leftrightarrow\quad  b\lessgtr 1\,.
\eeq
Since $g(1,f)\equiv 0$ it suffices to show that $\del_b g(b,f)<0$ for all $b,f>0$.
We have
\[\del_b g(b,f)=M'(b)\ba\psi\big({\textstyle\frac{1}{bf}}\big) 
- \frac{1}{b^2f}M(b)\ba\psi {'}\big({\textstyle\frac{1}{bf}}\big)-\ba\psi {'}(b)\,.\]
Using the properties of $M$ and $\ba\psi$ we get
\[\del_b g(b,f)<M'(b)\ba\psi\big({\textstyle\frac{1}{bf}}\big) \leq 0\qquad\text{whenever $bf\geq 1$.}\] 
On the other hand, if $bf< 1$ then
\[\del_b g(b,f)<M'(b)\ba\psi\big({\textstyle\frac{1}{bf}}\big) 
- \frac{1}{b^2f}M(b)\ba\psi {'}\big({\textstyle\frac{1}{bf}}\big)
=\frac{M(b)}{b}\ba\psi\big({\textstyle\frac{1}{bf}}\big)\left[m(b)-\ell\big({\textstyle\frac{1}{bf}}\big)\right]<0,\]
since $m(b)<\frac{1}{2}$ for all $b>0$, while $\ell(q)>\frac{1}{2}$ for all $q>1$ (see Section \ref{aux}). 
This proves \eq{sumI1}.

\medskip

\subsection{Proof of Theorem \ref{sum_grI} part (ii)}
Using \eq{Iint3} and the fact that $B\mapsto \cG(B;b,f)$ is increasing we have
\[F=\frac{bf}{B}\gtrless 1\quad\Leftrightarrow\quad 
B\lessgtr bf \quad\Leftrightarrow\quad 
0=\cG(B;b,f)\lessgtr G(b,f)\,.\]
where 
\[ G(b,f):=\cG(bf;b,f)=\ba\psi(bf)+\fa\psi(f)-N(f)\ba\psi(b)\,.\]
We have $G(b,1)\equiv 0$, while
\[\del_b G(b,f)= f\ba\psi {'}(bf)-N(f)\ba\psi {'}(b)\,,\qquad 
\del_f G(b,f)= b\ba\psi {'}(bf)+\fa\psi {'}(f)-N'(f)\ba\psi(b)\,.\]
Using the properties of  the auxiliary functions we see that each term in  $\del_f G(b,f)$ is 
non-negative whenever $0<b\leq 1$, and that their sum in this case is strictly positive. 
Thus, \eq{sumI2} holds for $0<b<1$. For $b\geq1$ we consider instead how $G$ changes 
along hyperbolas $bf=const$. Differentiating in the direction of increasing $f$ we have, 
for $b\geq 1$, that 
\bea
	(-b,f)\cdot\nabla_{(b,f)}G(b,f)&=& b\ba\psi {'}(b)N(f)+f\fa\psi {'}(f)-fN'(f)\ba\psi(b)\nn\\
	&>& b\ba\psi {'}(b)N(f)-fN'(f)\ba\psi(b)\nn\\
	&=& \ba\psi(b)N(f)[\ell(b)-n(f)]>0\label{dir_der}
\eea
by the properties of the auxiliary functions $\ell$ and $n$ recorded in Section \ref{aux}. 
Again, since $G(b,1)\equiv 0$, it follows that \eq{sumI2} holds in the set $(b,f)\in\{b\geq 1,\, bf\geq 1\}$.
We finally observe that  
\[G(1,f)=\ba\psi(f)+\fa\psi(f)< 0\quad\text{whenever} \quad f< 1\,,\]
which, together with \eq{dir_der}, verifies \eq{sumI2} also in $\{b\geq 1,\, bf<1\}$. 
This proves \eq{sumI2}.

\medskip

\subsection{Proof of Theorem \ref{sum_grI} part (iii)}
Concerning the outgoing contact $C$ in Group I interactions we first observe that $C=1$
in $\fa R\ba R$ (Ic) interactions. Indeed, by Remark \ref{entr_jmp},
if $C\neq 1$ then the entropy would take different values on either side of the outgoing contact.
However, we have already shown that the outgoing backward and forward waves are both 
rarefactions in this case, whence (Remark \ref{entr_jmp}) the entropy takes the same value on 
each side of the outgoing contact. 

For the other two interactions in Group I we shall use \eq{Iint1} together with the following:
\begin{lemma}\label{out_in}
	For Ia and Ib interactions with incoming strengths $b$, $f$, and outgoing strengths 
	$B$, $C$, $F$ (see Figure 1), the following holds:
	\begin{itemize}
		\item For Ia ($\fa S \ba S$) intersections: $f<F<1<B<b
		\quad\text{and}\quad B+F<b+f\,.$
		\item For Ib ($\fa S\ba R$) interactions: $F<f<1$.
	\end{itemize}
\end{lemma}
\bp Consider first Ia interactions for which $f<1<b$. The properties of the auxiliary 
functions show that
\[\cG(b;b,f)=\ba\psi(b)\big[1-N(f)\big] + \fa\psi(f)\big[1-M(b)\big] >0=\cG(B;b,f).\]
As $\cG(\cdot;b,f)$ is increasing this shows that $b>B$, and since $bf=BF$ it follows
that $F>f$. Combining this with \eq{sumI1} and \eq{sumI2}, we obtain that $f<F<1<B<b$. 
But then $b-f>B-F>0$, whence $(b-f)^2>(B-F)^2$. Since $bf=BF$ it then follows that 
$(b+f)^2>(B+F)^2$, i.e.\ $B+F<b+f$.

Next, for Ib interactions the incoming strengths satisfy $b,\, f<1$.
As $bf=BF$, we have $f>F$ if and only if $B>b$, which is the case if and only if
\[\cG(B;b,f)=0>\cG(b;b,f)= \ba\psi(b)\big[1-N(f)\big] + \fa\psi(f)\big[1-M(b)\big]\,.\]
For $b,\, f<1$ both terms in  the expression on the right are negative. This shows 
that $f>F$ in Ib interactions.
\ep

For a Ia interaction we use $f<1<b$ in the explicit expressions for 
$\ba \phi$ and $\fa \phi$, together with \eq{Iint3}, to rewrite \eq{Iint1} as
\[C=\frac{A(b+f,bf)}{A(B+F,bf)}\,,\] 
where the function $A$ is defined in \eq{Axi}. We then use the property 
\eq{Aa_prop} of $A$ together with the estimate $B+F<b+f$ from Lemma \ref{out_in}, 
to obtain that $bf\gtrless 1$ if and only if $C\gtrless 1$.

Finally, in $\fa S\ba R$-interactions $f,\, b<1$, and by parts (i) and (ii) of Theorem  
\ref{sum_grI} we have $B,\, F<1$. Again, using the explicit expressions for $\ba \phi$ 
and $\fa \phi$, together with \eq{Iint3}, we  rewrite \eq{Iint1} as
\beq\label{SR_C1}
	C=\frac{D(f)}{D(F)}\,.
\eeq
where the function $D$ is defined in \eq{D}. Since $D$ is strictly increasing and 
$f>F$ by Lemma \ref{out_in}, we obtain that $C>1$ in $\fa S\ba R$ interactions.
This concludes the proof of part (iii) of Theorem \ref{sum_grI}.

\medskip

\subsection{Proof of Theorem \ref{sum_grI} part (iv)}
We argue as in Section \ref{rip_vac} and observe that the map $B\mapsto \cG(B;b,f)$ 
is strictly increasing. The interaction Riemann problem is vacuum-free if and only if
$B=B(b,f)>0$, or equivalently $\cG(B;b,f)=0>\cG(0;b,f)$, i.e.\
\beq\label{I_no_vac}
	\ba\psi(b)+\nu M(b)>\frac{\fa\psi(f)-\nu}{N(f)}\,.
\eeq
Using the properties of the auxiliary functions $\ba\psi$, $\fa\psi$ and $M$ (see 
Section \ref{aux}) we have:
\begin{itemize}
\item for Ia and Ib interactions, $f<1$ such that \eq{I_no_vac} is satisfied:
\[\lhs{I_no_vac} >\ba\psi(b)>-\nu>\ba\psi(f)-\nu>\rhs{I_no_vac};\]
\item for Ic interactions $b<1<f$ and \eq{I_no_vac} reduces to the condition that
\[b^\zeta+f^{-\zeta}>1\qquad\qquad \mbox{(no vacuum)}\,.\]
\end{itemize}
The ``vacuum curve'' $b^\zeta+f^{-\zeta}=1$ is sketched in the right diagram in Figure 2. 
This establishes part (iv) of Theorem \ref{sum_grI}.

\noindent
This concludes the proof of Theorem \ref{sum_grI}.

\section{Proof of Theorem \ref{sum_grII}}\label{Gr_II}

\subsection{Group II: interactions with a contact} Consider the interactions in Group II 
listed in \eq{gr2}. Referring to the middle diagram in Figure 1 we use 
\eq{bkwd_wave}-\eq{frwd_wave} to traverse the waves before and after interaction. 
We obtain three equations for the outgoing strengths $B$, $C$, $F$ in terms of the 
incoming strengths $f$ and $c$:
\bea
	c\fa\phi(f) &=& C\ba\phi(B)\fa\phi(F)\label{IIint1}\\
	\fa\psi(f) &=& -\ba\psi(B)+\fa\psi(F)\sqrt{CB\ba\phi(B)}\label{IIint2}\\
	f&=&BF\,.\label{IIint3}
\eea
Eliminating $F$ and $C$, using the relation \eq{rel3} and the definition \eq{N}, 
and rearranging, we obtain the following nonlinear equation for $B=B(f,c)$
\beq\label{IIB}
	\cH(B;f,c)=0
\eeq
where
\beq\label{H}
	\cH(B;f,c):= \ba\psi(B)+\sqrt{c}N(f)\ba\psi\big({\textstyle\frac{B}{f}}\big)+\fa\psi(f)\,.
\eeq
We observe that since $\ba\psi$ is strictly increasing, the map $B\mapsto \cH(B;f,c)$
has the same property.

We proceed to analyze the strengths of the outgoing backward and forward waves 
$B$ and $F$. This information is then used to verify the claims about the outgoing 
contact $C$ in Theorem \ref{sum_grII}. We then establish part (iv) of  Theorem 
\ref{sum_grII} concerning vacuum formation in interactions with a contact.
Finally, in Section \ref{IId_trans} we include a partial analysis of the location of the 
transition curve $\{C=1\}$ in the $(f,c)$-plane, and show that it is the graph 
of a non-monotone function $c\mapsto \mathfrak f(c)$ satisfying
\beq\label{f_lims}
	\lim_{c\downarrow 1}\mathfrak f(c)=\lim_{c\uparrow \infty}\mathfrak f(c)=+\infty\,.
\eeq

\medskip

\subsection{Proof of Theorem \ref{sum_grII} part (i)}
By the monotonicity of $B\mapsto \cH(B;f,c)$, the defining equation 
\eq{IIB} for $B=B(f,c)$, and the definition \eq{N} of the function $N$, we have
\[B\gtrless 1\quad \Leftrightarrow \quad 0\gtrless \cH(1;f,c)=(1-\sqrt{c})\fa\psi(f)\,.\]
Then, using the sign properties of $\fa\psi$, we obtain: for $\fa S J$ interactions we have $f<1$ 
and \eq{IIa and IIb} follows, while for an $\fa R J$ interactions we have $f>1$, and 
\eq{IIc and IId} follows. This verifies part (i) of Theorem \ref{sum_grII}.

\medskip

\subsection{Proof of Theorem \ref{sum_grII} part (ii)}
We observe that $\cH(f;f,c)=\ba\psi(f)+\fa\psi(f)$.
Thus, by \eq{IIint3}, the monotonicity of $\cH(\cdot;f,c)$, and the defining equation 
\eq{IIB} for $B=B(f,c)$, we have
\[F\gtrless 1\quad \Leftrightarrow \quad f\gtrless B
\quad \Leftrightarrow \quad \cH(f;f,c)\gtrless 0
\quad \Leftrightarrow \quad \ba\psi(f)+\fa\psi(f)\gtrless 0
\quad \Leftrightarrow \quad f\gtrless 1\,.\]
The last equivalence follows from the properties of the auxiliary functions
$\ba\psi$ and $\fa\psi$. This verifies part (ii) of Theorem \ref{sum_grII}.

\medskip

\subsection{Proof of Theorem \ref{sum_grII} part (iii)}
As for Group I, in order to determine the type of the transmitted contact
we will consider each interaction separately. The outgoing ratio $C$
occurs in both \eq{IIint1} and \eq{IIint2}, and we will make use of both relations.

\medskip
\paragraph{\bf Case IIa} ($\fa S\jdown$) In this case $f,\, c<1$, and from the analysis 
above we have that $F<1<B$. Thus $f=BF>F$, while a direct evaluation shows that 
$B\ba\phi(B)>1$. As $\fa\psi$ is increasing we have $0>\fa\psi(f)>\fa\psi(F)$, and it follows that
\[\fa\psi(f)+\ba\psi(B)>\fa\psi(F)>\fa\psi(F)\sqrt{B\ba\phi(B)}\,.\]
Comparing with \eq{IIint2} shows that $C<1$ in this case. We further claim that $C>c$, which, 
according to \eq{IIint1}, is the case if and only if 
\[\fa\phi(f)\equiv \fa\phi(BF)>\ba\phi(B)\fa\phi(F)\,,\]
which is equivalent to $E(BF)>E(B)E(F)$, where the function $E$ is defined in \eq{E}.
We apply part (a) of Lemma \ref{E_lem} and conclude that this last inequality is 
indeed satisfied since $BF=f<1$.

\medskip
\paragraph{\bf Case IIb} ($\fa S\jup$) In this case $f<1<c$, and from the analysis 
above we have that $F,\, B<1$. Thus $f=BF<F$ and the properties of the auxiliary 
functions $\fa\psi$ and $\ba\phi$ show that 
\[\fa\psi(f)+\ba\psi(B)<\fa\psi(f)<\fa\psi(F)<\fa\psi(F)B^\zeta=\fa\psi(F)\sqrt{B\ba\phi(B)}\,.\]
A comparison with \eq{IIint2} shows that $C>1$. We further claim that $C<c$, which, 
according to \eq{IIint1}, is the case if and only if 
\[\fa\phi(f)\equiv \fa\phi(BF)<\ba\phi(B)\fa\phi(F)\,.\]
Using the functions $D$, $E$ (defined in \eq{D} and \eq{E}) we rewrite this as
\beq\label{C_less_c}
	E(BF)<\frac{E(B)E(F)}{D(B)}\,.
\eeq
Now, as $B,\, F<1$, part (b) of Lemma \ref{E_lem} shows that $E(BF)<E(B)E(F)$.
At the same time $D$ is a strictly increasing function with $D(1)=1$, such that 
$D(B)<1$. It follows that \eq{C_less_c} is satisfied and $C<c$.

\medskip
\paragraph{\bf Case IIc} ($\fa R\jdown$) In this case $c<1<f$, whence, by the analysis
above, $B<1<F$. This means that the backward and forward outgoing waves are
both rarefactions, and it follows from Remark \ref{entr_jmp} that $C=c$ in this case.

\medskip
\paragraph{\bf Case IId} ($\fa R\jup$) In this case $c,\, f>1$, and the earlier analysis 
shows that $1<F,\, B<f$. 
Using the explicit expressions for $\fa\phi$ and $\ba\phi$ in 
\eq{IIint1}, together with \eq{IIint3}, we get
\beq\label{IId_Cc}
	C=\frac{\fa\phi(f)c}{\fa\phi(F)\ba\phi(B)}=\frac{c}{D(B)}\,,
\eeq
where $D$ is defined in \eq{D}. Since $B>1$, $D(1)=1$, and $D$ is strictly increasing, 
it follows from \eq{IId_Cc} that $C<c$ in IId ($\fa R\jup$) interactions.

Next we want to verify that $C$ may be either $\gtrless 1$ depending on the 
incoming strengths $f$ and $c$. First, as $B,\, F>1$ and $BF=f$, it follows 
from the defining relation $\cH(B;f,c)=0$ and the explicit expressions for 
$N$, $\fa\psi$ and $\ba\psi$, that
\beq\label{Bfc}
	(\sqrt{c}-1)f^\zeta +1= \sqrt{c}B^\zeta +\frac{\ba\psi(B)}{\nu}\,.
\eeq
The right-hand side is strictly increasing with respect to $B$, such that this relation 
provides $B$ as a function of $f$ and $c$: 
\[B=\mathcal B(f,c)=\text{unique $B$-value satisfying \eq{Bfc} for given $f,\, c>1$.}\] 
By substituting $B=\mathcal B(f,c)$ into \eq{IId_Cc} we obtain the strength
$C=\mathcal C(f,c)$ of the outgoing contact. We now have:
\begin{lemma}\label{IId_C}
	Consider an $\fa R\jup$ interaction with incoming strengths $f,\, c>1$. Then, as 
	$f$ increases from $1$ to $\infty$ while $c>1$ is kept fixed, the strength $C=\mathcal C(f,c)$ 
	of the outgoing contact (determined from \eq{Bfc} and \eq{IId_Cc}) decreases from $c$ 
	to $0$.
\end{lemma}
\bp
	Recall that $B\mapsto D(B)$ increases from $1$ to $\infty$ as $B$ increases 
	from $1$ to $\infty$. Thus the claim follows from the expression \eq{IId_Cc} for $C$
	provided the map $f\mapsto \cB(f,c)$, for fixed $c>1$, increases from 
	$1$ to $\infty$ as $f$ increases from $1$ to $\infty$. Now, $\cB(f,c)$ is the unique 
	root of \eq{Bfc}. For fixed $c>1$ the left-hand side of \eq{Bfc} increases from $1$ to 
	$\infty$ as $f$ increases from $1$ to $\infty$. At the same time the right-hand side 
	of \eq{Bfc} increases from $1$ to $\infty$ as $B$ increases from $1$ to $\infty$. 
	This shows that $f\mapsto \cB(f,c)$ increases 
	from $1$ to $\infty$ as $f$ increases from $1$ to $\infty$.
\ep
This establishes part (iii) of Theorem \ref{sum_grII}.

\medskip

\subsection{Proof of Theorem \ref{sum_grII} part (iv)}
We argue as in Section \ref{rip_vac} and observe that the map $B\mapsto \cH(B;f,c)$ 
is strictly increasing. The interaction Riemann problem is vacuum-free if and only if
$B=B(f,c)>0$, or equivalently $\cH(B;f,c)=0>\cH(0;f,c)$, i.e.
\beq\label{II_no_vac}
	\nu\big[1+\sqrt{c}N(f)\big] - \fa\psi(f)>0\,.
\eeq
\begin{itemize}
	\item IIa and IIb: for an $\fa S J$ interaction we have $f<1$, such that 
	$\lhs{II_no_vac}>0$. Thus no vacuum occurs in $\fa S J$ interactions.
	\item IIc and IId: for $\fa R J$ interactions we have $f>1$. Using the explicit
	expressions for the functions $\ba \psi$, $\fa\psi$, $N$, we get that no vacuum 
	occurs if and only if 
	\beq\label{IIcd_no_vac}
		(1-\sqrt{c})f^\zeta<2\,.
	\eeq
	For IId interactions $c>1$, \eq{IIcd_no_vac} is satisfied, and no vacuum 
	occurs. On the other hand, for IIc interactions we have $c<1$, and no vacuum 
	appears if and only if $f<f^*(c)$, where $f^*(c)$ is defined in \eq{IIc_vac}.
\end{itemize}

\noindent
This concludes the proof of Theorem \ref{sum_grII}.

\medskip

\subsection{Location of the transition curve $\{C=1\}$ for IId interactions}\label{IId_trans}
We now want to determine more precisely the ``transition" curve in the $(f,c)$-plane 
across which the outgoing contact changes type IId interactions. As above we denote 
the outgoing strength of the contact by $\mathcal C(f,c)$. By Lemma \ref{IId_C} 
we know that for each $c>1$ there is a unique $f$-value $\mathfrak f(c)$ such that 
$\mathcal C(\mathfrak f(c),c)=1$. We derive an expression for $\mathfrak f(c)$ as 
follows. First, from \eq{IId_Cc} with $C=1$, we get that
\[c=D(B)\,.\]
Let $\delta:=D^{-1}$ such that 
\[B=\delta(c)=\text{outgoing backward strength when no outgoing contact occurs.}\]
Substituting this value for $B$ into \eq{Bfc} then yields the corresponding value of 
the incoming forward wave $f$, i.e.\ $\mathfrak f(c)$:
\[\mathfrak f(c)=\left[\delta(c)^\zeta+\frac{\fa\psi(\delta(c))
+\ba\psi(\delta(c))}{\nu(\sqrt{c}-1)}\right]^\frac{1}{\zeta}\,,\]
where we have made use of the explicit expressions for $\fa\psi$.
It turns out that this is a non-monotone function. Figure 7 shows the graph of the function 
\[\Upsilon(B):=B^\zeta+\frac{\fa\psi(B)+\ba\psi(B)}{\nu(\sqrt{D(B)}-1)}\]
when $\gamma=\frac{4}{3}$. (The plot shows the same qualitative features for 
other values of $\gamma$.) As $c\mapsto \delta (c)$ is strictly increasing it follows that 
$c\mapsto \mathfrak f(c)$ is also non-monotone. In particular, the plot indicates that 
$\mathfrak f(c)$ has a unique minimum. Finally, a direct evaluation 
verifies \eq{f_lims}. 
\begin{figure}\label{upsilon}
	\centering
	\includegraphics[width=7.5cm,height=6.5cm]{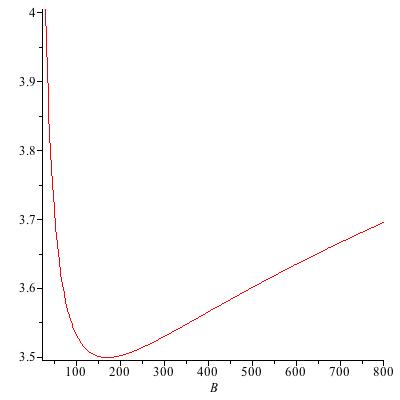}
	\caption{Graph of $\Upsilon(B)$ ($\gamma=\frac{4}{3}$)}
\end{figure}
\begin{remark}\label{C=1_curve}
	We have not proved that $\mathfrak f(c)$ has a unique minimum, nor
	determined its absolute minimum. In particular, we have not determined 
	whether its absolute minimum is smaller or larger than $f^*(0)=2^{1/\zeta}$.
	Figure 3 should be read with these provisions in mind. 
\end{remark}

\section{Proof of Theorem \ref{sum_grIII}: setup and part $\mathrm{(i)}$}\label{Gr_III_i}
The analysis of interactions involving overtaking waves is more involved. From Theorems
\ref{sum_grI} and \ref{sum_grII} we see that most outcomes in Groups I and II interactions
are independent of the adiabatic constant $\gamma$. Indeed, only the transition curve 
for the outgoing contact $C$ in IId interactions, as well as the transition curves for vacuum 
in Ic and IIc interactions, depend explicitly on the adiabatic constant. Also, these dependencies 
are ``stable" in the sense that the transition curves are present and qualitatively similar for 
all values of $\gamma>1$. The situation for Group III interactions is markedly different.

We consider the pairwise interactions of overtaking backward waves of strengths $x$ 
and $y$ (see right diagram in Figure 1). It will turn out that the location and properties of 
the transition curve for the reflected (forward) wave, i.e.\ the locus $\{F(x,y)=1\}$, depends 
sensitively on $\gamma$. At the same time, the transition curve of the outgoing backward 
wave, viz.\ $\{B(x,y)=1\}$, as well as the vacuum transition 
curve, both depend on $\gamma$. However, in the latter cases the dependence is stable
in the above sense. The situation is depicted in Figures 5 and 6.

For reference we  introduce the following three regions delimited by the lines $x=1$ and 
$y=1$ in the first quarter of the $(x,y)$-plane:
\begin{itemize}
	\item{} IIIa $=\{x,\, y>1\}$, corresponding to $\ba S\ba S$-interactions\\
	\item{} IIIb $=\{x>1>y>0\}$, corresponding to $\ba S\ba R$-interactions\\
	\item{} IIIc $=\{y>1>x>0\}$, corresponding to $\ba R\ba S$-interactions.
\end{itemize}
(The region $\{x,\, y<1\}$ corresponds to two backward rarefaction waves, which do not meet.)

In this section we determine  the type of transmitted (backward) wave in Group III interactions. 
Sections \ref{group3_reflected} and \ref{group3_contact} provide the analysis of the reflected 
(forward) and contact waves, respectively, while Section \ref{group3_vac} gives the conditions 
for vacuum formation in overtaking interactions. Together these results establish parts (i)-(iv)
of Theorem \ref{sum_grIII}.

\subsection{Group III: overtaking interactions} 
Consider the interactions in Group III listed in \eq{gr3}. Referring to the right diagram in 
Figure 1 we use \eq{bkwd_wave}-\eq{frwd_wave} to traverse the waves before and after 
interaction. We obtain three equations for the outgoing strengths $B$, $C$, $F$ in terms of the 
incoming strengths $x$ and $y$:
\bea
	\ba\phi(x)\ba\phi(y)  &=& C\ba\phi(B)\fa\phi(F)\label{IIIint1}\\
	\ba\psi(x)+\ba\psi(y)\sqrt{x\ba\phi(x)} &=& \ba\psi(B)-\fa\psi(F)\sqrt{CB\ba\phi(B)}\label{IIIint2}\\
	xy&=&BF\,.\label{IIIint3}
\eea
Eliminating $F$ and $C$, using \eq{rel3} and \eq{M}, and rearranging, yield the following 
nonlinear equation for $B=B(x,y)$
\beq\label{IIIB}
	\cK(B;x,y)=0\,,
\eeq
where
\beq\label{cK}
	\cK(B;x,y):= \ba\psi(B)+\ba\psi\big({\textstyle\frac{B}{xy}}\big)M(x)M(y)-\ba\psi(x)-\ba\psi(y)M(x)\,.
\eeq
We observe that since $\ba\psi$ is strictly increasing, the map $B\mapsto \cK(B;x,y)$
has the same property.

\subsection{Proof of Theorem \ref{sum_grIII} part (i)}\label{part1_grIII}
By the monotonicity of $B\mapsto \cK(B;x,y)$, the defining equation 
\eq{IIIB} for $B=B(x,y)$, and the definition \eq{M} of the function $M$, we have
\[B\gtrless 1\quad \Leftrightarrow \quad 0\gtrless \cK(1;x,y)=M(x)K(x,y)\,,\]
where
\beq\label{K}
	K(x,y):=\ba\psi\big({\textstyle\frac{1}{xy}}\big)M(y)
	+\fa\psi\big({\textstyle\frac{1}{x}}\big)-\ba\psi(y)\,.
\eeq
We are not able in general to solve the equation $K(x,y)=0$ explicitly for $x$ or $y$. 
Instead we determine the properties of the zero-level of $K$ by estimating the
partials $\del_xK$, $\del_yK$ in Proposition \ref{III_K_partials}. We also determine 
the relative locations and intersections of $\{K(x,y)=0\}$ with the hyperbola $xy=1$, 
see Proposition \ref{III_K_prop2}. 
This information will be used in Section \ref{group3_reflected} and depends on $\gamma$.
\begin{proposition}\label{III_K_partials}
	The partials of the function $K(x,y)$ defined in \eq{K} satisfy
	\beq\label{K_partials}
		\del_x K(x,y)<0\,,\qquad\del_y K(x,y)<0\qquad\quad \mbox{for all $x,y>0$}\,.
	\eeq
\end{proposition}
\bp
	As $\ba\psi$, $\fa\psi$ are strictly increasing we have
	\beq\label{Kx}
		\del_x K(x,y)=-{\textstyle\frac{1}{x^2y}}\ba\psi {'}\big({\textstyle\frac{1}{xy}}\big)M(y)
		-{\textstyle\frac{1}{x^2}}\fa\psi {'}\big({\textstyle\frac{1}{x}}\big)<0\,.
	\eeq
	On the other hand
	\beq\label{Ky}
		\del_y K(x,y)=-{\textstyle\frac{1}{xy^2}}\ba\psi {'}\big({\textstyle\frac{1}{xy}}\big)M(y)
		+\ba\psi\big({\textstyle\frac{1}{xy}}\big)M'(y)-\ba\psi {'}(y)\,,
	\eeq
	such that $\del_y K(x,y)<0$ if and only if
	\beq\label{K_partials_2}
		{\textstyle\frac{1}{xy}}\ba\psi {'}\big({\textstyle\frac{1}{xy}}\big)
		-\ba\psi\big({\textstyle\frac{1}{xy}}\big)m(y)+\frac{y\ba\psi {'}(y)}{M(y)}>0\,,
	\eeq
	where the function $m$ is analyzed in Section \ref{aux}. 
	By the properties of the auxiliary functions the last inequality is trivially satisfied
	whenever $xy\geq 1$. For $xy<1$ we use instead that $\ba\psi {'}>0$ such that \eq{K_partials_2}
	follows provided 
	\[{\textstyle\frac{1}{xy}}\ba\psi {'}\big({\textstyle\frac{1}{xy}}\big)
	>\ba\psi\big({\textstyle\frac{1}{xy}}\big)m(y)\,.\]
	For $xy<1$ the latter inequality is equivalent to 
	\[\ell\big({\textstyle\frac{1}{xy}}\big) >m(y)\,,\]
	where the auxiliary function $\ell$ is analyzed in Section \ref{aux}. In particular, $\ell$ 
	satisfies $\ell(q)>\frac{1}{2}$ for $q>1$, while $m(y)<\frac{1}{2}$ for all $y>0$. This shows 
	that $\del_y K(x,y)<0$ for all $x,\,y>0$.
\ep
It follows that $\{B=1\}=\{K(x,y)=0\}$ is given by a graph $y=k(x)$, for a $C^2$-smooth and strictly 
decreasing function $k(x)$. 
A calculation shows that:
\[\lim_{y\downarrow 0} K(x,y) =\infty,\,\, \lim_{y\uparrow \infty} K(x,y) =-\infty\quad\mbox{for $x>0$},
\quad \lim_{x\downarrow 0} K(x,y) =\infty\quad\mbox{for $y>0$},\]
and
\[\lim_{x\uparrow \infty} K(x,y) =K_\infty(y):= -\nu M(y)-\ba\psi(y)-\frac{\kappa}{\sqrt{a}}\quad\mbox{for $y>0$.}\]
The function $K_\infty(y)$ is a strictly decreasing function and satisfies
\[K_\infty(0)=\nu-\frac{\kappa}{\sqrt{a}}>0\,,\qquad \lim_{y\uparrow\infty} K_\infty(y)=-\infty\,,\]
and has a unique root $y=\hat y(\gamma)\in(0,1)$ (given in \eq{yhat}). 
This concludes the proof of part (i) of Theorem \ref{sum_grIII}. 

Before considering the outgoing forward wave $F$ we need to analyze the relative 
positions and intersections of the curves $\{xy=1\}$ and $\{B=1\}=\{K(x,y)=0\}$. These depend 
on $\gamma$ and are given by the properties of the function $\alpha:=\fa\psi-\ba\psi$ which 
is analyzed in Section \ref{alfa}. 
\begin{proposition}\label{III_K_prop2}
	Let $\{y=k(x)\}=\{B=1\}$ be as in part (i) of Theorem \ref{sum_grIII}. Let $y_0=y_0(\gamma)$ denote the 
	root different from $1$ of the function $\alpha(y,a)$ when $\gamma\neq\frac{5}{3}$, and set $y_0=1$
	when $\gamma=\frac{5}{3}$. Let $x_0=\frac{1}{y_0}$.
	The relative positions and intersections of the curves $\{xy=1\}$ and $\{y=k(y)\}$ are then given as follows:
	\begin{itemize}
		\item[(a)] For $1<\gamma<\frac{5}{3}$ (see Figure 8): $y_0<1<x_0$ and 
			\begin{itemize}
				\item[$\bullet$] the curves intersect at $(1,1)$ tangentially and at $(x_0,y_0)$ transversally, 
				and only at these points
				\item[$\bullet$]  $\frac{1}{x}<k(x)$ for $x\in (x_0,\infty)$
				\item[$\bullet$]  $\frac{1}{x}>k(x)$ for $x\in (0,1)\cup (1,x_0)$
			\end{itemize}
		\item [(b)] For $\gamma=\frac{5}{3}$ (see Figure 9): $y_0=x_0=1$ and
			\begin{itemize}
				\item[$\bullet$]  the curves intersect only at $(1,1)$ (tangentially)
				\item[$\bullet$]  $\frac{1}{x}<k(x)$ for $x\in (1,\infty)$
				\item[$\bullet$]  $\frac{1}{x}>k(x)$ for $x\in (0,1)$
			\end{itemize}
		\item[(c)] For $\gamma>\frac{5}{3}$ (see Figure 10): $x_0<1<y_0$ and
			\begin{itemize}
				\item[$\bullet$]  the curves intersect at $(1,1)$ tangentially and at $(x_0,y_0)$ transversally, 
				and only at these points
				\item[$\bullet$]  $\frac{1}{x}<k(x)$ for $x\in (x_0,1)\cup (1,\infty)$
				\item[$\bullet$]  $\frac{1}{x}>k(x)$ for $x\in (0,x_0)$.
			\end{itemize}
	\end{itemize}
\end{proposition}
\begin{proof}
	From \eq{K_partials} we have that $K(x,y)\gtrless 0$ for $y\lessgtr k(x)$, while
	\beq\label{K_xy=1}
		K\big({\textstyle\frac{1}{y}},y\big)=\alpha(y,a)\,.
	\eeq
	The statements about the location of the intersection points and the relative positions of $\{xy=1\}$ and 
	$\{y=k(x)\}$, follow immediately from this and the properties of the function $\alpha(y,a)$
	as detailed in Section \ref{alfa} (see Figures 8, 9, 10). 
	Calculating $k'(x)=-\frac{\del_x K(x,y)}{\del_y K(x,y)}$ at the points of 
	intersection shows that $k'(1)=-1$ and that $k'(x_0)\gtrless -\frac{1}{x_0^2}$ if and only if 
	$\alpha'(y_0)\lessgtr 0$. The tangency and transversality claims follow from this
	and the properties of $\alpha$.
\end{proof}
\begin{figure}\label{IIIb_vacuum}
	\centering
	\includegraphics[width=7.7cm,height=5.7cm]{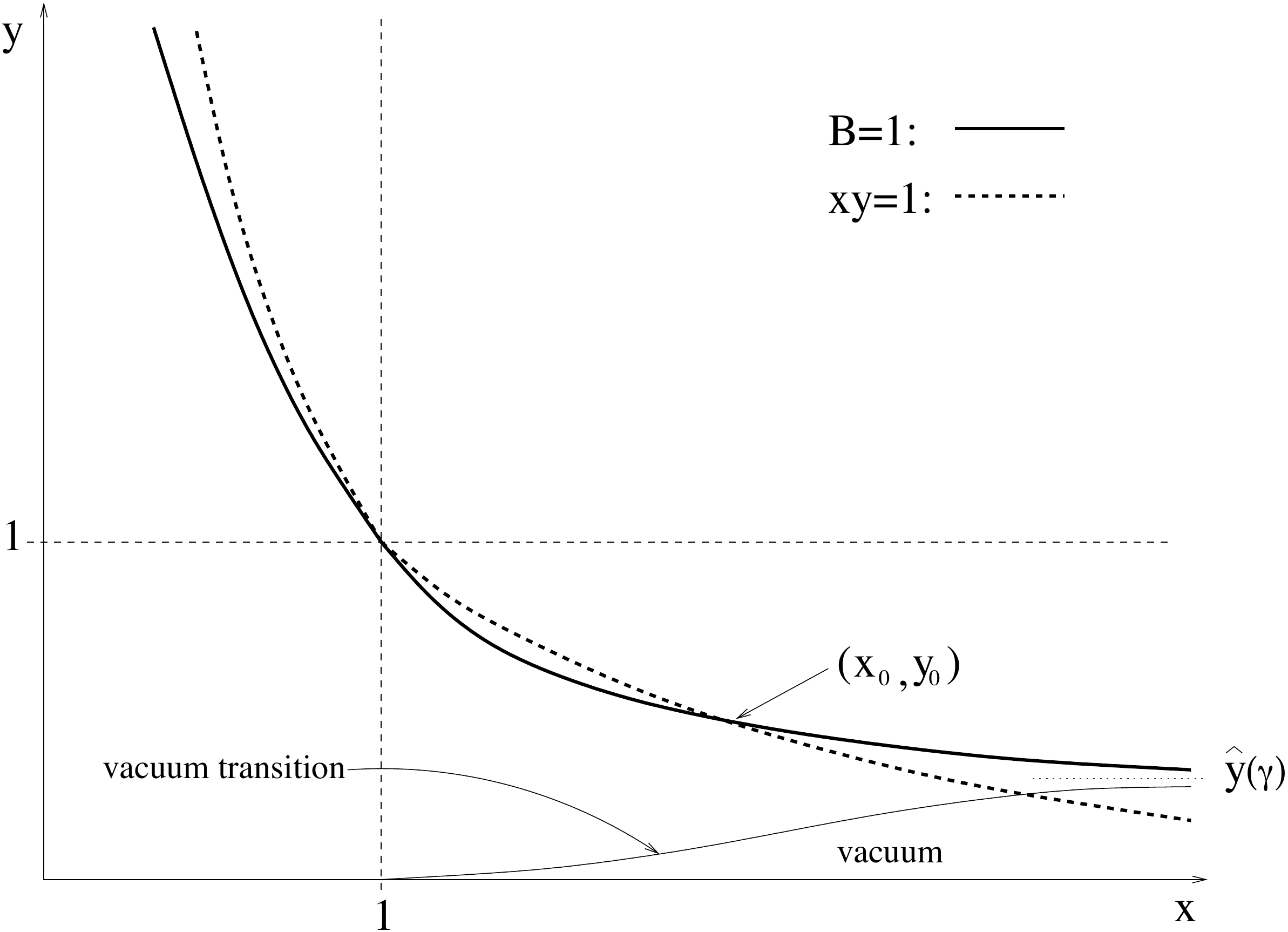}
	\caption{The sets $\{B=1\}$, $\{xy=1\}$, and their intersections for $\gamma<\frac{5}{3}$ (schematic).}
\end{figure}
%
\begin{figure}
	\centering
	\includegraphics[width=7.7cm,height=5.7cm]{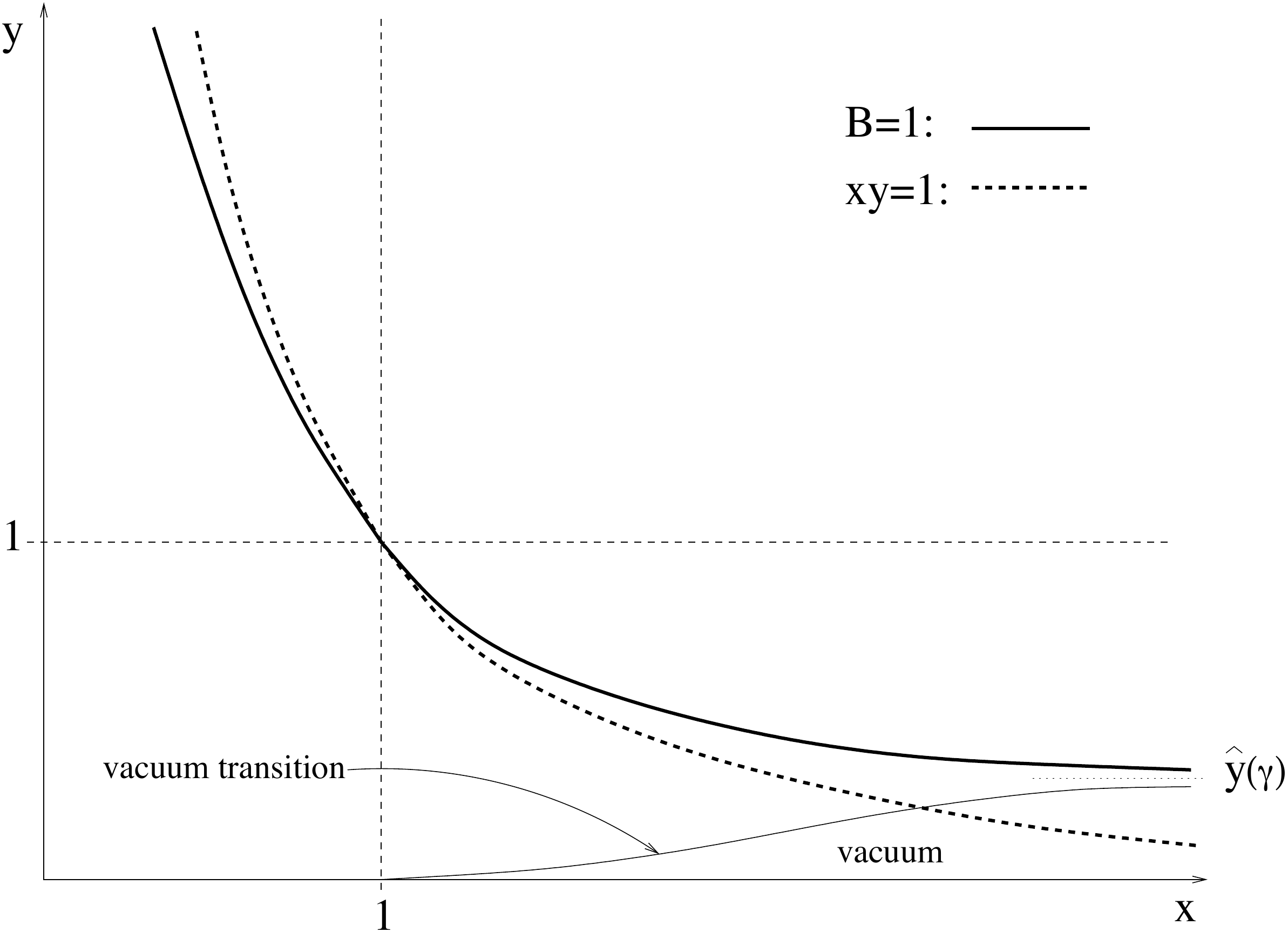}
	\caption{The sets $\{B=1\}$, $\{xy=1\}$, and their intersection for $\gamma=\frac{5}{3}$ (schematic).}
\end{figure}
%
\begin{figure}
	\centering
	\includegraphics[width=7.7cm,height=5.7cm]{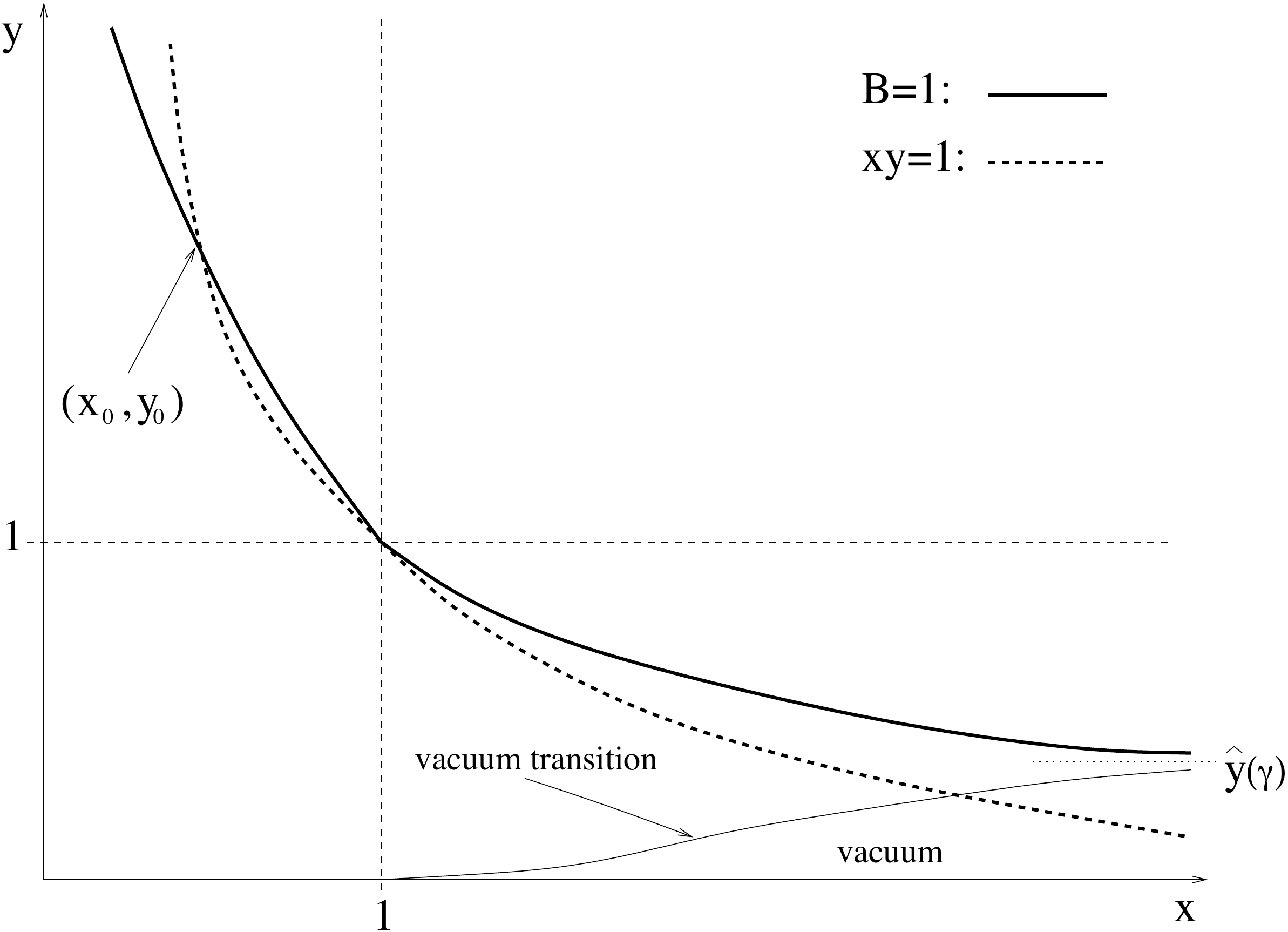}
	\caption{The sets $\{B=1\}$, $\{xy=1\}$, and their intersections for $\gamma>\frac{5}{3}$ (schematic).}
\end{figure}

\medskip

\section{Proof of Theorem \ref{sum_grIII} part $\mathrm{(ii)}$}\label{group3_reflected}
We now consider the reflected wave in overtaking interactions.
To determine the type of the outgoing forward wave $F$ we use \eq{IIIint3} together with
the defining equation \eq{IIIB} for $B$:
\beq\label{F_crit}
	F\gtrless 1\quad \Leftrightarrow\quad  xy\gtrless B \quad \Leftrightarrow\quad  
	H(x,y)\gtrless \cK(B;x,y)=0\,,
\eeq
where we have introduced the function
\beq\label{HH}
	H(x,y):=\cK(xy;x,y)=\ba\psi(xy)-\ba\psi(x)-\ba\psi(y)M(x)\,.
\eeq
We note that
\[H(1,y)=H(x,1)=H(0,y)\equiv 0\qquad \forall \,\, x,y>0\,.\]
To keep the lengths of the proofs to a reasonable length we have collected some 
parts of the analysis of $H$ in Sections \ref{H_in_IIIa}-\ref{eta}. 
We consider separately the cases $\gamma\in(1,\frac{5}{3})$,
$\gamma=\frac{5}{3}$, and $\gamma>\frac{5}{3}$.

\subsection{The reflected wave $F$ in the case $1<\gamma<\frac{5}{3}$}
The following proposition details the properties of the reflected wave when the adiabatic constant 
is between $1$ and $\frac{5}{3}$. We recall that $x_0:=\frac{1}{y_0}$, where $0<y_0<1$ is the 
unique root different from $1$ of the function $\alpha=\fa\psi-\ba\psi$ defined in Section \ref{alfa}.
\begin{proposition}\label{<5/3}
	Consider the interactions of two overtaking backward waves listed in \eq{gr3}.
	Let the left and right incoming waves have strengths $x$ and $y$, respectively.	For 
	$1<\gamma<\frac{5}{3}$ ($0<a<\frac{1}{4}$) the outgoing reflected 
	wave $F=F(x,y)$ is given as follows.
	\begin{itemize}
		\item[(a)] $\ba S\ba S$-interactions ($\, \mathrm{IIIa}$, $x,\, y>1$) yield $F>1$: the
		reflected wave is a rarefaction.
		\item[(b)] $\ba S\ba R$-interactions ($\, \mathrm{IIIb}$, $0<y<1<x$) may yield either 
		type of reflected wave. More precisely, in the region $0<y<1<x$, the set 
		$\{H(x,y)=0\}\equiv\{F=1\}$ coincides with a graph $y=h(x)$, and the 
		reflected wave $F$ is a
		\begin{itemize}
			\item[(b1)] rarefaction (i.e.\ $F>1$) if and only if $y<h(x)$
			\item[(b2)] shock (i.e.\ $F<1$) if and only if $y>h(x)$.
		\end{itemize}
		The location of the graph $y=h(x)$ is given as follows. 
		Let $y=k(x)$ be the graph 
		along which $B=1$ (defined in Section \ref{part1_grIII}). Then, in the region 
		$0<y<1<x$:
		\begin{itemize}
			\item[(b3)] the three graphs $y=h(x)$, $y=k(x)$, and $y=\frac{1}{x}$ 
			all pass through $(x_0,y_0)$,
			\item[(b4)] $h(x)<k(x)<\frac{1}{x}$ for $1<x<x_0$
			\item[(b5)] $h(x)>k(x)>\frac{1}{x}$ for $x>x_0$.
			\item[(b6)] $y=h(x)$ has horizontal asymptote $y_*(\gamma)\in(\hat y(\gamma),1)$ as $x\uparrow +\infty$, and
			\beq\label{h_lim}
				\lim_{x\downarrow 1} h(x)=y_1(\gamma):=\textstyle\left[\frac{3}{4(1-a)}\right]^\frac{1}{\zeta}<1\,.
			\eeq
		\end{itemize}
		\item[(c)] $\ba R\ba S$ interactions ($\, \mathrm{IIIc}$, $0<x<1<y$) yield $F<1$: the reflected wave is a shock.
	\end{itemize}
	The situation is summarized in Figure 5 (left diagram) and in Figure 11.
\end{proposition}
\begin{figure}
	\centering
	\includegraphics[width=7.7cm,height=5.7cm]{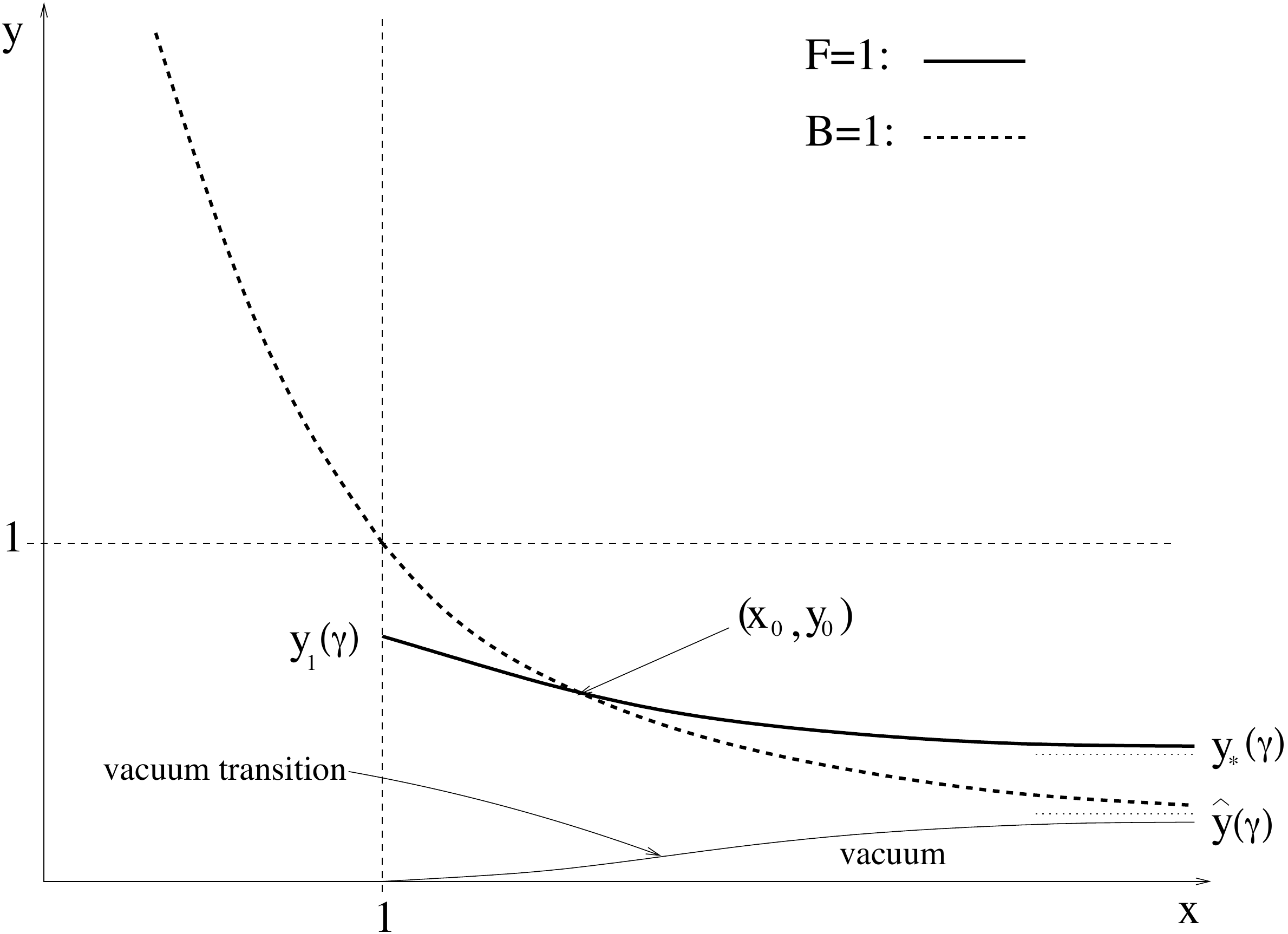}
	\caption{The sets $\{F=1\}$, $\{B=1\}$, and their intersection for $\gamma<\frac{5}{3}$ (schematic).}
\end{figure}
\begin{proof} We consider each region in turn:
	\begin{itemize}
		\item[(a)] This part follows directly from \eq{F_crit} and \eq{H_crit}.
		\item[(b)] For this part we split the argument into several steps.
		\begin{itemize} 
		\item[1.] We first observe that by definition \eq{M} of the auxiliary function $M$, we have
		\beq\label{H_xy=1}
			H\big({\textstyle\frac{1}{y}},y\big)=M\big(\textstyle\frac{1}{y}\big)\alpha(y,a)\,.
		\eeq
		Together with \eq{IIIint3}, \eq{K_xy=1}, and the properties of the map $\alpha(\cdot,a)$ (see Figure 15),
		this shows: in the region IIIb ($0<y<1<x$), and for $1<\gamma<\frac{5}{3}$, 
		the two curves $\{xy=1\}$ and $\{B=1\}\equiv\{y=k(x)\}$, and the set 
		$\{F=1\}\equiv\{H(x,y)=0\}$ (yet to be shown to be a curve), intersect at, and only 
		at, the point $(x_0,y_0)$.

		\item[2.] Next we consider the equation $H(x,y)=0$ in the lower sub-region 
		\[\text{IIIb1}:=\{0<y<xy<1<x\}\,.\] 
		Using the explicit expressions for the auxiliary functions $\ba\psi$ and $M$ we
		may solve explicitly for $y$ as a function of $x$ and obtain that:
		in IIIb1 the set $\{H(x,y)=0\}$ coincides with the graph $y=h_1(x)$ where
		\beq\label{h_1}
			\qquad\qquad\quad\,\,\,\,
			h_1(x)=\left[\frac{M(x)-\frac{1}{\nu}\ba\psi(x)-1}{M(x)-x^\zeta}\right]^\frac{1}{\zeta}
			=\left[\frac{\sqrt{x+ax^2}-\sqrt{x+a}-\frac{\kappa}{\nu}(x-1)}
			{\sqrt{x+ax^2}-x^\zeta\sqrt{x+a}}\right]^\frac{1}{\zeta}\!\!.\!\!  
		\eeq
		Note that the numerator and denominator are both positive by Lemma \ref{extra1}.
		A calculation shows that $h_1(x)<\frac{1}{x}$ if and only if $\alpha(\frac{1}{x},a)<0$,
		which, according to the analysis in Section \ref{alfa}, holds if and only if $x<x_0=\frac{1}{y_0}$ 
		(since $\gamma<\frac{5}{3}$). A direct evaluation shows that
		\[\lim_{x\downarrow 1} h_1(x)=\left[{\textstyle\frac{3}{4(1-a)}}\right]^\frac{1}{\zeta}=:y_1(\gamma)\,,\]
		which is $<1$ since $a<\frac{1}{4}$. This verifies \eq{h_lim}. 
		We conclude that $\{F=1\}\cap \text{IIIb1}$ is given as the graph $\{(x,h_1(x))\,|\, 1<x<x_0\}$. 
		Furthermore, it follows from \eq{M1} and \eq{M2} in Lemma \ref{extra1} that
		\[\{F\gtrless 1\}\cap \text{IIIb1} = \{y\lessgtr h_1(x),\, 1<x<x_0\}\,.\]
		By the intersection properties verified in step 1 above, and the fact that  
		\[\lim_{x\downarrow 1}k(x)=1\,,\] 
		it follows that the graph $\{y=h_1(x)\,|\, 1<x<x_0\}$ lies strictly below the curve segment
		$\{y=k(x)\, |\, 1<x<x_0\}$.
			
		\item[3.] The analysis of the equation $H(x,y)=0$ in the upper sub-region 
		\[\text{IIIb2}:=\{0<y<1<xy\}\]
		is more involved. In particular it is not possible to solve for $y$ explicitly in terms of $x$
		or vice versa. Instead we will make use of the analysis in Section \ref{del_H}. First, \eq{1st}
		shows that the map $y\mapsto H(x,y)$ has no root in $(\textstyle\frac{1}{x},1)$, whenever 
		$x\leq x_0$ and $1<\gamma<\frac{5}{3}$. We infer that $\{F=1\}\cap \text{IIIb2}$ lies 
		in the region $x>x_0$. We next locate the zero-level of $H$ more precisely. 
		By \eq{HH} we have $H(x,y)=0$ if and only if 
		\[\ba\psi(y)=\frac{\ba\psi(xy)-\ba\psi(x)}{M(x)}\,.\]
		For $(x,y)\in\{F=1\}\cap\,$IIIb2 we thus get by \eq{M} (see \eq{K} for the definition of $K$)
		\[\qquad \qquad K(x,y)=\ba\psi\big(\textstyle\frac{1}{xy}\big)M(y)+\fa\psi\big(\frac{1}{x}\big)
		-\frac{\ba\psi(xy)-\ba\psi(x)}{M(x)} =\ba\psi\big(\textstyle\frac{1}{xy}\big)M(y)
		-\frac{\ba\psi(xy)}{M(x)} <0\,.\]
		That is, in sub-region IIIb2 the zero level of $H$ (yet to be shown to be a graph 
		$y=h_2(x)$) lies above $y=k(x)$, and therefore (according to part (a) of Proposition 
		\ref {III_K_prop2}) also above $y=\frac{1}{x}$.  
		Finally, since $1<\gamma<\frac{5}{3}$, \eq{1st} and \eq{2nd} show that  
		$\{F(x,y)=1\} \cap \{x>x_0,\, k(x)<y<1\}$ is a graph of a $C^2$-function $h_2(x)$. 
		The graph $y=h(x)$ in the statement of 
		Proposition \ref{<5/3} is then the concatenation of $y=h_1(x)$ for $x<x_0$
		with $y=h_2(x)$ for $x>x_0$.
		
		\item[4.] We finally note that the inequalities between $k(x)$ and $\frac{1}{x}$ 
		were established in Part (a) of Proposition \ref{III_K_prop2}. 
		This concludes the proof of part (b).
	\end{itemize}
	\item[(c)] By \eq{F_crit} the reflected wave satisfies $F<1$ if and only if $H(x,y)<0$.
	When $\gamma<\frac{5}{3}$ we have $a<\frac{1}{4}$, and Lemma \ref{IIIc_H_a<1/4}
	shows that $H(x,y)<0$ in IIIc.
	\end{itemize}	
\end{proof}

\begin{remark}
	While numerical plots indicate that the function $y=h(x)$ is monotone decreasing,
	we have not been able to prove this.
\end{remark}

\medskip

\subsection{The reflected wave $F$ in the case $\gamma=\frac{5}{3}$}
This is a limiting case; for clarity of exposition we treat it separately. The following proposition 
details the properties of the reflected wave in the particular case when the adiabatic
constant is that of a monatomic gas, $\gamma=\frac{5}{3}$.
\begin{proposition}\label{=5/3}
	Consider the interactions of two overtaking backward waves listed in \eq{gr3}.
	Let the left and right incoming waves have strengths $x$ and $y$, respectively.	For 
	$\gamma=\frac{5}{3}$ ($a=\frac{1}{4}$) the outgoing reflected 
	wave $F=F(x,y)$ is given as follows.
	\begin{itemize}
		\item[(a)] $\ba S\ba S$-interactions ($\, \mathrm{IIIa}$, $x,\, y>1$) yield $F>1$: 
		the reflected wave is a rarefaction.
		\item[(b)] $\ba S\ba R$-interactions ($\, \mathrm{IIIb}$, $0<y<1<x$) may yield either 
		type of reflected wave. More precisely, in the region $0<y<1<x$, the set 
		$\{H(x,y)=0\}\equiv\{F=1\}$ coincides with a graph $y=h(x)$, and the 
		reflected wave $F$ is a
		\begin{itemize}
			\item[(b1)] rarefaction (i.e.\ $F>1$) if and only if $y<h(x)$
			\item[(b2)] shock (i.e.\ $F<1$) if and only if $y>h(x)$.
		\end{itemize}
		The location of the graph $y=h(x)$ is given as follows. 
		Let $y=k(x)$ be the graph along which $B=1$ (defined in Section \ref{part1_grIII}). Then, in the region 
		$0<y\leq1\leq x$:
		\begin{itemize}
			\item[(b3)] the three graphs $y=h(x)$, $y=k(x)$, and $y=\frac{1}{x}$ 
			all pass through $(1,1)$, and
			\item[(b4)] $h(x)>k(x)>\frac{1}{x}$ for $x>1$.
		\end{itemize}
		\item[(c)] $\ba R\ba S$ interactions ($\, \mathrm{IIIc}$, $0<x<1<y$) yield $F<1$: 
		the reflected wave is a shock.
	\end{itemize}
	The situation is summarized in Figure 5 (right diagram) and in Figure 12.
\end{proposition}
\begin{proof} We consider each region in turn:
	\begin{itemize}
		\item[(a)] This part follows directly from  \eq{F_crit} and \eq{H_crit}.
		\item[(b)] The proof of this part follows the proof for part (b) of Proposition \ref{<5/3}.
		We split the argument into similar steps.
		\begin{itemize} 
		\item[1.] From \eq{H_xy=1}, \eq{K_xy=1}, and the properties of the map $\alpha(\cdot;\frac{1}{4})$
		(see Figure 16) it follows that: in region IIIb ($0<y<1<x$) when $\gamma=\frac{5}{3}$, 
		the two curves $\{xy=1\}$ and $\{B=1\}\equiv\{y=k(x)\}$, and the set 
		$\{F=1\}\equiv\{H(x,y)=0\}$ (yet to be shown to be a curve), do not intersect.

		\item[2.] Next we will show that the equation $H(x,y)=0$ has no solution in the lower sub-region 
		\[\text{IIIb1}:=\{0<y<xy<1<x\}\] 
		whenever $a\geq \frac{1}{4}$. (We will use this again in the proof of Proposition \ref{=>5/3}). 
		As in the proof of Proposition \ref{<5/3} we get that if 
		$(x,y)\in\text{IIIb1}$ satisfies $H(x,y)=0$, then $y^\zeta=h_1(x)$ where $h_1$ is 
		given by \eq{h_1} (with $a=\frac{1}{4}$). In particular, since $y<\frac{1}{x}$ in IIIb1, 
		we would have that $h_1(x)<\frac{1}{x}$, which is equivalent to 
		\[\alpha\big(\textstyle\frac{1}{x}\big)<0\,.\]
		For $a\geq \frac{1}{4}$, or equivalently, $\gamma\geq \frac{5}{3}$, this implies $x<1$
		(see Figures 16 and 17), which contradicts the assumption that $(x,y)\in\text{IIIb1}$. 	
		This shows that the set $\{F=1\}$ does not meet the sub-region IIIb1 when $\gamma\geq \frac{5}{3}$.
			
		\item[3.] Next consider the equation $H(x,y)=0$ in the upper sub-region 
		\[\text{IIIb2}:=\{0<y<1<xy\}\,.\]
		As in the proof of Proposition \ref{<5/3} we use the results of Section \ref{del_H}. For $a=\frac{1}{4}$ it 
		is shown there that the map $y\mapsto H(x,y)$ has exactly one root in $(\frac{1}{x},1)$ for 
		each fixed $x>1$ (recall $x_0=1$ in the present case). It follows that the set 
		$\{F(x,y)=1\} \cap \{x>1,\, k(x)<y<1\}$ consists of a $C^2$-smooth graph $\{y=h_2(x)\}$. 		
		
		\item[4.] We finally note that the inequalities between $k(x)$ and $\frac{1}{x}$ were established 
		in Part (b) of Proposition \ref{III_K_prop2}. This concludes the proof of part (b).
	\end{itemize}
	\item[(c)] As in the proof of Proposition \ref{<5/3} this part follows from Lemma \ref{IIIc_H_a<1/4}. 
	\end{itemize}	
\end{proof}

\begin{figure}\label{BandF=1_gam53}
	\centering
	\includegraphics[width=7.7cm,height=5.7cm]{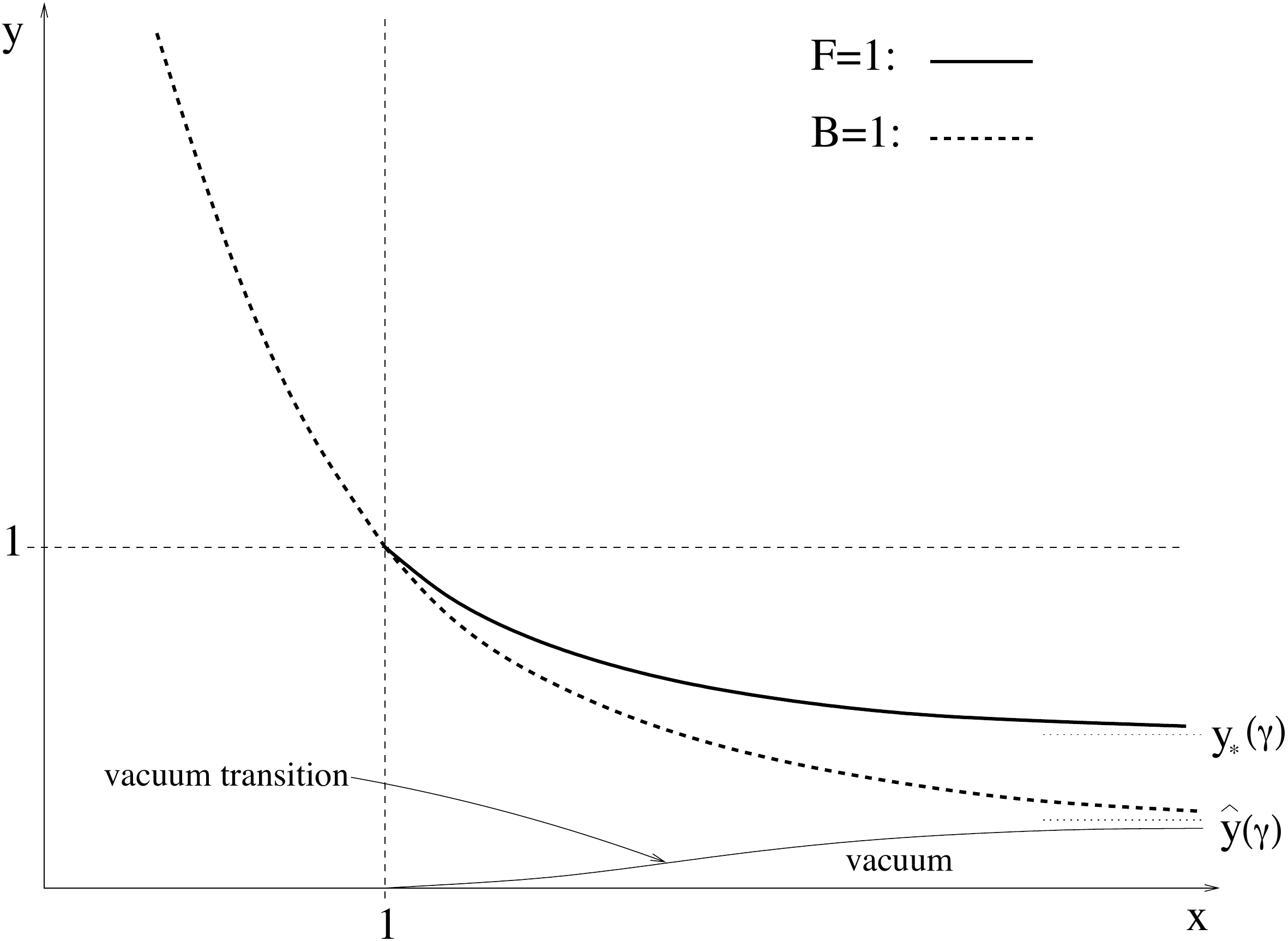}
	\caption{The sets $\{F=1\}$, $\{B=1\}$, and their intersection for $\gamma=\frac{5}{3}$ (schematic).}
\end{figure}

\medskip

\subsection{The reflected wave $F$ in the case $\gamma>\frac{5}{3}$}
This case is the most complicated one to analyze. In particular, now the set
$\{F=1\}\equiv \{H=0\}$ either meets all three of IIIa, IIIb, and IIIc, or just IIIa and IIIc.  
More precisely we shall see that $\{H(x,y)=0\}$ is a curve in the $(x,y)$ plane which always 
meets IIIa and IIIc, and which meets IIIb if and only if $\frac{5}{3}<\gamma<2$. See Figures
13 and 14. 
\begin{proposition}\label{=>5/3}
	Consider the interactions of two overtaking backward waves listed in \eq{gr3}.
	Let the left and right incoming waves have strengths $x$ and $y$, respectively.	For 
	$\gamma>\frac{5}{3}$ ($a>\frac{1}{4}$) the outgoing reflected 
	wave $F=F(x,y)$ is given as follows.
	\begin{itemize}
		\item[(a)] $\ba S\ba S$-interactions ($\, \mathrm{IIIa}$, $x,\, y>1$) may yield either type of 
		reflected wave. 
		More precisely, in the region $x,\, y>1$, the set $\{F=1\}$ coincides 
		with the graph of the strictly decreasing function
		\beq\label{j}
			j(x)=\frac{2a}{(1-a)^2x}\left[a(1+x)+\sqrt{a^2(1+x)^2+ax(1-a)^2}\right]\,.
		\eeq
		Set
		\beq\label{xy_bar}
			\bar x:=\frac{4a^2}{1-3a}\qquad\qquad \bar y:=\frac{2a\sqrt{a}}{(1-\sqrt{a})^2}
			\qquad\qquad y_*:=\frac{4a^2}{(1-a)^2}.
		\eeq
		Then the graph intersects the line $x=1$ at $(1,\bar y)$. 
		For $\frac{5}{3}<\gamma<2$ ($\frac{1}{4}<a<\frac{1}{3}$) the graph  
		intersects the line $y=1$ at $(\bar x,1)$, while for 
		$\gamma\geq 2$ ($a\geq \frac{1}{3}$) it has the horizontal asymptote $y=y_*>1$
		as $x\uparrow \infty$. The reflected wave is a: 
		\begin{itemize}
			\item[(a1)] rarefaction ($F>1$) if and only if $y>\max(j(x),1)$
			\item[(a2)] shock ($F<1$) if and only if $1<y<j(x)$.
		\end{itemize}
		\item[(b)] $\ba S\ba R$-interactions ($\, \mathrm{IIIb}$, $0<y<1<x$) may or may not yield either type of 
		reflected wave.
		First, for $\gamma\geq 2$ ($a\geq \frac{1}{3}$) the set $\{F=1\}$
		does not meet $\mathrm{IIIb}$ and the reflected wave is necessarily a rarefaction ($F>1$).

		On the other hand, for $\frac{5}{3}<\gamma< 2$ ($\frac{1}{4}<a< \frac{1}{3}$) the set $\{F=1\}$
		meets $\mathrm{IIIb}$ along a graph $y=h(x)$, defined for $x>\bar x$, with $h(\bar x)=1$.
		The reflected wave $F$ is a:
		\begin{itemize}
			\item[(b1)] rarefaction ($F>1$) if and only if  $y<\min(h(x),1)$
			\item[(b2)] shock ($F<1$) if and only if $x>\bar x$ and $h(x)<y<1$.
		\end{itemize}
		The graph $y=h(x)$ in the region $0<y<1<x$ lies above the graph $y=k(x)$ 
		along which $B=1$ (defined in Section \ref{part1_grIII}). 
		\item[(c)] $\ba R\ba S$-interactions ($\, \mathrm{IIIc}$, $0<x<1<y$) may 
		yield either type of reflected wave. More precisely, in $\mathrm{IIIc}$ the set $\{F=1\}$ 
		consists of a graph $y=i(x)$, $0<x<1$, with the properties:
		\begin{itemize}
			\item[(c1)] a rarefaction results ($F>1$) if and only if $y<i(x)$
			\item[(c2)] a shock results ($F<1$) if and only if $y>i(x)$
			\item[(c3)] $i(x)\equiv y_0$ (the unique zero of $\alpha(y,a)$) for $0<x<\frac{1}{y_0}$
			\item[(c4)] $\frac{1}{x}<i(x)<\min\big\{\frac{\bar y}{x},y_0\big\}$ for $\frac{1}{y_0}<x<1$.
		\end{itemize}
	\end{itemize}
	The situation is summarized in Figures 6, 13, and 14.
\end{proposition}
\begin{proof} We consider each region in turn:
	\begin{itemize}
		\item[(a)] IIIa: The set $\{F=1\}$ is the zero-level of the function $H(x,y)$ whose behavior on IIIa
		is analyzed in Section \ref{H_in_IIIa}. The expression in \eq{j} is what results from solving \eq{H_pos} 
		(with equality) for $y$ in terms of $x$. A calculation shows that this is a decreasing function of $x>1$, and that 
		it has the intersection and asymptotic properties as described above. The conclusions follows 
		from this together with \eq{F_crit} and \eq{H_pos}.
		\item[(b)] IIIb: As shown in the proof of part (b) of Proposition \ref{=5/3} (step 2) the set $\{F=1\}$ does 
		not meet subregion IIIb1$:=\{0<y<xy<1<x\}$ whenever $a>\frac{1}{4}$. Also, according to the analysis in Section \ref{del_H},
		$\{F=1\}$ meets subregion IIIb2$:=\{0<y<1<xy<x\}$ if and only if $\frac{1}{4}<a< \frac{1}{3}$. The properties in (b1) and (b2) 
		follow from \eq{3rd} and \eq{4th}. The fact that $y=h(x)$ (denoted $h_2(x)$ in
		Section \ref{del_H}) lies above $y=k(x)$ is proved as in part (b) of Proposition \ref{<5/3} (step 3).
		\item[(c)] IIIc: It is convenient to consider separately the two sub-domains
		\[\mathrm{IIIc1}:=\{0<x<xy\leq 1<y\}\qquad\text{and}\qquad \mathrm{IIIc2}:=\{0<x<1<xy<y\}\,.\]
		In IIIc1 we use the definition of $H$ and the explicit expressions for the auxiliary functions 
		$\ba\psi$ and $M$ to find that 
		\[H(x,y)=\ba\psi(xy)-\ba\psi(x)-\ba\psi(y)M(x)=x^\zeta\,\alpha(y,a)\,,\]
		where $\alpha:=\fa\psi-\ba\psi$ has a unique zero $y_0>1$ (since $\gamma>\frac{5}{3}$, 
		see analysis in Section \ref{alfa} and Figure 17). The properties of $\alpha$ then shows that
		the restriction of $H$ to IIIc1 satisfies
		\beq\label{H_sign1}
			H(x,y)\,\,\left\{\begin{array}{ll}
			<0 & \text{for $0<x<\frac{1}{y}$ and $y>y_0$}\\\\
			\equiv 0 & \text{along $(x,y_0)$, $0<x<\frac{1}{y_0}$}\\\\
			>0 & \text{for $0<x<\frac{1}{y}$ and $1<y<y_0$}\,.\\
			\end{array}\right.
		\eeq 
		In sub-region IIIc2 we consider instead how $H(x,y)$ varies as $(x,y)$ moves along 
		hyperbolas $\{xy=\bar C\}$ ($\bar C=$ constant) in the direction of increasing $y$-values. For this we make use 
		of the properties of the function $\eta(y,a)$, which is analyzed in Section \ref{eta}. 
		In the rest of this proof we assume $\bar C>1$. Let's define the directional derivative
		\[\mathfrak d(x,y):=(-x,y)\cdot \nabla_{(x,y)} H(x,y) = {\textstyle\frac{\kappa}{2}} x^\zeta \eta(y,a)\,,\]
		where $\eta(y,a)$ is defined in \eq{eta_def}. We first observe from \eq{eta_prop}-\eq{eta_def} that the leading 
		term in $\mathfrak d(\frac{\bar C}{y},y)$ for $y\gg 1$ is proportional to $-y^{\frac{1}{2}-\zeta}$.
		Since $\zeta<\frac{1}{2}$ this shows that $\mathfrak d(x,y)\big|_{\mathrm{IIIc2}}$ tends to 
		$-\infty$ as $y\uparrow\infty$. 
		
		We next consider the sign of $\mathfrak d(x,y)$ as we ``start out" along $\{xy=\bar C\}$ from 
		$(1,\bar C)$, in the direction of increasing $y$. As detailed in Section \ref{eta}, the sign 
		of $\mathfrak d(x,y)$ coincides with that of 
		\[\frac{\sqrt{a(1+a)}}{1-\sqrt{a}}-\sqrt{y+a}\,.\]
		Thus, if the constant $\bar C$ satisfies
		\[\frac{\sqrt{a(1+a)}}{1-\sqrt{a}}-\sqrt{\bar C+a}\leq 0\quad\Leftrightarrow\quad 
		\bar C\geq \bar y:=\frac{2a\sqrt{a}}{(1-\sqrt{a})^2}\,,\]
		then, since $H(1,\bar C)=0$, $H(x,y)<0$ along $\{xy=\bar C\}$. On the other hand, if $1\leq \bar C<\bar y$,
		then $H$ increases along $\{xy=\bar C\}$ as $y$ increases, until $y=\bar y$, after which it decreases
		to $-\infty$. We see from this that $H$ has:
		\begin{itemize}
			\item no zero along $\{xy=\bar C\}$ when $\bar C> \bar y$,
			\item a unique zero along $\{xy=\bar C\}$ when $1<\bar C\leq \bar y$.
		\end{itemize}
		It follows that the zeros of $H(x,y)$ in IIIc2 lie along a curve $y=i(x)$, where $i$ satisfies
		$i(\frac{1}{y_0})=y_0$ and $i(1)=\bar y$.
		It remains to argue that the graph $y=i(x)$ lies below the line $y=y_0$, and we do this
		by showing that $H(x,y_0)<0$ for $\frac{1}{y_0}<x<1$. Indeed, by using the property 
		$\fa\psi(y_0)=\ba\psi(y_0)$ together with the explicit expressions for the auxiliary functions 
		$\ba\psi$ and $M$, we have that
		\[\del_x H(x,y_0)=-\del_x\big(\alpha(xy_0,a)\big)\qquad\text{for $\frac{1}{y_0}<x<1$.}\]
		Integrating from $\frac{1}{y_0}$ to $x$ we obtain
		\[H(x,y_0)=-\alpha(xy_0,a)<0\,,\]
		where the latter inequality follows from the properties of $\alpha(\cdot,a)$ when 
		$a>\frac{1}{4}$.
	\end{itemize}	
\end{proof}

\begin{figure}\label{BandF=1_large_gamma}
	\centering
	\includegraphics[width=7.7cm,height=5.7cm]{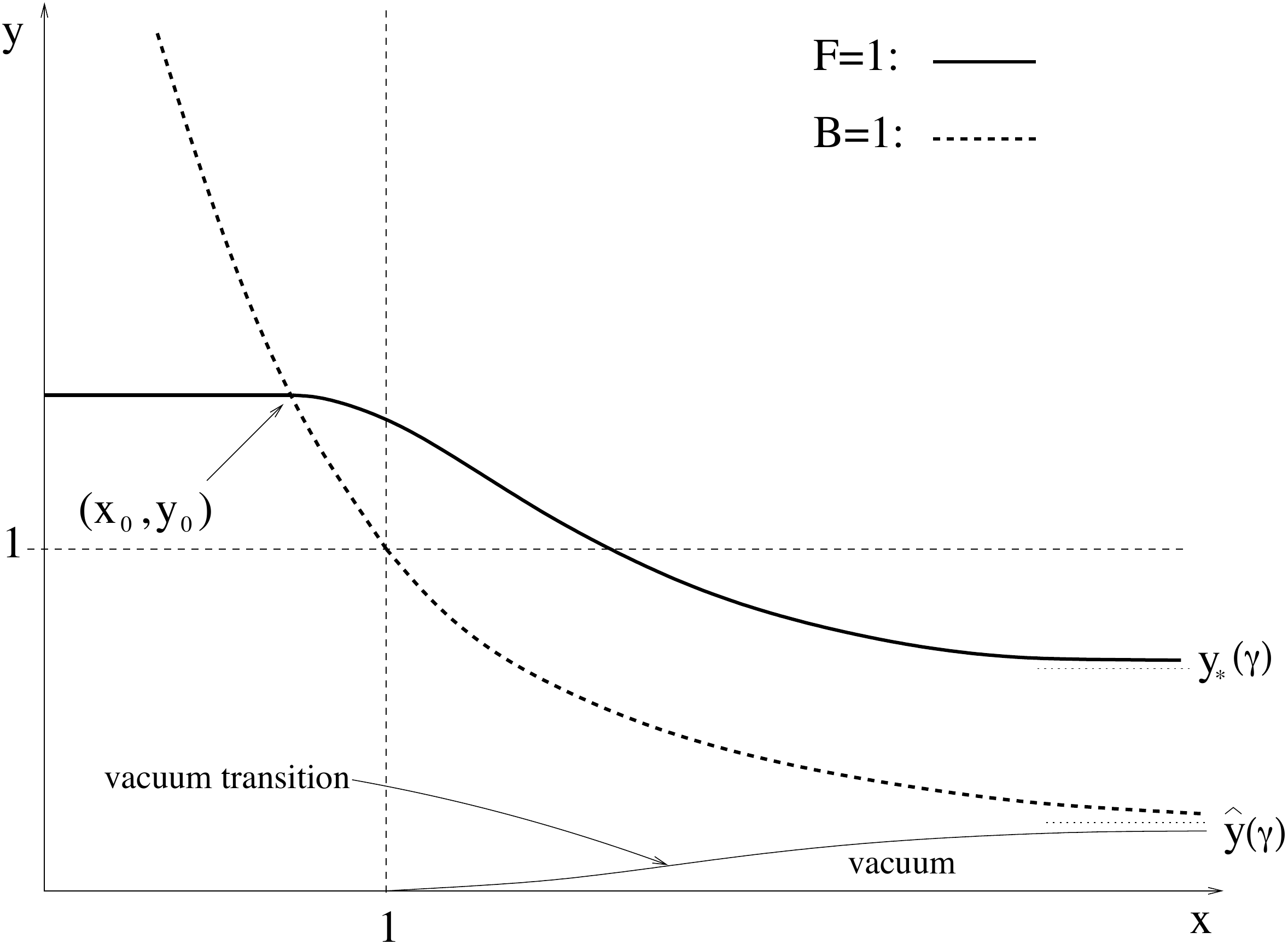}
	\caption{The sets $\{F=1\}$, $\{B=1\}$, and their intersection for $\frac{5}{3}<\gamma<2$ (schematic).}
\end{figure}
%

\begin{figure}
	\centering
	\includegraphics[width=7.7cm,height=5.7cm]{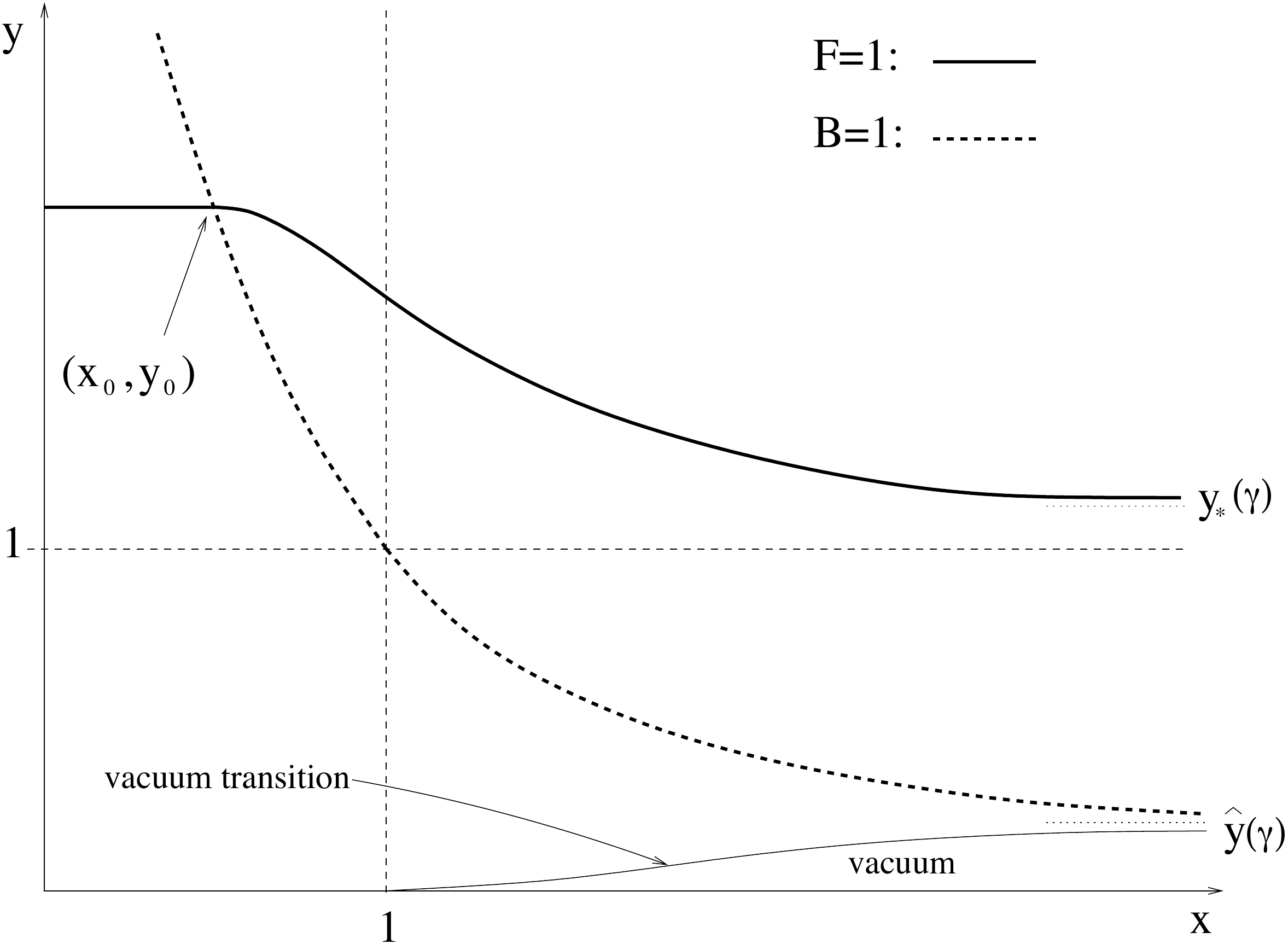}
	\caption{The sets $\{F=1\}$, $\{B=1\}$, and their intersection for $\gamma\geq 2$ (schematic).}
\end{figure}

\medskip

\section{Proof of Theorem \ref{sum_grIII} part $\mathrm{(iii)}$}\label{group3_contact}
To analyze the outgoing contact discontinuity in Group III interactions it is 
advantageous to work in $(\tau, u, S)$-space where we track two quantities, $\tau$ and $S$,
that change across contacts. The wave curves in these variables were 
recorded in Section \ref{bckgr}.

\subsection{Equations for outgoing waves in $(\tau, u, S)$-variables}
We refer to Group III interactions in Figure 1 and denote the specific volumes ratios
$\tau_{right}/\tau_{left}$ across the incoming waves by $\tx$
(leftmost) and $\ty$ (rightmost). We use capital letters $L$, $C$,
$I$ to denote the specific volume ratios across the outgoing backward,
contact, and forward waves, respectively.

We use the expressions for the wave curves in $(\tau, u, S)$-space to 
traverse the waves before and after interaction. From 
\eq{p_tau_u_S}-\eq{frwd_wave_tau_u_s} we obtain the
following equations for the outgoing strengths $L$, $I$, $C$:
\bea
    \ba\xi(\tx)+\ba\xi(\ty)\sqrt{\tx^{1-\gamma} \exp({\ba\eta(\tx)})} &=&
    \ba\xi(L)-\fa\xi(I)\sqrt{C L^{1-\gamma}\exp({\ba\eta(L)})}\label{IIIcon1},\\
    \ba\eta(\tx)+\ba\eta(\ty)  &=& \ba\eta(L)+\gamma\log C+\fa\eta(I)\label{IIIcon2},\\
    \tx\ty&=&C LI\,.\label{IIIcon3}
\eea
As for Groups I and II we consider IIIa, IIIb, and IIIc
interactions separately. 
We find it necessary to make a further breakdown and consider each
combination of outgoing extreme waves within each of IIIa, IIIb, and IIIc. For the 
most part the type of the outgoing contact follows readily from \eq{IIIcon1}-\eq{IIIcon3}
and the properties of the auxiliary functions. The only exception is the case 
$\ba S\ba S \rightarrow \ba S J \fa R$ which requires additional arguments.

\medskip

\subsection{Outgoing contact in IIIa-interactions}\label{C_IIIa}
These are $\ba S\ba S$-interactions, which corresponds to $a<\tx,\, \ty<1$, or, in terms
of incoming pressure ratios $x$, $y>1$. It follows from the analysis in Section \ref{part1_grIII} that, 
independently of the value of $\gamma$, the outgoing backward wave is a shock, i.e.\ $a<L<1$. 
There are therefore only two possibilities for the extreme outgoing waves in this case. (The analysis 
in Section \ref{group3_reflected} shows that both can occur when $\gamma>\frac{5}{3}.$) 
We treat them separately and show that the outgoing contact satisfies $C<1$ in both cases.

\medskip

\paragraph{\bf{Case 1:}} $\ba S J \fa S$. In this case $a<L<1<I<\frac{1}{a}$. 
By \eq{IIIcon2} and \eq{eta_fwd_bckwd} we have that
\beq\label{1_1}
	\Gamma(\tx)\Gamma(\ty)=\Gamma(L)\Gamma(I) C^\gamma\,.
\eeq
Assume for contradiction that $C\geq 1$; then \eq{1_1} together with \eq{Gamma_2} give
\[\Gamma(L)\leq \Gamma(\tx)\Gamma(\ty)\Gamma(I^{-1})\,.\]
Also, if $C\geq 1$ then \eq{IIIcon3} gives $\tx\ty\geq LI>a$. Since $a<\tx$, $\ty$, $\tx\ty<1$, 
Lemma \ref{Gamma_prop} gives $\Gamma(\tx)\Gamma(\ty)<\Gamma(\tx\ty)$,
such that
\[\Gamma(L)< \Gamma(\tx\ty)\Gamma(I^{-1})\,.\]
Again, since $C\geq 1$, \eq{IIIcon3} gives $\tx\ty I^{-1}=CL\geq L>a$, and we have
$a<\tx\ty$, $I^{-1}$, $\tx\ty I^{-1}<1$. Lemma \ref{Gamma_prop} applies and gives
$\Gamma(\tx\ty)\Gamma(I^{-1}) < \Gamma(\tx\ty I^{-1})$, such that
\[\Gamma(L)< \Gamma(\tx\ty I^{-1})\,.\]
As $\Gamma$ is strictly decreasing on $(a,\frac{1}{a})$, we conclude that $L>\frac{\tx\ty}{I}$.
However, by \eq{IIIcon3}, this implies that $C<1$ and we reach a contradiction. Thus $C<1$.

\medskip

\paragraph{\bf{Case 2:}} $\ba S J \fa R$. In this case $a<L<1$ and 
$I\leq1$. As $\tx<1$ and $\fa\xi(I)\geq 0$, \eq{IIIcon1} and \eq{eta_fwd_bckwd} give that 
\beq\label{2_1}
	\ba \xi(\tx)+\ba \xi(\ty)\sqrt{\Gamma(\tx)}
	<\ba \xi(\tx)+\ba \xi(\ty)\sqrt{\tx^{1-\gamma}\Gamma(\tx)}
	\leq \ba\xi(L)\,.
\eeq
Also, in this case \eq{eta_fwd_bckwd} shows that \eq{IIIcon2} reduces to
\[\Gamma(\tx)\Gamma(\ty)=\Gamma(L)C^\gamma\,.\]
Assuming, again for contradiction, that  $C\geq1$ we thus obtain 
\beq\label{IIIconSS_proof_2}
	\Gamma(\tx)\Gamma(\ty)\geq\Gamma(L)\,.
\eeq 
We proceed to show that this leads to a contradiction with \eq{2_1}.
For this we set $\Omega(s):=\sqrt{\Gamma(s)}$ ($s\in(a,1)$) and define the 
function $\Lambda:=\ba\xi(\Omega^{-1})$ which is analyzed in Section \ref{Gamma}.
Now let
\[z_1:=\Omega(\tx)>1,\qquad z_2:=\Omega(\ty)>1,\qquad z_3:=\Omega(L)>1\,,\]
where we have used that $\Gamma$ and hence $\Omega$ are decreasing. 
Hence \eq{IIIconSS_proof_2} reduces to
\beq\label{IIIconSS_proof_4}
	z_1z_2\geq z_3\,,
\eeq
and Lemma \ref{Lambda_prop} gives
\[\Lambda(z_1)+z_1\Lambda(z_2) \geq \Lambda(z_1z_2).\]
Thus, by \eq{IIIconSS_proof_4} and the fact that $\Lambda(z)$ is 
an increasing function, we obtain $\Lambda(z_1)+z_1\Lambda(z_2) \geq\Lambda(z_3)$,
or equivalently:
\beq 
	\ba\xi(\tx)+\sqrt{\Gamma(\tx)}\ba\xi(\ty)
	\geq \ba\xi(L)\,.\label{IIIconSS_proof_contra}
\eeq 
This contradicts \eq{2_1}, and we conclude that $C<1$.

This establishes the first statement in part (iii) of Theorem \ref{sum_grIII}.

\medskip

\subsection{Outgoing contact in IIIb-interactions}
These are $\ba S\ba R$-interactions for which $a<\tx<1<\ty$. There are now four 
possible combinations of outgoing forward and backward waves. (The analysis in Section 
\ref{group3_reflected} shows that they can all occur when $\gamma>\frac{5}{3}$). 
We demonstrate that the outgoing contact discontinuity $C$ always satisfies $C>1$ 
by considering each case separately.

\medskip

\paragraph{\bf Case 1:}  $\ba R {J} \fa S $. In this case $L\geq1$, $1\leq I<\frac{1}{a}$
and \eq{eta_fwd_bckwd} shows that \eq{IIIcon2} reduces to
\[\Gamma(\tx)=\Gamma(I)C^\gamma\,.\] 
Since $a<\tx <1\leq I$, \eq{Gamma_1} shows that $\Gamma(\tx)>\Gamma(I)$, such that $C>1$.

\medskip

\paragraph{\bf Case 2:} $\ba R {J} \fa R $. In this case $L\geq1,\ I\leq1$ and 
\eq{eta_fwd_bckwd} shows that \eq{IIIcon2} reduces to
\[\Gamma(\tx)=C^\gamma.\]
Since $a<\tx <1$, \eq{Gamma_1} shows that $\Gamma(\tx)>1$, such that $C>1$.

\medskip

\paragraph{\bf Case 3:} $\ba S {J} \fa R $. In this case $a<L\leq1,\ I\leq1$ and 
\eq{eta_fwd_bckwd} shows that \eq{IIIcon2} reduces to
\[\Gamma(\tx)=\Gamma(L)C^\gamma.\]
We now argue by contradiction: if $C\leq1$ then 
$\Gamma(\tx)\leq\Gamma(L)$, and \eq{Gamma_1} shows that $L\leq \tx$. 
However, combining this with $C\leq 1<\ty$ and \eq{IIIcon3} gives
\[L<L\ty\leq \tx\ty=CLI\leq LI,\] 
contradicting $I\leq1$. Hence $C>1$.

\medskip

\paragraph{\bf Case 4:} $\ba S {J} \fa S$. In this case $a<L<1<I<\frac{1}{a}$ and
\eq{IIIcon2}, \eq{eta_fwd_bckwd}, and \eq{Gamma_1} give
\beq\label{reln_gam}
	\Gamma(\tx)=\Gamma(L)\Gamma(I)C^\gamma< \Gamma(L)C^\gamma\,. 
\eeq
>From \eq{IIIcon1} and the fact that $\ba\xi(\ty)$, $\fa\xi(I)<0$, 
we obtain that $\ba\xi(\tx)>\ba\xi(L)$. As $\ba\xi$ is strictly decreasing, we infer that 
$\tx<L$. Thus $\Gamma(\tx)>\Gamma(L)$, which combined with \eq{reln_gam} gives 
$C>1$.

\medskip

\subsection{Outgoing contact in IIIc-interactions}
These are $\ba R\ba S$-interactions for which $a<\ty<1<\tx$. As for 
$\ba S\ba R$-interactions there are four possible combinations of extreme outgoing waves. 
(The analysis in Section \ref{group3_reflected} shows that they can all occur when 
$\gamma>\frac{5}{3}$). We claim that the outgoing contact discontinuity $C$ 
always satisfies $C>1$. Considering the same cases as for $\ba S\ba R$-interactions 
it turns out that the arguments for Cases 1, 2, and 3 for $\ba R\ba S$-interactions are 
identical to those for $\ba S\ba R$-interactions, upon interchanging $\tx$ and $\ty$.
We therefore only need to consider Case 4, $\ba S {J} \fa S$, where the outgoing 
strengths satisfy $a<L<1<I<\frac{1}{a}$. As above we use \eq{IIIcon2}, \eq{eta_fwd_bckwd}, 
and \eq{Gamma_1} to obtain
\beq\label{gam_reln2}
	\Gamma(\ty)=\Gamma(L)\Gamma(I)C^\gamma< \Gamma(L)C^\gamma\,.
\eeq
Since $\tx>1$ we have $\ba\xi(\tx)<0$ and $\sqrt{\tx^{1-\gamma}\exp(\ba\eta(\tx))}<1$.
At the same time $\fa\xi(I)<0$, and we get from \eq{IIIcon1} that $\ba\xi(\ty)>\ba\xi(L)$. 
As $\ba\xi$ is strictly decreasing we have $\ty<L$ and \eq{gam_reln2} yields $C>1$.

With this we have established the second statement in part (iii) of Theorem \ref{sum_grIII}.

\medskip

\section{{Proof of Theorem \ref{sum_grIII} part $\mathrm{(iv)}$}}\label{group3_vac}
For this part of the proof we use the pressure ratios $x$ and $y$ of the incoming waves.
We argue as in Section \ref{rip_vac} and observe that the map $B\mapsto \cK(B;x,y)$ 
is strictly increasing (see \eq{cK} for the definition of the function $\cK$). 
The interaction Riemann problem is vacuum-free if and only if
$B=B(x,y)>0$, or equivalently $\cK(B;x,y)=0>\cK(0;x,y)$, i.e.
\beq\label{III_no_vac}
	\cK (0;x,y)=-\nu-\nu M(x)M(y)-\ba\psi(x)-\ba\psi(y)M(x)<0\,.
\eeq
Rearranging the last expression we have that the overtaking-wave interaction produces 
no vacuum if and only if the incoming parameters $x$ and $y$ satisfy
\beq\label{III_no_vac0}
	-\nu M(y)-\ba\psi(y)< \frac{\nu + \ba\psi(x)}{M(x)}\,.
\eeq
Since the function $\ba\psi$ is strictly increasing with $\ba\psi(0)=-\nu$ and $\ba\psi(1)=0$,
it follows that the inequality \eq{III_no_vac0} is satisfied whenever $y>1$. Consequently,
a vacuum is never generated in $\ba S\ba S$ (IIIa) and $\ba R\ba S$ (IIIc) interactions.

On the other hand, depending on the incoming pressure ratios $x$ and $y$, a 
vacuum may or may not emerge from an $\ba S\ba R$ (IIIb) interaction. In this case 
$y<1<x$ and \eq{III_no_vac0} takes the explicit form
\beq\label{no_vac1}
	v(x):=\frac{\nu\sqrt{x+a}+\kappa(x-1)}{\sqrt{x+a x^2}}>\nu(1-2y^\zeta)\,.
\eeq
\begin{lemma}\label{III_vac}
	The function $v(x)$ is strictly decreasing for $x>1$, for all values of $\gamma>1$.
\end{lemma}
\bp
	Differentiating  and collecting terms with coefficients $\nu$ and $\kappa$, respectively,
	we obtain that $v'(x)<0$ (for $x>1$) if and only if 
	\[\sqrt{\frac{x+a}{1+a}}<\frac{1+2a x+x^2}{1+(1+2a)x}\,.\]
	This relation holds since the right- and left-hand sides are 
	separated by a linear function:
	\[\sqrt{\frac{x+a}{1+a}}<\frac{x-1}{2(1+a)}+1
	<\frac{1+2a x+x^2}{1+(1+2a)x}\qquad\text{for $x>1$}\,.\]
\ep
Thus, a vacuum appears in a IIIb interaction if and only if the incoming strengths $y<1<x$ 
satisfy
\beq\label{IIIb_vac}
	0<y\leq V(x):=\left[\textstyle\frac{1}{2}\left(1-\frac{v(x)}{\nu}\right)\right]^\frac{1}{\zeta}
	\qquad\text{where $\zeta= \textstyle\frac{\gamma-1}{2\gamma}$,}
\eeq
and $v(x)$ is given by \eq{no_vac1}.
A calculation shows that the strictly increasing function $V(x)$ satisfies $V(1)=0$ and 
\beq\label{yhat}
	\lim_{x\uparrow\infty} V(x)
	=\left[\textstyle\frac{1}{2}\left(1-\sqrt{\zeta}\right)\right]^\frac{1}{\zeta}=:\hat y(\gamma)<1\,.
\eeq
See Figures 5, 6, 8, 9-14 for schematic plots of the vacuum transition-curve $y=V(x)$. 
As indicated in these figures the horizontal asymptote $\hat y(\gamma)$ of $y=V(x)$
coincides with the horizontal asymptote of the transition curve $\{B=1\}$ in 
$\ba S\ba R$ (IIIb) interactions.

\medskip

This completes the proof of Theorem \ref{sum_grIII}.

\medskip

\section{Definitions and properties of auxiliary functions}\label{aux}

\subsection{Auxiliary functions for wave curves}
The functions $\ba\phi$, $\fa\phi$, $\ba\psi$, $\fa\psi$ were defined in 
Section \ref{bckgr}. A calculation shows that they are all $C^2$ functions 
with Lipschitz continuous 2nd derivatives. Furthermore:
\begin{itemize}
	\item $\ba \phi(q)$ is strictly decreasing, tends to $+\infty$ as $q\downarrow 0$,
	tends to $a$ as $q\to\infty$, and $\ba \phi(1)=1$.
	\item $\fa \phi(q)$ is strictly decreasing, tends to $\frac{1}{a}$ as $q\downarrow 0$,
	tends to $0$ as $q\to\infty$, and $\fa \phi(1)=1$.
	\item $\ba \psi(q)$ is increasing, tends to $-\nu$ (with infinite slope) as 
	$q\downarrow 0$, tends to $+\infty$ as $q\to\infty$, and $\ba \psi(1)=0$.
	\item $\fa \psi(q)$ is increasing, tends to $-\sqrt{\frac{1-a}{a}}$ (with finite 
	slope) as $q\downarrow 0$, tends to $+\infty$ as $q\to\infty$, and $\fa \psi(1)=0$.
\end{itemize}
For reference we record the relation
\beq\label{rel3}
	\sqrt{q\, {\fa{\phi}}(q)}
	{\ba{\psi}}\big({\textstyle\frac{1}{q}}\big)
	=-{{\fa{\psi}}(q)}\,.
\eeq

\medskip

\subsection{The functions $M$, $N$, $m$, $n$, $\ell$, $A$, $D$ and $E$, and their properties}
\label{Metc} In this section we define a number of auxiliary functions and list some 
useful properties. To verify these requires mostly routine calculations which are not 
included. We recall that the parameter $a\in(0,1)$ is defined in \ref{a}.
Define the functions $M$ and $N$ by
\beq\label{M}
	M(q):= -\frac{\ba\psi(q)}{\fa\psi\big(\frac{1}{q}\big)}\equiv \sqrt{q\ba\phi(q)}=
	\left\{\begin{array}{ll}
	q^\zeta \qquad & 0<q<1\\\\
	\sqrt{\frac{q+a q^2}{q+a}}\qquad & q>1
	\end{array}\quad\right\}
\eeq
and
\beq\label{N}
	N(q):= -\frac{\fa\psi(q)}{\ba\psi\big(\frac{1}{q}\big)}\equiv \sqrt{q\fa\phi(q)}=
	\left\{\begin{array}{ll}
	\sqrt{\frac{q+a q^2}{q+a}} \qquad & 0<q<1\\\\
	q^\zeta\qquad & q>1
	\end{array}\quad\right\}\,.
\eeq
Both $M(q)$ and $N(q)$ take the value $1$ at $q=1$, are 
increasing and convex down, tend to zero as $q\downarrow 0$ (with infinite slope), 
and tend to $+\infty$ as $q\to \infty$. 
\begin{lemma}\label{extra1}
	For $x>1$ the function $M(x)$ defined in \eq{M} satisfies the following 
	inequalities for all values of $\gamma>1$:
	\bea
		&& M(x)=\sqrt{\frac{x+ax^2}{x+a}}\, >\, x^\zeta \label{M1}\\
		&& M(x)=\sqrt{\frac{x+ax^2}{x+a}}\, >\, 1+\frac{\kappa(x-1)}{\nu\sqrt{x+a}}\,. \label{M2}
	\eea
\end{lemma}
\begin{proof} 
	By squaring both sides in \eq{M1} and rearranging we obtain the equivalent 
	condition that the function
	\[f(x):=x+ax^2-x^{2\zeta+1}-ax^{2\zeta}\]
	satisfies $f(x)>0$ for $x>1$. A calculation shows that $f'''(x)>0$ 
	for $x>1$, while $f(1)=f'(1)=f''(1)=0$. This proves \eq{M1}.
	By squaring, rearranging, and canceling a factor of $(x-1)$ 
	we obtain that \eq{M2} holds if and only if 
	\[(1-\zeta)x+(1+\zeta)>2\sqrt{\frac{x+a}{1+a}}\,.\]
	Squaring again and simplifying gives that this holds if and only if $(x-1)^2>0$.
\end{proof}
\noindent Define the function $m$ by
\beq\label{m}
	m(q):=\frac{qM'(q)}{M(q)}
	\left\{\begin{array}{ll}
	\equiv \zeta \qquad & 0<q<1\,,\\\\
	= \frac{a  (1+2a q+q^2)}{2(q+a)(1+a q)}\qquad & q>1\,.
	\end{array}\right.
\eeq
A calculation shows that $m$ is non-decreasing and with range $[\zeta,\frac{1}{2})$.
Hence the function
\beq\label{n}
	n(q):=\frac{qN'(q)}{N(q)}=m\big({\textstyle\frac{1}{q}}\big) \qquad (q>0)
\eeq
is non-increasing and with range $[\zeta,\frac{1}{2})$. Define the function 
\beq\label{ell}
	\ell(q):=\frac{q\ba\psi {'}(q)}{\ba\psi(q)}
	=\left\{\begin{array}{ll}
	\frac{\zeta q^\zeta}{q^\zeta-1} \qquad & 0<q<1\,,\\\\
	\frac{q(q+2a +1)}{2(q-1)(q+a)}\qquad & q>1\,.
	\end{array}\right.
\eeq
A calculation shows that $\ell$ is decreasing, has a vertical asymptote at $q=1$, and satisfies
$\ell(q)>\frac{1}{2}$ for all $q>1$.
Define the function $A$ by
\beq\label{Axi}
	A(q,\xi):= \frac{1+aq +a^2\xi}{a^2+a q +\xi}\qquad (q,\, \xi>0)\,.
\eeq
A calculation shows that
\beq\label{Aa_prop}
	\del_q A(q,\xi)\gtrless 0\quad \forall q>0 \qquad\Leftrightarrow\qquad \xi\gtrless 1\,.
\eeq
Define the function $D$ by 
\beq\label{D}
	D(q):=q^\frac{1}{\gamma}\left(\frac{1+aq}{q+a}\right)\qquad (q>0)\,.
\eeq
A calculation shows that $D$ is strictly increasing, $D(1)=1$, and that $D(q)\uparrow\infty$ as 
$q \uparrow\infty$.
\begin{lemma}\label{E_lem}
	Define the function $E$ by
	\beq\label{E}
		E(x):=\frac{1+ax}{x+a} \qquad x>0\,.
	\eeq
	Then
	\begin{itemize}
		\item[(a)] for $0<x<1<y$:\qquad $E(xy)>E(x)E(y) \quad\Leftrightarrow\quad xy<1$,
		\item[(b)] for $0<x,\, y<1$:\qquad $E(xy)<E(x)E(y)$.
	\end{itemize}
\end{lemma}
\bp
	A calculation shows that, since $0<a<1$, $E(xy)>E(x)E(y)$ if and only if
	$(1-xy)(1-x)(1-y)<0$. Parts (a) and (b) follow directly from this.
\ep


%

\medskip

\subsection{The function $\alpha$ and its properties}\label{alfa}
The function $\alpha$ plays a key role in several parts of the arguments. We define  
\beq\label{alp}
	\alpha(q)\equiv \alpha(q,a):=\fa\psi(q)-\ba\psi(q)\qquad q>0.
\eeq
We need to locate the roots of $\alpha$, and these depend sensitively on the value of 
$a=\frac{\gamma-1}{\gamma+1}$. As a first step we introduce 
\beq
	\tilde{\alpha}(q)\equiv \tilde{\alpha}(q,a):=
	\frac{\kappa(q-1)}{\sqrt{q+a}}-\nu\big(q^\zeta-1\big)\,,\qquad q>0\,,
\eeq 
such that 
\beq\label{rel_alphas}
	\alpha(q)=\left\{\begin{array}{rl} 
	\tilde{\alpha}(q)\quad & 0<q<1\\
	-\tilde{\alpha}(q)\quad & q\geq1\,,
	\end{array}\right.
\eeq 
where the constants $\kappa,\nu,\zeta$ are defined in terms of $a$ in Section \ref{bckgr}. 
To analyze $\tilde\alpha$ we introduce the new variable 
\beq\label{z_def}
	z(q)\equiv z(q,a):=\frac{q-1}{\sqrt{q+a}}\,,\quad q>0\,,
\eeq
which satisfies
\beq\label{z_primes} 
	z'(q)=\frac{q+2a+1}{2(q+a)^{\frac{3}{2}}},\qquad
	z''(q)=-\frac{q+4a+3}{4(q+a)^{\frac{5}{2}}}\,.
\eeq
For fixed $a\in(0,1)$ the function $z(q)$ is strictly increasing with a strictly increasing 
inverse  $q(z)$ defined for $z\in(-a^{-1/2},\infty)$, and with $q(0)=1$.
The first two derivatives of the latter are given by
\beq\label{derivs}
	q^{\prime}(z)=\frac{1}{z^\prime(q)}=\frac{2(q+a)^{\frac{3}{2}}}{q+2a+1}\,,\qquad 
	q^{\prime\prime}(z)=-\frac{z^{\prime\prime}(q)}{(z^\prime(q))^3} =\frac{2(q+4a+3)(q+a)^2}{(q+2a+1)^3}\,,
\eeq
where $q=q(z)$. Now, in terms of the variable $z(q)$, we have 
\[\tilde \alpha(q)=\beta(z(q))\,,\]
where 
\beq\label{beeta}
	\beta(z):=\kappa z-\nu\big[q(z)^\zeta-1\big].
\eeq
Thus, in order to locate the zeros of $\tilde\alpha$ we may as well determine the 
zeros of $\beta(z)$ for $z\in (-a^{-1/2},\infty)$, and then translate back to $q$-locations. 
(This turns out to be easier than to determine directly the roots of $\tilde\alpha$.)
We do so by considering the derivatives of $\beta(z)$. 
Differentiating $\beta(z)$ and using \eq{derivs} we have
\beq\label{beeta_prime}
	\beta'(z)=2\nu\zeta\left[\frac{\sqrt{1+a}}{2} - 
	\frac{q^{\zeta-1}(q+a)^\frac{3}{2}}{q+2a+1} \right]
\eeq
and
\[\beta''(z)=\frac{2q^{\zeta-2}(q+a)^2}{\gamma^\frac{3}{2} (q+2a+1)^3}
\cdot (q-1)(q-\bar q),\]
where 
\[ q=q(z)\qquad\text{and}\qquad\bar q:=\frac{2a(2a+1)}{1-a}\,.\]
This shows that the sign of $\beta^{\prime\prime}(z)$ is the same as that of the function
\beq\label{Qz}
	Q(z):=(q(z)-1)(q(z)-\bar q)\,,
\eeq 
such that $Q(z)$ has the zeros $z_1:=0$ and $z_2:=z(\bar q)$. We note that
\beq 
	z_2 \gtreqqless z_1\quad \Leftrightarrow\quad \bar q \gtreqqless 1\quad \Leftrightarrow \quad 
	a\gtreqqless\frac{1}{4} \quad \Leftrightarrow\quad \gamma\gtreqqless\frac{5}{3}.
\eeq 
To locate the zeros of $\beta(z)$ we consider these regimes separately. The
approach is the same in each case: the signs of $\beta''(z)$, together with 
end-point values of $\beta''$, $\beta'$ and $\beta$, determine the number and locations 
of the roots of $\beta$. We therefore only detail the argument in the representative case 
when $1<\gamma<\frac{5}{3}$. 

In this case the function $Q(z)$ has two distinct roots $z_1=0$ and $z_2=z(\bar q)\in(-a^{-1/2},0)$.
>From \eq{Qz} we have that $\beta^{\prime\prime}(z)>0$ when
$z\in(-a^{-1/2},z_2)\cup(0,\infty)$, and $\beta^{\prime\prime}(z)<0$ when
$z\in(z_2,0)$. Hence $\beta^{\prime}(z)$ is increasing on $(-a^{-1/2},z_2)\cup(0,\infty)$
and decreasing in $(z_2,0)$. Next, from \eq{z_def}, \eq{beeta}, and $\zeta<\frac{1}{2}$ 
we obtain
\beq\label{beta_prime_0infty}
	\lim_{z\downarrow{-a^{-1/2}}}\beta^{\prime}(z)=-\infty,
	\quad\beta^{\prime}(0)=0,\quad
	\lim_{z\rightarrow{\infty}}\beta^{\prime}(z)=\nu\zeta\sqrt{1+a}>0.
\eeq
It follows that $\beta^{\prime}(z)>0$ has two roots: one at $z_3\in(-a^{-\frac{1}{2}},z_2)$ and 
one at $z=0$. Also, $\beta'(z)<0$ on $(-a^{-\frac{1}{2}},z_3)$, and $\beta'(z)>0$ on $(z_3,0)\cup(0,\infty)$.

Next, by \eq{z_def} and \eq{beeta} we have $\beta(-a^{-1/2})=\nu-\kappa a^{-1/2}>0$
and $\beta(0)=0$. It follows that $\beta$ itself has exactly two roots: a leftmost root 
$z_4\in (-a^{-1/2},z_3)$ and the other one at $z=0$. Translating back to $q$-variables we 
conclude that $\tilde\alpha(q)$ has exactly two zeros: $q(z_4)=:y_0\in (0,1)$ and $q=1$. 
Finally, it follows from this and \eq{rel_alphas} that the same holds for $\alpha(q)$ itself. 
See Figure 15.

Similar arguments show that
\begin{itemize}
	\item when $\gamma=\frac{5}{3}$ the function $\alpha(q)$ has a single root at $q=1$ 
	(Figure 16), and
	\item when $\gamma>\frac{5}{3}$ the function $\alpha(q)$ has exactly two roots: one at $q=1$
	and one at $y_0\in (1,\infty)$ (Figure 17).
\end{itemize}

\begin{figure}\label{alpha1}
	\centering
	\includegraphics[width=7cm,height=5cm]{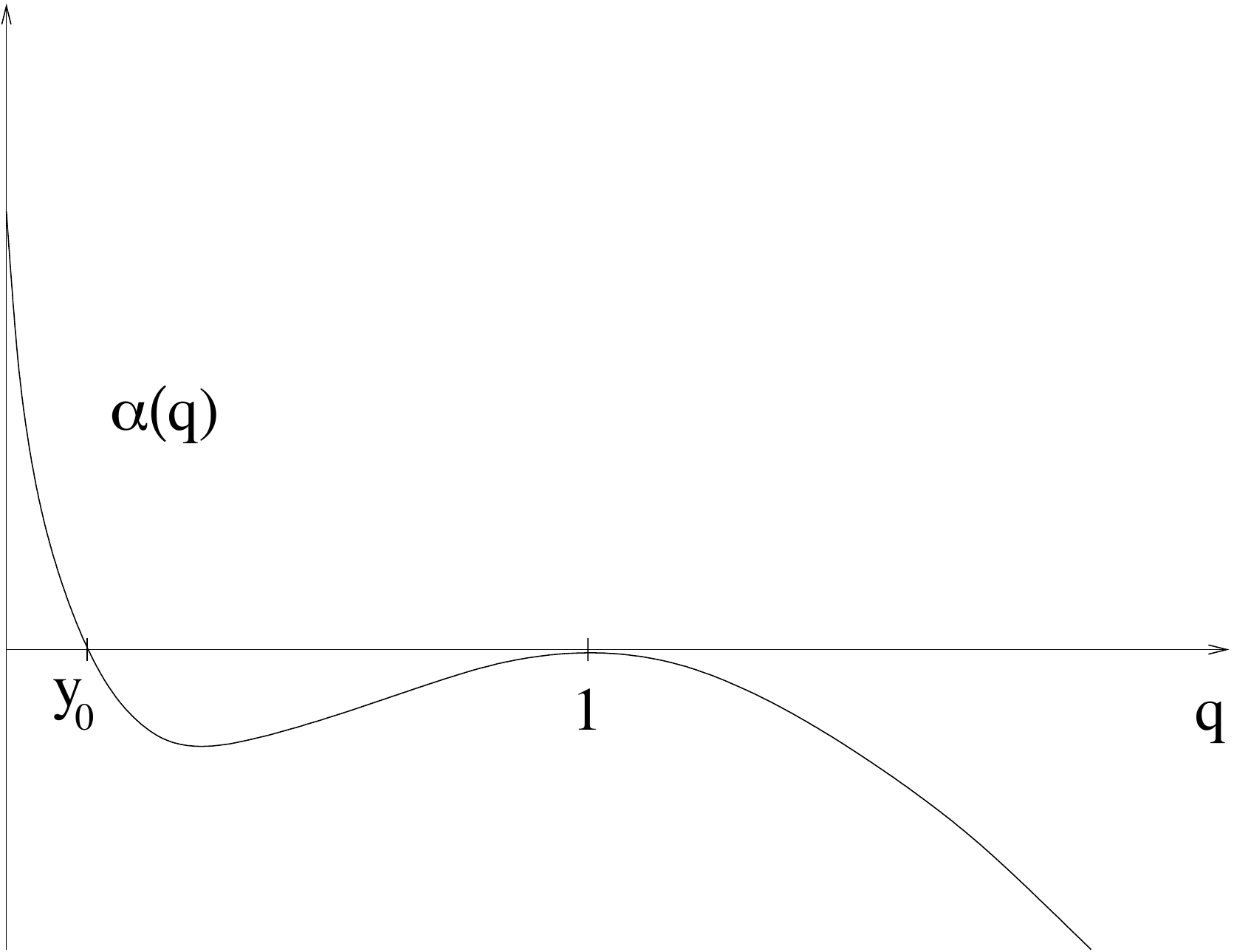}
	\caption{$q\mapsto \alpha(q,a)$ when $0<a<\frac{1}{4}$ ($1<\gamma<\frac{5}{3}$, schematic)}
\end{figure}

\begin{figure}\label{alpha2}
	\centering
	\includegraphics[width=7cm,height=5cm]{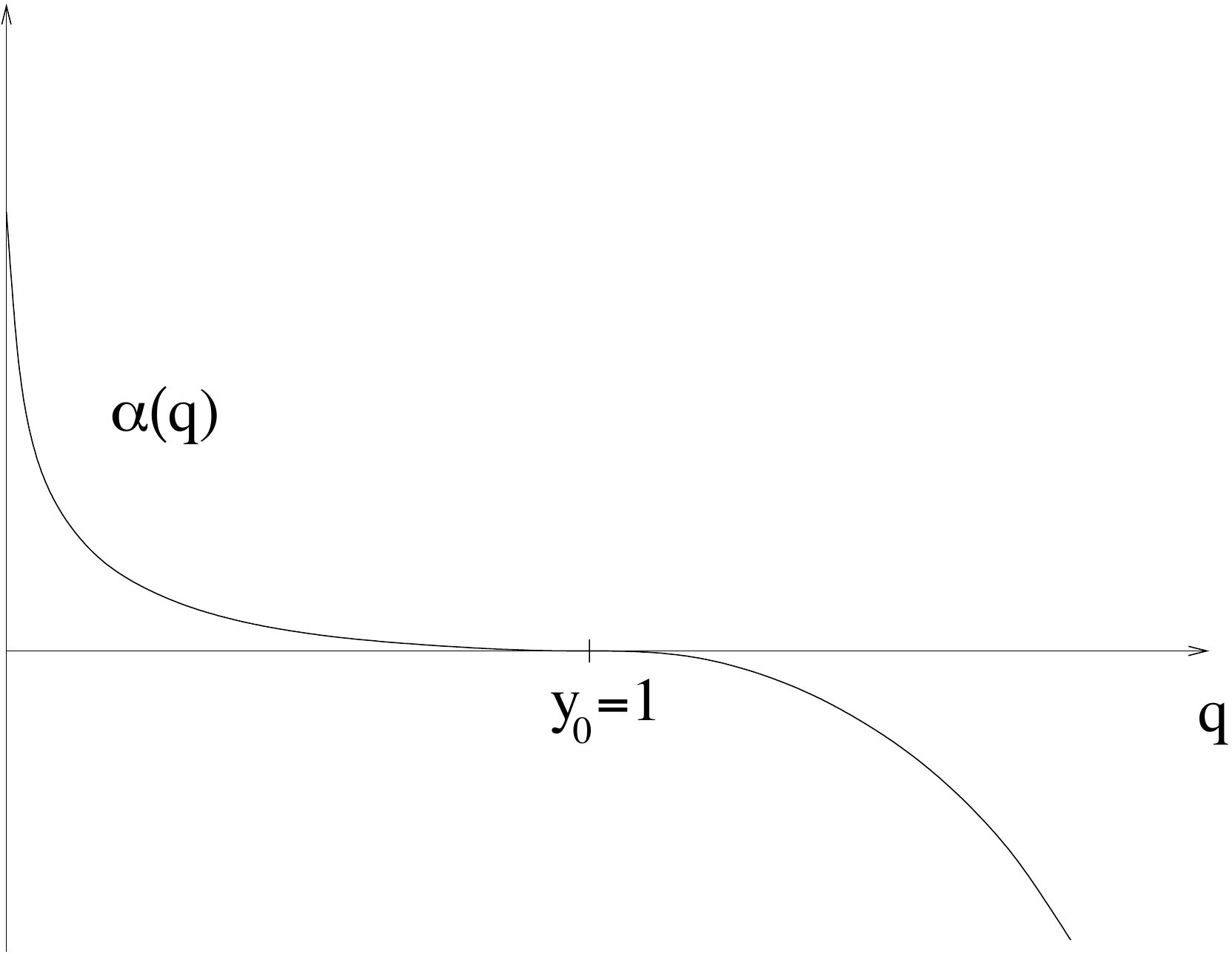}
	\caption{$q\mapsto \alpha(q,\frac{1}{4})$ ($\gamma=\frac{5}{3}$, schematic)}
\end{figure}
%
\begin{figure}\label{alpha3}
	\centering
	\includegraphics[width=7cm,height=5cm]{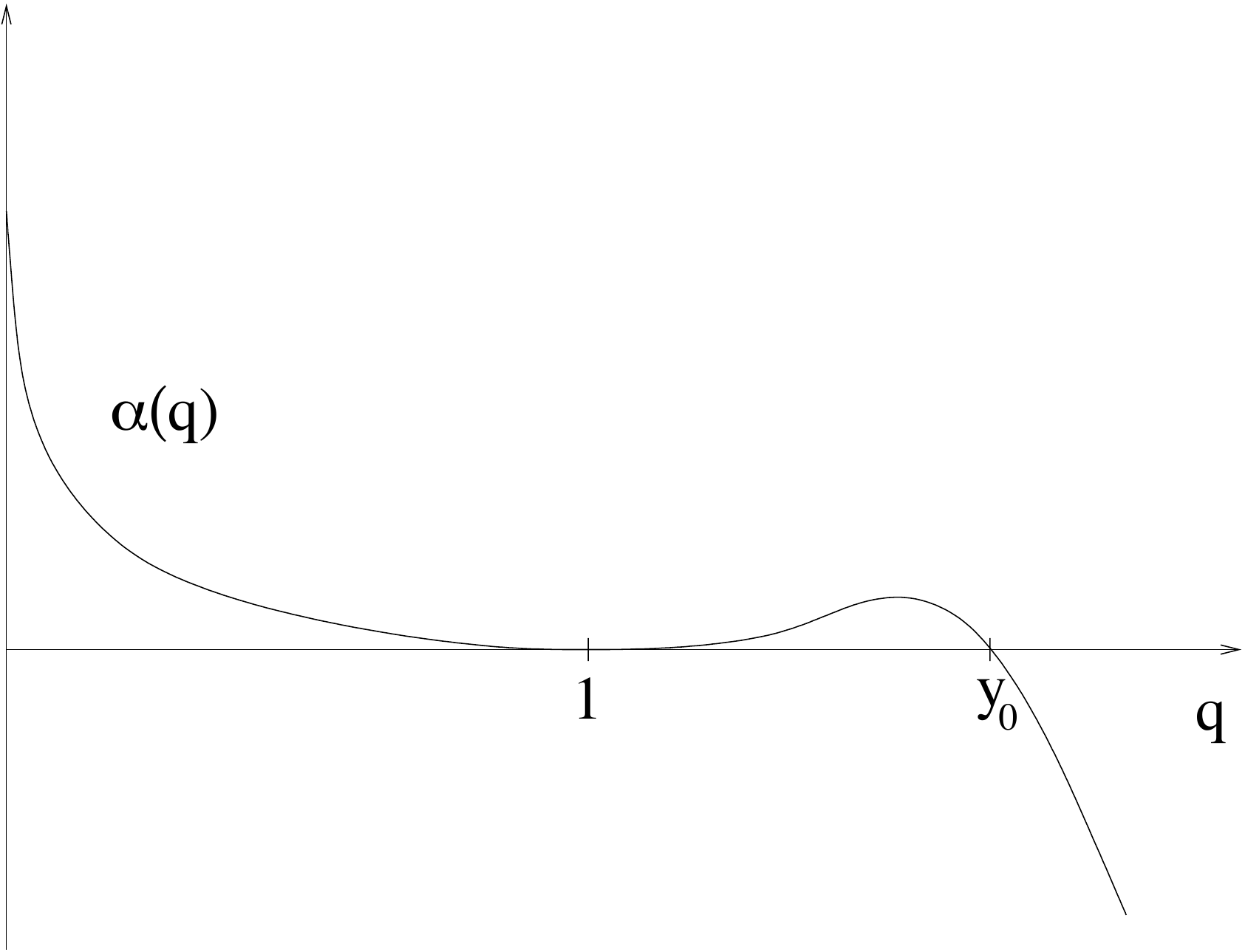}
	\caption{$q\mapsto \alpha(q,a)$ when $a>\frac{1}{4}$ ($\gamma>\frac{5}{3}$, schematic)}
\end{figure}

\medskip

\subsection{The functions $\Gamma$ and $\Lambda$, and their properties}\label{Gamma}
We define the function 
\beq\label{Gamma_def}
	\Gamma(s)\equiv\Gamma(s,a):=\frac{s^\gamma(1-a s)}{s-a}, \qquad
	a<s<\frac{1}{a}\,,
\eeq 
such that the functions $\ba\eta$ and $\fa\eta$ in \eq{bkwd_wave_tau_u_s}-\eq{frwd_wave_tau_u_s}
may be expressed as follows:
\beq\label{eta_fwd_bckwd}
	{\ba{\eta}(l)}= \left\{\begin{array}{ll}
	0 \quad & l>1\\\\
	\log\Gamma(l)\quad & a <l<1
	\end{array}\right\}\qquad
	{\fa{\eta}}(\iota)=
	\left\{\begin{array}{ll}
	0\quad & 0< \iota<1\\\\
	\log\Gamma(\iota)\quad & 1<\iota<\frac{1}{a}
	\end{array}\right\}.
\eeq
The function $\Gamma$ satisfies the relation
\beq\label{Gamma_2}
	\Gamma(s)\Gamma({\textstyle\frac{1}{s}})=1,\qquad a<s<\frac{1}{a}\,.
\eeq
A calculation shows that 
\beq\label{Gamma_diff}
	\Gamma'(s)=-a \gamma s^{\gamma-1}\frac{(s-1)^2}{(s-a)^2}\,,
\eeq
such that $\Gamma$ is a positive and strictly decreasing function on $(a,\frac{1}{a})$,
and satisfies
\beq\label{Gamma_1}
	\Gamma(s)\gtrless 1\qquad\text{for}\qquad s\lessgtr1.
\eeq
\begin{lemma}\label{Gamma_prop}
	The function $\Gamma$ has the following property: for $s$, $t$, $st\in(a,\frac{1}{a})$
	\beq\label{Gamma_convexity}
		\Gamma(s)\Gamma(t)<\Gamma(st)\qquad \text{if and only if}\qquad
		(1-st)(1-s)(1-t)>0.
	\eeq 
\end{lemma}
\bp
	By using the expression for $\Gamma$ we have that $\Gamma(s)\Gamma(t)<\Gamma(st)$
	if and only if 
	\[\left[\frac{1-as}{s-a}\right]\left[\frac{1-at}{t-a}\right]<\frac{1-ast}{st-a}\,.\]
	Using the assumption that $s$, $t$, and $st$ are all larger than $a$,
	a calculation shows that this holds if and only if $(1-st)(1-s)(1-t)>0$.
\ep
We introduce the auxiliary function 
\beq\label{Lambda_def}
	\Lambda(z):=\ba\xi(\Omega^{-1}(z)),\qquad z\in(1,\infty),
\eeq 
where $\ba\xi$ is defined in Section \ref{curves} and 
\[\Omega(s):=\sqrt{\Gamma(s)},\qquad s\in(a,1).\] 
\begin{lemma}\label{Lambda_prop}
The function $\Lambda:(1,\infty)\to (0,\infty)$ is increasing and has the property that  
\beq\label{2_2}
	\Lambda(z_1)+z_1\Lambda(z_2) \geq \Lambda(z_1z_2)
	\qquad\text{for all} \quad z_1,\, z_2\geq 1\,.
\eeq 
\end{lemma}
\bp
First, since $\Lambda(1)=0$, \eq{2_2} follows from
\beq\label{2_3}
	\Lambda^{\prime}(z)\geq \Lambda^\prime(z_1 z), \quad z\in(1,\infty)
\eeq
by integration from $z=1$ to $z=z_2$. In turn, \eq{2_3} follows if we show
that $\Lambda$ is a concave:
\beq\label{concave}
	\Lambda^{\prime\prime}(z)<0, \qquad z\in(1,\infty).
\eeq
To establish \eq{concave} we first calculate $\Lambda'(z)$:
\[\Lambda^{\prime}(z)= \frac{\ba\xi {'}({s})}{\Omega'({s})}
=\frac{\sqrt{1+a}}{a\gamma}\cdot\frac{\sqrt{1-as}(s+1-2a)}{s^{\frac{\gamma}{2}-1}(s-1)^2}
\qquad \text{where}\quad s=\Omega^{-1}(z)\in(a,1)\,.\]
Since $\Omega'(s)<0$ we see that \eq{concave} holds if and only if 
\[\frac{d}{ds}\left[\frac{\sqrt{1-as}(s+1-2a)}{s^{\frac{\gamma}{2}-1}(s-1)^2}\right]>0
\qquad\text{for}\quad s\in(a,1).\]
A calculation shows that this is the case if and only if the polynomial
\[\mathcal{P}(a,s):= 2 a^2 s^3-(10 a^2-5a+1)s^2+(8a^3-10a^2+14a-6)s -(6a^2-5a+1)\]
satisfies
\beq\label{P}
	\mathcal{P}(a,s)<0 \qquad\text{whenever}\quad 0<a<1\quad\text{and}\quad a<s<1\,.
\eeq
To verify \eq{P} we fix $s\in (0,1)$ and consider the map $a\mapsto\mathcal P(a,s)$ 
for $0<a<1$. Since 
\[\del_a^3\mathcal P(a,s)=48s>0, \qquad \del_a^2\mathcal P(1,s)=4(s-1)^2(s-3)<0,\]
while
\[\del_a\mathcal P(s,s)=(s-1)^2(4s^2+12s+5)>0\qquad\text{and}\qquad 
\mathcal P(s,s)=(s-1)^3(2s^2+4s+1)<0\,,\]
it follows that \eq{P} indeed holds. 
\ep

\medskip

\subsection{The function $H(x,y)$ in region IIIa}\label{H_in_IIIa}
Consider the function $H$ defined in \eq{HH}. For $(x,y)\in$ IIIa $=\{x>1,\, y>1\}$ we
use the explicit expressions for $\ba\psi$ and $M$ to get
\[H(x,y)=\kappa\left[\frac{(xy-1)}{\sqrt{xy+a}} - \frac{(x-1)}{\sqrt{x+a}} 
- \frac{(y-1)}{\sqrt{y+a}}\cdot\sqrt{\frac{x+ax^2}{x+a}}\right]\qquad (x,\, y>1)\,. \]
Thus
\beann
	&& \nqquad H(x,y) \gtrless 0  \quad\Leftrightarrow\quad 
	\frac{(xy-1)}{\sqrt{xy+a}} \gtrless  \frac{(x-1)}{\sqrt{x+a}} 
	+ \frac{(y-1)}{\sqrt{y+a}}\cdot\sqrt{\frac{x+ax^2}{x+a}} \\
	&\Leftrightarrow& \quad (xy-1)\sqrt{x+a}\sqrt{y+a} 
	\gtrless  (x-1)\sqrt{y+a}\sqrt{xy+a} +  (y-1)\sqrt{xy+a}\sqrt{x+ax^2}\,.
\eeann
For $x,\, y>1$ this holds if and only if the same inequality with squared RHS and LHS 
holds. After squaring each side, collecting positive and negative terms, canceling the common factor $\sqrt{x}(x-1)(y-1)$, 
and rearranging, we obtain:
\[ H(x,y) \gtrless 0 \quad\Leftrightarrow\quad 2(xy+a)\sqrt{y+a}\sqrt{1+ax}
 \lessgtr  \sqrt{x}\left[ (1-a)xy^2+2a(xy+a)+(3a+1)y \right].\]
Again each side is positive; squaring, collecting terms, canceling the common factor $(1-xy)^2$, and rearranging,
finally yield
\beq\label{H_pos}
	H(x,y) \gtrless 0 \quad \Leftrightarrow \quad 
	xy\left[(1-a)^2y-4a^2\right]\gtrless 4a^2(y+a)\qquad\text{when $x,\, y>1$.}
\eeq
In particular, if $0<a\leq \frac{1}{4}$ then $(1-a)^2y-4a^2>4a^2(y+a)$ whenever $y>1$. Thus
\beq\label{H_crit}
	H(x,y)>0\qquad\text{whenever $x,\, y>1$ and $0<a\leq \textstyle\frac{1}{4}$.}
\eeq

\medskip

\subsection{The function $H(x,y)$ in region IIIb2}\label{del_H}
Consider the function $H$ defined in \eq{HH}. 
We want to determine the its zeros in the subregion IIIb2 $:=\{0<y<1<xy\}$. 
>From \eq{HH} and the explicit expressions for $\ba\psi$ and $M$, we have 
\[H(x,y)=\frac{\kappa(xy-1)}{\sqrt{xy+a}}-\frac{\kappa(x-1)}{\sqrt{x+a}}
-\nu(y^\zeta-1)\sqrt{\frac{x+ax^2}{x+a}}\qquad\text{for $\frac{1}{x}<y<1$.}\]
Fixing $x>1$ we first note that
\beq\label{H(x,1/x)}
	H\big(x,\textstyle\frac{1}{x}\big)=M(x)\alpha\big(\textstyle\frac{1}{x},a\big)
	\quad\text{while}\quad H(x,1)\equiv 0,\quad\forall x>0\,.
\eeq
We study $y\mapsto H(x,y)$ as $y$ increases from $\frac{1}{x}$ to $1$
by analyzing $y\mapsto \del_y H(x,y)$. As we shall see, its 
behavior depends on whether $0<a\leq \frac{1}{4}$, $\frac{1}{4}<a<\frac{1}{3}$, 
or $a\geq \frac{1}{3}$.
Introducing the functions $\theta(z,a)$ and $Q(z,x,a)$ by
\beq\label{theta_def}
	\theta (z,a):=
	\frac{(z+2a+1)z^{1-\zeta}}{(z+a)^\frac{3}{2}}\qquad\text{for $z>1$,}
\eeq
\beq\label{q_def}
	Q(z,x,a):= \frac{\sqrt{1+a}}{2}\cdot\frac{x^\zeta}{M(x)}\cdot\theta (z,a)-1\qquad\text{for $x,\, z>1$,}
\eeq
we have
\beq\label{del_theta1}
	\del_y H(x,y)=\nu\zeta M(x)y^{\zeta-1}Q(xy,x,a) \qquad\text{for $(x,y)\in$ IIIb2.}
\eeq
%
%
We next observe the following points:
\begin{itemize}
	\item [(A1)] By \eq{del_theta1} the sign of $\del_y H(x,y)$ is the same as that of $Q(z,x,a)|_{z=xy}$. 
	\item [(A2)] According to Lemma \ref{extra1},
		\beq\label{del_theta2}
			Q(1,x,a)=\frac{x^\zeta}{M(x)}-1<0\,,
		\eeq
		such that, in IIIb2, $y\mapsto H(x,y)$ ``starts out" decreasing at $y=\frac{1}{x}$.
	\item [(A3)] A calculation shows that $Q(x,x,a)> 0$ for $x>1$ if and only if 
		\beq\label{del_theta3}
			(1-3a)(x-\bar x)>0\qquad \text{where}\quad \bar x=\frac{4a^2}{1-3a}\,.
		\eeq
	\item [(A4)] We have $\bar x<0$ for $a>\frac{1}{3}$, while
		\[0<\bar x\leq 1 \quad\Leftrightarrow\quad 0<a\leq \textstyle\frac{1}{4}\,,\qquad
		\bar x> 1 \quad\Leftrightarrow\quad a\in\big(\frac{1}{4},\frac{1}{3}\big)\,.\]
		For $a=\frac{1}{3}$, $\bar x$ is undefined and $Q(x,x,\textstyle\frac{1}{3})<0$.
	\item [(A5)] By \eq{q_def}, the sign of $\del_z Q(z,x,a)$ coincides with that of $\del_z\theta (z,a)$. 
		A calculation shows that the latter is given by
		\beq\label{del_theta4}
			\del_z\theta (z,a)=\frac{(1-a)(z-1)(z-\hat z)}{2(1+a)z^\zeta(z+a)^\frac{5}{2}}\,,\qquad
			\text{where}\quad\hat z:=\frac{2a(2a+1)}{1-a}\,.
		\eeq
		We have
		\[\hat z\gtrless 1\qquad\Leftrightarrow\qquad a\gtrless \frac{1}{4}\,.\]
\end{itemize}
We can now analyze the zeros of $y\mapsto H(x,y)$ for $\frac{1}{x}<y<1$:
\begin{itemize}
	\item $0<a\leq\frac{1}{4}$. By (A2)-(A5) we have 
	\[\qquad Q(1,x,a)<0\quad\text{and}\quad Q(x,x,a)>0,\quad\text{while}\quad\del_z Q(z,x,a)>0\quad
	\text{for $1<z<x$.}\]
	It follows from this and \eq{del_theta1} that $y\mapsto \del_y H(x,y)$ changes sign once from negative
	to positive as $y$ increases from $\frac{1}{x}$ to $1$. From \eq{H(x,1/x)} and  the properties of the map $\alpha$,
	we have $H(x,\frac{1}{x})\gtrless 0$ according to $x\gtrless x_0$. (Here $x_0=\frac{1}{y_0}$, $y_0$ 
	being the unique root of $\alpha(\cdot,a)$ different from $1$ when 
	$a<\frac{1}{4}$, and $y_0=1$ when $a=\frac{1}{4}$.) 
	We conclude from \eq{H(x,1/x)} that when $1<\gamma<\frac{5}{3}$ then the map $y\mapsto H(x,y)$ has 
	\bea
		&\bullet&\text{no root in $(\textstyle\frac{1}{x},1)$ when $x\leq x_0$,}\label{1st}\\
		&\bullet&\text{exactly one root $y=h_2(x)\in(\textstyle\frac{1}{x},1)$ when $x>x_0$.}\label{2nd}
	\eea
	As $H$ is a $C^2$-map it follows that $h_2(x)$ is a $C^2$-function. 
	\item $a>\frac{1}{4}$. By \eq{H(x,1/x)} and the properties of $\alpha(y,a)$ we have 
	$H(x,\frac{1}{x})>0=H(x,1)$ for $x>1$. As above we have $Q(1,x,a)<0$, but now $\hat z>1$ (by (A5)) and 
	\[\del_z Q(z,x,a)<0 \quad\text{for $z\in(1,\hat z)$,}\quad \del_z Q(z,x,a)>0  \quad\text{for $z\in(\hat z,\infty),$}\]
	and $Q(z,x,a)\to\infty$ as $z\uparrow \infty$ (by \eq{q_def}, \eq{theta_def}, and the fact that $\zeta<\frac{1}{2}$).
	Thus, the map $z\mapsto Q(z,x,a)$ has a unique zero in $(\hat z,\infty)$ whenever $a>\frac{1}{4}$.
	By \eq{del_theta1} this root corresponds to a zero of $y\mapsto \del_yH(x,y)$ in  the interval $(\frac{1}{x},1)$ if and only 
	if $Q(x,x,a)>0$. By (A4) and \eq{del_theta3} this is the case if and only if $a<\frac{1}{3}$ and $x>\bar x$.
	We conclude that, for each fixed $x>1$ the map $y\mapsto H(x,y)$ has: 
	\bea
		&\bullet& \text{a unique root $y=h_2(x)\in(\textstyle\frac{1}{x},1)$ if and only if $x>\bar x$, 
		when $\textstyle\frac{1}{4}<a<\frac{1}{3}$},\label{3rd}\\
		&\bullet& \text{no root in $\textstyle(\frac{1}{x},1)$ when $a\geq \textstyle\frac{1}{3}$.}\label{4th}
	\eea
	Again, in the former case $h_2$ is $C^2$-smooth.
\end{itemize}

\medskip

\subsection{The function $H(x,y)$ in region IIIc}\label{eta}
Consider the function $H$ defined in \eq{HH}. In IIIc$:=\{0<x<1<y\}$ we analyze how $H(x,y)$
varies along hyperbolas $xy=conts.$ Again, by using the explicit expressions for $\ba\psi$ and $M$
we have
\beq\label{eta_prop}
	(-x,y)\cdot \nabla_{(x,y)} H(x,y) = {\textstyle\frac{\kappa}{2}} x^\zeta \eta(y,a)
	\qquad\text{for $0<x<1$ and $y>1$,}
\eeq
where the function $\eta$ is given by
\beq\label{eta_def}
	\eta(y,a):=\frac{2}{\sqrt{1+a}} - \frac{(1-a)y^2+(5a+1)y+2a^2}{(1+a)(y+a)^\frac{3}{2}}
	\qquad\text{for $y\geq 1$.}
\eeq
A calculation shows that $\eta(y,a)$ factors as follows:
\[\eta(y,a)=-\frac{(1-a)\left(\sqrt{y+a}-\sqrt{1+a}\right)^2}{(1+a)(y+a)^\frac{3}{2}}
\left[\sqrt{y+a}-\frac{\sqrt{a(1+a)}}{1-\sqrt{a}}\right]\!\!
\left[\sqrt{y+a}+\frac{\sqrt{a(1+a)}}{1+\sqrt{a}}\right].\]
It follows from this and \eq{eta_prop} that the directional derivative $(-x,y)\cdot \nabla_{(x,y)} H(x,y)$ 
in IIIc has the same sign as 
\[\frac{\sqrt{a(1+a)}}{1-\sqrt{a}}-\sqrt{y+a}\,,\]
which is strictly decreasing in $y$. In particular, since 
\[\left.\left[\frac{\sqrt{a(1+a)}}{1-\sqrt{a}}-\sqrt{y+a}\right]\right|_{y=1}\leq 0 \qquad\text{if and only if}\qquad a\leq \textstyle\frac{1}{4}\,,\]
and since $H(1,y)=H(x,1)\equiv 0$, we obtain the following:
\begin{lemma}\label{IIIc_H_a<1/4}
	For $a\leq \textstyle\frac{1}{4}$ the function $H(x,y)$ is strictly negative in $\mathrm{IIIc}=\{0<x<1<y\}$.
\end{lemma}

\medskip
{\bf{Acknowledgement.}}
Geng Chen would like to thank Robin Young for helpful discussions.



\begin{bibdiv}
\begin{biblist}
\bib{ch}{book}{
   author={Chang, Tung},
   author={Hsiao, Ling},
   title={The Riemann problem and interaction of waves in gas dynamics},
   series={Pitman Monographs and Surveys in Pure and Applied Mathematics},
   volume={41},
   publisher={Longman Scientific \& Technical},
   place={Harlow},
   date={1989},
   pages={x+272},
   isbn={0-582-01378-X},
   review={\MR{994414 (90m:35122)}},
}
\bib{coufr}{book}{
   author={Courant, R.},
   author={Friedrichs, K. O.},
   title={Supersonic flow and shock waves},
   note={Reprinting of the 1948 original;
   Applied Mathematical Sciences, Vol. 21},
   publisher={Springer-Verlag},
   place={New York},
   date={1976},
   pages={xvi+464},
   review={\MR{0421279 (54 \#9284)}},
}
\bib{gl}{article}{
   author={Glimm, James},
   title={Solutions in the large for nonlinear hyperbolic systems of
   equations},
   journal={Comm. Pure Appl. Math.},
   volume={18},
   date={1965},
   pages={697--715},
   issn={0010-3640},
   review={\MR{0194770 (33 \#2976)}},
}
\bib{gr}{article}{
   author={Greenberg, James M.},
   title={On the interaction of shocks and simple waves of the same family},
   journal={Arch. Rational Mech. Anal.},
   volume={37},
   date={1970},
   pages={136--160},
   issn={0003-9527},
   review={\MR{0274957 (43 \#715)}},
}
\bib{had}{book}{
   author={Hadamard, Jacques},
   title={Le\c cons sur la propagation des ondes et les \'equations de l'hydrodynamique},
   publisher={Chelsea},
   place={New York},
   date={1949},
}
\bib{jc}{collection}{
   title={Classic papers in shock compression science},
   series={High-pressure Shock Compression of Condensed Matter},
   editor={Johnson, James N.},
   editor={Ch{\'e}ret, Roger},
   publisher={Springer-Verlag},
   place={New York},
   date={1998},
   pages={xii+524},
   isbn={0-387-98410-0},
   review={\MR{1625293 (99m:76076)}},
}
\bib{jou}{book}{
   author={Jouguet, {\'E}mile},
   title={M\'ecanique des Explosifs},
   language={French},
   publisher={O.\ Doin et Fils},
   place={Paris},
   date={1917},
}
\bib{ll}{book}{
   author={Landau, L. D.},
   author={Lifshitz, E. M.},
   title={Course of theoretical physics. Vol. 6},
   edition={2},
   note={Fluid mechanics;
   Translated from the third Russian edition by J. B. Sykes and W. H. Reid},
   publisher={Pergamon Press},
   place={Oxford},
   date={1987},
   pages={xiv+539},
   isbn={0-08-033933-6},
   isbn={0-08-033932-8},
   review={\MR{961259 (89i:00006)}},
}
\bib{vneu}{book}{
   author={von Neumann, John},
   title={Collected works. Vol. VI: Theory of games, astrophysics,
   hydrodynamics and meteorology},
   series={General editor: A. H. Taub. A Pergamon Press Book},
   publisher={The Macmillan Co.},
   place={New York},
   date={1963},
   pages={x+538 pp. (1 plate)},
   review={\MR{0157876 (28 \#1105)}},
}
\bib{smol}{book}{
   author={Smoller, Joel},
   title={Shock waves and reaction-diffusion equations},
   series={Grundlehren der Mathematischen Wissenschaften [Fundamental
   Principles of Mathematical Sciences]},
   volume={258},
   edition={2},
   publisher={Springer-Verlag},
   place={New York},
   date={1994},
   pages={xxiv+632},
   isbn={0-387-94259-9},
   review={\MR{1301779 (95g:35002)}},
}
\bib{ri}{article}{
   author={Riemann, Bernhard},
   title={Ueber die Fortpflanzung ebener Luftwellen von endlicher Schwingungsweite},
   language={German},
   journal={Abhandlungen der Gesellschaft der Wissenschaften zu G\"ottingen},
   volume={8},
   date={1860},
   pages={43--65},
}
\bib{rj}{book}{
   author={Ro{\v{z}}destvenski{\u\i}, B. L.},
   author={Janenko, N. N.},
   title={Systems of quasilinear equations and their applications to gas
   dynamics},
   series={Translations of Mathematical Monographs},
   volume={55},
   note={Translated from the second Russian edition by J. R. Schulenberger},
   publisher={American Mathematical Society},
   place={Providence, RI},
   date={1983},
   pages={xx+676},
   isbn={0-8218-4509-8},
   review={\MR{694243 (85f:35127)}},
}
\end{biblist}
\end{bibdiv}
\end{document}